\documentclass[11pt,reqno]{amsart}
\usepackage{natbib}
\usepackage{float}
\usepackage{graphicx,amsmath,amssymb,epsfig,array,multirow}

\usepackage{xr}
\theoremstyle{plain}
\newtheorem{theorem}{Theorem}[section]
\newtheorem{corollary}{Corollary}[section]
\newtheorem{lemma}{Lemma}[section]

\theoremstyle{definition}

\newtheorem{remark}{Comment}[section]

\newtheorem*{algorithm}{Algorithm}

\newcommand{\Ep}{{\mathrm{E}}}
\newcommand{\En}{{\mathbb{E}_n}}

\newcommand{\R}{{\mathbb{R}}}
\newcommand{\sign}{{\rm sign}}
\newcommand{\Var}{{\rm Var}}

\renewcommand{\Pr}{{\mathrm{P}}}
\renewcommand{\hat}{\widehat}
\renewcommand{\tilde}{\widetilde}

\begin{document}

\title[Many Moment Inequalities]{Inference on Causal and Structural Parameters using Many Moment Inequalities}

\thanks{We are grateful to Jin Hahn, Adam Rosen, Azeem Shaikh, participants at Cowles Summer Conference 2013, Asian Meeting of Econometric Society 2013, Bernoulli Society Satellite Conference, and seminar participants at UCL and USC. V. Chernozhukov and D. Chetverikov  are supported by a National Science Foundation grant. This paper was previously circulated under the title ``Testing many moment inequalities.''}

\author[Chernozhukov]{Victor Chernozhukov}
\author[Chetverikov]{Denis Chetverikov}
\author[Kato]{Kengo Kato}

\address[V. Chernozhukov]{
Department of Economics and Operations Research Center, MIT, 50 Memorial Drive, Cambridge, MA 02142, USA.
}
\email{vchern@mit.edu}

\address[D. Chetverikov]{
Department of Economics, UCLA, 315 Portola Plaza, Bunche Hall Room 8283, Los Angeles, CA 90024, USA.
}
\email{chetverikov@econ.ucla.edu}

\address[K. Kato]{
Department of Statistical Science, Cornell University, 1194 Comstock Hall, Ithaca, NY 14853, USA.
}
\email{kk976@cornell.edu}

\date{First public version: December 2013 (arXiv:1312.7614v1). This version: \today.}

\begin{abstract}
This paper considers the problem of testing {\em many} moment inequalities where the number of moment inequalities, denoted by $p$,  is possibly much larger than the sample size $n$.
There is a variety of economic applications where solving this problem allows to carry out inference on causal and structural parameters; a notable example is the market structure model of \cite{CilibertoTamer2009} where $p=2^{m+1}$ with $m$ being the number of firms that could possibly enter the market. We consider the test statistic given by the maximum of $p$ Studentized (or $t$-type) inequality-specific statistics, and analyze various ways to compute critical values for the test statistic. Specifically, we consider critical values based upon (i) the union bound combined with a moderate deviation inequality for self-normalized sums, (ii) the multiplier and empirical bootstraps, and (iii) two-step and three-step variants of (i) and (ii) by incorporating the selection of uninformative inequalities that are far from being binding and a novel selection of weakly informative inequalities that are potentially binding but do not provide first order information. We prove validity of these methods, showing that under mild conditions, they lead to tests with the error in size decreasing polynomially in $n$ while allowing for $p$ being much larger than $n$; indeed $p$ can be of order $\exp (n^{c})$ for some $c > 0$. Importantly, all these results hold without any restriction on the correlation structure between $p$ Studentized statistics, and also hold uniformly with respect to suitably large classes of underlying distributions. Moreover, in the online supplement, we show validity of a test based on the block multiplier bootstrap in the case of dependent data under some general mixing conditions.
\end{abstract}
\maketitle

\section{Introduction}
\label{sec: introduction}
In recent years, the moment inequalities framework has developed into a powerful tool for inference on causal and structural parameters in partially identified models. Many papers studied models with a finite and fixed (and so asymptotically small) number of both conditional and unconditional moment inequalities; see the list of references below. In practice, however, the number of moment inequalities implied by the model is often large. For example, one of the main classes of partially identified models arise from problems of estimating games with multiple equilibria, and even relatively simple static games typically produce a large set of moment inequalities; see, for example, Theorem 1 in \cite{GH11}. More complicated dynamic models, including dynamic games of imperfect information, produce even larger sets of moment inequalities. Researchers therefore had to rely on ad hoc, case-specific, arguments to select a small subset of moment inequalities to which the methods available in the literature so far could be applied. In this paper, we develop systematic methods to treat {\em many} moment inequalities. Our methods are universally applicable in any setting leading to many moment inequalities.\footnote{In some special settings, such as those studied in Theorem 4 of \cite{GH11}, the number of moment inequalities can be dramatically reduced without blowing up the identified set (and so without any subjective choice). However, there are no theoretically justified procedures that would generically allow to decrease the number of moment inequalities in all settings.

In addition, it is important to note that in practice, it may be preferable to use more inequalities than those needed for sharp identification of the model. Indeed, selecting inequalities for statistical inference and selecting a minimal set of inequalities that suffice for sharp identification are rather different problems since the latter problem relies upon the knowledge of the inequalities and does not take into account the noise associated with estimation of inequalities. For example, if a redundant inequality can be estimated with high precision, it may be beneficial to use it for inference in addition to inequalities needed for sharp identification since such an inequality may improve finite sample statistical properties of the inferential procedure.}

There is a variety of economic applications where the problem of testing many moment inequalities appears. One example is the discrete choice model where a consumer is selecting a bundle of products for purchase and moment inequalities come from a revealed preference argument \citep[see][]{Pakes10}.
In this example, one typically has many moment inequalities because the number of different combinations of products from which the consumer is selecting is huge. Another example is the market structure model of \cite{CilibertoTamer2009} where the number of moment inequalities equals the number of possible combinations of firms presented in the market, which is exponentially large in the number of firms that could potentially enter the market. Yet another  example is a dynamic model of imperfect competition of \cite{BBL07}, where deviations from the optimal policy serve to define many moment inequalities.  Other prominent examples leading to many moment inequalities are studied in \cite{BMM11}, \cite{GH11}, \cite{CRS13}, and \cite{CR13} where moment inequalities are used to provide sharp identification regions for parameters in partially identified models. In all these applications, testing moment inequalities allows to carry out inference on structural and causal parameters. In addition, we note that, as explained in \cite{SP18}, our results help to test conditional independence, a concept that plays a particularly important role in causal machine learning; see \cite{P09}.

Many examples above have a very important feature -- the large number of inequalities generated are ``unstructured" in the sense that they can not be viewed as some unconditional moment inequalities generated from a small number of conditional inequalities with a low-dimensional conditioning variable. This means that the existing inference methods for conditional moment inequalities, albeit fruitful in many cases, do not apply to this type of framework, and our methods are precisely aimed at dealing with this important case. We thus view our methods  as strongly complementary to the existing literature.\footnote{A small number of conditional inequalities gives rise to a large number of unconditional inequalities, but these have a certain continuity and tightness structure, which the literature on conditional moment inequalities heavily exploits/relies upon. Our approach works even if such structure is not available and can handle many unstructured moment inequalities. In addition, when such structure is available, our bootstrap methods automatically exploit it leading to powerful tests of structured moment inequalities arising from conversion of a small or large number of conditional moment inequalities.}

There are also many empirical studies where many moment inequalities framework could be useful. Among others, these are \cite{CilibertoTamer2009} who estimated the empirical importance of firm heterogeneity as a determinant of the market structure in the US airline industry,\footnote{\cite{CilibertoTamer2009} had 2742 markets and used four major airline companies and two aggregates of medium size and low cost companies that lead to $2^{4+2+1}=128$ moment inequalities, which is already a large number. However, as established in Theorem 1 of \cite{GH11}, sharp identification bounds in the Ciliberto and Tamer model would require around $2^{2^{4+2}}=2^{64}$  inequalities.} \cite{Holmes11} who estimated the dynamic model of the Wal-Mart expansion,\footnote{\cite{Holmes11} derived moment inequalities from ruling out deviations from the observed Wal-Mart behavior as being suboptimal. He considered the set of potential deviations where the opening dates of some Wal-Mart stores are reordered, and explicitly acknowledged that this leads to the enormous number of inequalities (in fact, this is a number of permutations of 3176 Wal-Mart stores, up to a restriction that the stores opened in the same year can not be permuted). Therefore, he restricted attention to deviations consisting of pairwise resequencing where each deviation switches the opening dates of only two stores. However, one could argue that deviations in the form of block resequencing where the opening dates of blocks of stores are switched are also informative since one of the main features of the Wal-Mart strategy is to pack stores closely together, so that it is easy to set up a distribution network and save on trucking costs.} and  \cite{Ryan12} who estimated the welfare costs of the 1990 Amendments to the Clean Air Act on the U.S. Portland cement industry.\footnote{\cite{Ryan12} adapted an estimation strategy proposed in \cite{BBL07}. He had 517 market-year observations and considered 1250 alternative policies to generate a set of inequalities.}

To formally describe the problem, let $X_{1},\dots,X_{n}$ be a sequence of independent and identically distributed (i.i.d.) random vectors in $\R^{p}$,
where $X_{i}=(X_{i1},\dots,X_{i p})^T$, with a common distribution denoted by $\mathcal L_X$. For  $1 \leq j \leq p$, write  $\mu_{j} := \Ep[ X_{1j} ]$. We are interested in testing the null hypothesis
\begin{align}
&H_{0}: \mu_{j} \leq 0 \quad \text{for all} \ j=1,\dots,p, \label{eq: null hypothesis}
\intertext{against the alternative}
&H_{1}: \mu_{j} > 0 \quad \text{for some} \ j=1,\dots,p. \label{eq: alternative hypothesis}
\end{align}
We refer to (\ref{eq: null hypothesis}) as the moment inequalities, and we say that the $j$th moment inequality is satisfied (violated) if $\mu_j\leq 0$ ($\mu_{j}>0$). Thus $H_0$ is the hypothesis that all the moment inequalities are satisfied.
The primal feature of this paper is that the number of moment inequalities $p$ is allowed to be larger or even much larger than the sample size $n$. 

We consider the test statistic given by the maximum over $p$  Studentized (or $t$-type) inequality-specific statistics (see (\ref{eq: test statistic}) ahead for the formal definition), and propose a number of methods for computing critical values. Specifically, we consider critical values based upon (i) the union bound combined with a moderate deviation inequality for self-normalized sums, and (ii) bootstrap methods.
We will call the first option the {\em SN method} (SN refers to the abbreviation of ``Self-Normalized''). Among bootstrap methods, we consider multiplier and empirical bootstrap procedures abbreviated as {\em MB and EB methods}. The SN method is analytical and is very easy to implement. As such, the SN method is particularly useful for grid search when the researcher is interested in constructing the confidence region for the identified set in the parametric model defined via moment inequalities as in Appendix \ref{sec: confidence region} of the online supplement.
Bootstrap methods are simulation-based and computationally harder. However, an important feature of bootstrap methods is that they take into account the correlation structure of the data and yield lower critical values leading to more powerful tests than those obtained via the SN method. In particular, if the researcher incidentally repeated the same inequality twice or, more importantly, included inequalities with very similar informational content (that is, highly correlated inequalities), the MB/EB methods would be able to account of this and would automatically disregard or nearly disregard these duplicated or nearly duplicated inequalities, without inflating the critical value.


We also consider two-step methods by incorporating inequality selection procedures. The two-step methods get rid of most of {\em uninformative} inequalities, that is inequalities $j$ with $\mu_j<0$ if $\mu_j$ is not too close to $0$. By dropping the uninformative inequalities, the two-step methods produce more powerful tests than those based on the one-step methods, that is, methods without the inequality selection procedures. 

Moreover, we develop novel three-step methods by incorporating double inequality selection procedures.
The three-step methods are suitable in parametric models defined via moment inequalities and allow to drop {\em weakly informative} inequalities in addition to uninformative inequalities.\footnote{The same methods can be extended to nonparametric models as well. In this case, $\theta$ appearing below in this paragraph should be considered as a sieve parameter.} Specifically, consider the model consisting of inequalities $\Ep[g_j(\xi,\theta)]\leq 0$ for all $j=1,\dots,p$ where $\xi$ is a vector of observable random variables, $\theta$ a vector of structural or causal parameters, and $g_1,\dots,g_p$ a set of known functions. Suppose that the researcher is interested in testing the null hypothesis $\theta=\theta_0$ against the alternative $\theta\neq \theta_0$ based on the i.i.d. data $\xi_1,\dots,\xi_n$, so that the problem reduces to (\ref{eq: null hypothesis})-(\ref{eq: alternative hypothesis}) by setting $X_{i j}=g_j(\xi_i,\theta_0)$. We say that the inequality $j$ is weakly informative if the function $\theta\mapsto \Ep[g_j(\xi,\theta)]$ is flat or nearly flat at $\theta=\theta_0$. Dropping weakly informative inequalities allows us to derive tests with higher local power since these inequalities can only provide a weak signal of the violation of the null hypothesis when $\theta$ is close to $\theta_0$.


We prove validity of these methods for computing the critical values, uniformly in suitable classes of distributions $\mathcal L_X$. We derive non-asymptotic bounds on the rejection probabilities, where ``non-asymptotic'' means that the bounds hold with fixed $n$ (and $p$, and all the other parameters), and the dependence of the constants involved in the bounds are stated explicitly.
Notably, under mild conditions, these methods lead to the error in size decreasing polynomially in $n$, while allowing for $p$ much larger than $n$; indeed, $p$ can be of order $\exp (n^{c})$ for some $c > 0$. In addition, we emphasize that although we are primarily interested in the case with $p$ (much) larger than $n$, our methods remain valid when $p$ is small or comparable to $n$.\footnote{When $p$ is small relative to $n$, other tests, e.g. the quasi likelihood-ratio test may be more powerful than the methods developed here; see Section \ref{sec: test statistic} for further discussion.}

An important feature of our methods is that increasing the set of moment inequalities has no or little effect on the critical value. In particular, as a function of the number of moment inequalities $p$, our critical values are always bounded from above by a slowly varying $(\log p)^{1/2}$ (up to a multiplicative constant). This implies that instead of making a subjective choice of inequalities, the researcher should use all (or at least a large set of) available inequalities since using more inequalities gives much larger values of the test statistic when added inequalities violate $H_0$. This feature of our methods is akin to that in modern high-dimensional/big-data techniques like the Lasso and the Dantzig selector that allow for the variable selection in exchange for small cost in the precision of model estimates; see, for example, \cite{BRT09} for an analysis and discussion of the methods of estimating high-dimensional models.

Our results can also be used for the construction of confidence regions for identifiable parameters in partially identified models defined by moment inequalities. In particular, we show in Appendix \ref{sec: confidence region} of the online supplement how to use our results for constructing confidence regions that are {\em asymptotically honest}, with the coverage being correct uniformly in suitably large classes of underlying distributions.

Moreover, we consider two extensions of our results in Appendix \ref{sec: some extensions} of the online supplement. In the first extension, we consider testing many moment inequalities for dependent data. In the second extension, we allow for {\em approximate} inequalities to account of the case where an approximation error arises either from estimated nuisance parameters or from the need to linearize the inequalities. Both of these extensions are important for inference in dynamic models such as those considered in \cite{BBL07}.

The literature on testing (unconditional) moment inequalities is large; see \cite{White00},
\cite{ChernozhukovHongTamer2007}, \cite{RomanoShaikh2008}, \cite{Rosen08}, \cite{AndrewsGuggenberger2009}, \cite{AndrewsSoares2010}, \cite{Canay2010}, \cite{Bugni2011}, \cite{AndrewsBarwick2012}, and \cite{RomanoShaikhWolf2012}. However, these papers deal only with a finite (and fixed) number of moment inequalities. There are also several papers on testing  conditional moment inequalities, which can be treated as an infinite number of unconditional moment inequalities; see \cite{AndrewsShi2013}, \cite{ChernozhukovLeeRosen2013}, \cite{LeeSongWhang13,LeeSongWhang13b}, \cite{Armstrong2011}, \cite{Chetverikov2011}, and \cite{ArmstrongChan2012}. However, when unconditional moment inequalities come from conditional ones, they inherit from original inequalities certain correlation structure that facilitates the analysis of such moment inequalities. In contrast, we are interested in treating many moment inequalities without assuming any correlation structure, motivated by important examples such as those in Cilberto and Tamer (2009), \cite{BBL07}, and Pakes (2010).  \cite{Menzel2008} considered inference for many moment inequalities, but with $p$ growing at most as $n^{2/7}$ (and hence $p$ being much smaller than $n$).
Also his approach and test statistics are different from ours. Finally, \cite{Allen14} recently suggested further extensions and refinements of our new methods. In particular, he noticed that the truncation threshold for our selection procedures can be taken slightly lower (in absolute value) than what we use; he studied an iterative procedure based on \cite{Chetverikov2011}; and he considered moment re-centering procedure similar to that developed in \cite{RomanoShaikhWolf2012}. The latter two possibilities were already noted in the previous versions of our paper.\footnote{See the 2013 version of our paper at arXiv:1312.7614v1.}



The remainder of the paper is organized as follows.
In the next section, we discuss several motivating examples.
In Section \ref{sec: test statistic}, we build our test statistic.
In Section \ref{sec: critical value}, we derive various ways of computing critical values for the test statistic, including the SN, MB, and EB methods and their two-step and three-step variants discussed above, and state results on their validity. In Section \ref{sec: power}, we discuss power properties of our methods. In Section \ref{sec: monte carlo}, we describe Monte Carlo simulations shedding light on how our methods perform in finite samples. Additional results, as well as all the proofs and the results of Monte Carlo simulations, are provided in the online supplement.

\subsection{Notation and convention}
For an arbitrary sequence $\{ z_{i} \}_{i=1}^{n}$, we write $\En [ z_{i} ] = n^{-1} \sum_{i=1}^{n} z_{i}$.
For $a,b \in \R$, we use the notation $a \vee b = \max \{ a,b \}$. For any finite set $J$, we let $| J |$ denote the number of elements in $J$.
The transpose of a vector $z$ is denoted by $z^{T}$.
Moreover, we use the notation $X_{1}^{n} = \{ X_{1},\dots, X_{n} \}$.
In this paper, we (implicitly) assume that the quantities such as $X_{1},\dots,X_{n}$ and $p$ are all indexed by $n$. We are primarily interested in the case where $p=p_{n} \to \infty$ as $n \to \infty$.
However, in most cases, we suppress the dependence of these quantities on $n$ for the notational convenience, and our results also apply to the case with fixed $p$.
Finally, throughout the paper, we assume that $n \geq 2$ and  $p\geq 2$.

\section{Motivating examples}\label{sec: motivating examples}
In this section, we provide three examples that motivate the framework where the number of moment inequalities $p$ is large and potentially much larger than the sample size $n$. In these examples, one actually has many conditional rather than unconditional moment inequalities. Therefore, we emphasize that our results cover the case of many conditional moment inequalities as well.\footnote{Indeed, consider conditional moment inequalities of the form
\begin{equation}\label{eq: conditional moment inequalities}
\Ep[g_j(Y) \mid Z]\leq 0 \quad \text{for all} \ j=1,\dots,p'
\end{equation}
where $(Y,Z)$ is a pair of random vectors and $g_1,\dots,g_{p'}$ is a set of functions with $p'$ being large. Let $\mathcal{Z}$ be the support of $Z$ and assume that $\mathcal{Z}$ is a compact set in $\R^{l}$. Then, following 
\cite{AndrewsShi2013}, one can construct an infinite set $\mathcal{I}$ of instrumental functions $I:\mathcal{Z}\to \R$ such that $I(z)\geq 0$ for all $z\in\mathcal{Z}$ and (\ref{eq: conditional moment inequalities}) holds if and only if
\[
\Ep[g_j(Y) I(Z) ]\leq 0 \quad \text{for all} \ j=1,\dots,p' \ \text{and all} \ I\in\mathcal{I}.
\]
In practice, one can choose a large subset $\mathcal{I}_n$ of $\mathcal{I}$ and consider testing $p=p'|\mathcal{I}_n|$ moment inequalities
\begin{equation}\label{eq: reduction set}
\Ep[g_j(Y) I(Z)]\leq 0 \quad \text{for all} \ j=1,\dots,p' \ \text{and all} \ I\in\mathcal{I}_n.
\end{equation}
If $\mathcal{I}_n$ grows sufficiently fast with $n$, the test of (\ref{eq: conditional moment inequalities}) based on (\ref{eq: reduction set}) will be consistent.} As these examples demonstrate, there is a variety of economic models leading to the problem of testing many unconditional and/or many conditional moment inequalities to which the methods available in the literature so far can not be applied,  and which, therefore, requires the methods developed in this paper. 

\subsection{Market structure model}
This example is based on \cite{CilibertoTamer2009}.\footnote{The market structure model is also often referred to as an entry game.}
Let $m$ denote the number of firms that could potentially enter the market. Let $m$-tuple $D=(D_1,\dots,D_m)$ denote entry decisions of these firms; that is, $D_j=1$ if the firm $j$ enters the market and $D_j=0$ otherwise. Let $\mathcal{D}$ denote the set of possible values of $D$. Clearly, the number of elements $d$ of the set $\mathcal{D}$ is $|\mathcal{D}|=2^m$.

Let $X$ and $\varepsilon$ denote the (exogenous) characteristics of the market as well as the characteristics of the firms that are observed and not observed by the researcher, respectively.
The profit of the firm $j$ is given by
\[
\pi_j(D,X,\varepsilon,\theta),
\]
where the function $\pi_j$ is known up to a parameter $\theta$. Assume that both $X$ and $\varepsilon$ are observed by the firms and that a Nash equilibrium is played, so that for each $j$,
\[
\pi_j((D_j,D_{-j}),X,\varepsilon,\theta)\geq \pi_j((1-D_j,D_{-j}),X,\varepsilon,\theta),
\]
where $D_{-j}$ denotes the decisions of all firms excluding the firm $j$. Then one can find set-valued functions $R_1(d,X,\theta)$ and $R_2(d,X,\theta)$ such that $d$ is the {\em unique} equilibrium whenever $\varepsilon\in R_1(d,X,\theta)$, and $d$ is {\em an} equilibrium whenever $\varepsilon\in R_2(d,X,\theta)$. When $\varepsilon\in R_1(d,X,\theta)$ for some $d\in \mathcal{D}$, we know for sure that $D=d$ but when $\varepsilon\in R_2(d,X,\theta)$, the probability that $D=d$ depends on the equilibrium selection mechanism, and, without further information, can be anything in $[0,1]$.  Therefore, we have the following bounds
\begin{multline*}
\Ep\left[1\{\varepsilon\in R_1(d,X,\theta) \mid X\right] \leq \Ep\left[1\{D=d\} \mid X\right] \\
\leq \Ep\left[1\{\varepsilon\in R_1(d,X,\theta)\cup R_2(d,X,\theta)\} \mid X\right],
\end{multline*}
for all $d\in \mathcal{D}$. Further, assuming that the conditional distribution of $\varepsilon$ given $X$ is known (alternatively, it can be assumed that this distribution is known up to a parameter that is a part of the parameter $\theta$), both the left- and the right-hand sides of these inequalities can be calculated. Denote them by $P_1(d,X,\theta)$ and $P_2(d,X,\theta)$, respectively, to obtain
\begin{equation}\label{eq: inequalities CT}
P_1(d,X,\theta)\leq \Ep\left[1\{D=d\} \mid X\right]\leq P_2(d,X,\theta)\text{ for all }d\in \mathcal{D}.
\end{equation}
These inequalities can be used for inference about the parameter $\theta$. Note that the number of inequalities in (\ref{eq: inequalities CT}) is $2|\mathcal{D}|=2^{m+1}$, which is a large number even if $m$ is only moderately large. Moreover, these inequalities are conditional on $X$. For inference about the parameter $\theta$, each of these inequalities can be transformed into a large and increasing number of unconditional inequalities as described above. Also, if the firms have more than two decisions, the number of inequalities will be even (much) larger. Finally, one can produce even larger set of inequalities in this example using the bounds of \cite{GH11}; see Section \ref{sec: monte carlo market structure} for details. Therefore, our framework is exactly suitable for this example.

\subsection{Discrete choice model with endogeneity}
Our second example is based on \cite{CRS13}. The source of many moment inequalities in this example is different from that in the previous example. Consider an individual who is choosing an alternative $d$ from a set $\mathcal{D}$ of available options. Let $M=|\mathcal{D}|$ denote the number of available options. Let $D$ denote the choice of the individual. From choosing an alternative $d$, the individual obtains the utility
\[
u(d,X,V),
\]
where $X$ is a vector of observable (by the researcher) covariates and $V$ is a vector of unobservable (by the researcher) utility shifters. The individual observes both $X$ and $V$ and makes a choice based on utility maximization, so that $D$ satisfies
\[
u(D,X,V)\geq u(d,X,V)\text{ for all }d\in \mathcal{D}.
\]
The object of interest in this model is the pair $(u,P_V)$ where $P_V$ denotes the distribution of the vector $V$.

In many applications, some components of $X$ may be endogenous in the sense that they are not independent of $V$. Therefore, to achieve (partial) identification of the pair $(u,P_V)$, following \cite{CRS13}, assume that there exists a vector $Z$ of observable instruments that are independent of $V$. 
Let $\mathcal{V}$ denote the support of $V$, and let $\mathcal{\tau}(d,X,u)$ denote the subset of $\mathcal{V}$ such that $D=d$ whenever $X=x$ and $V\in\mathcal{\tau}(d,x,u)$, so that
\begin{equation}\label{eq: tau inclusion}
V\in\tau(D,X,u).
\end{equation}
Then for any set $S\subset\mathcal{V}$,
\begin{equation}\label{eq: CRS identification}
\Ep\left[1\{V\in S\}\right]=\Ep\left[1\{V\in S\} \mid Z\right]\geq \Ep\left[1\{\tau(D,X,u)\subset S\} \mid Z\right],
\end{equation}
where the equality follows from independence of $V$ from $Z$, and the inequality from (\ref{eq: tau inclusion}). Note that the left-hand side of (\ref{eq: CRS identification}) can be calculated (for fixed distribution $P_V$) and equals $P_V(S)$, so that we obtain
\begin{equation}\label{eq: CRS inequalities}
P_V(S)\geq \Ep\left[1\{\tau(D,X,u)\subset S\} \mid Z\right]\text{ for all }S\in\mathcal{S},
\end{equation}
where $\mathcal{S}$ is some collection of sets in $\mathcal{V}$. Inequalities (\ref{eq: CRS inequalities}) can be used for inference about the pair $(u,P_V)$.
A natural question then is what collection of sets $\mathcal{S}$ should be used in (\ref{eq: CRS inequalities}). \cite{CRS13} showed that sharp identification of the pair $(u,P_V)$ is achieved by considering all unions of sets on the support of $\tau(D,X,u)$ with the property that the union of the interiors of these sets is a connected set. When $X$ is discrete with the support consisting of $m$ points, this implies that the class $\mathcal{S}$ may consist of $M\cdot2^m$ sets, which is a large number even for moderately large $m$. Moreover, as in our previous example, inequalities in (\ref{eq: CRS inequalities}) are conditional giving rise to even a larger set of inequalities when transformed into unconditional ones. Therefore, our framework is again exactly suitable for this example. 

Also, we note that the model described in this example fits as a special case into a Generalized Instrumental Variable framework set down and analyzed by \cite{CR13}, where the interested reader can find other examples leading to many moment inequalities.


\subsection{Dynamic model of imperfect competition}\label{sec: dynamic game}
This example is based on \cite{BBL07}. In this example, many moment inequalities arise from ruling out deviations from best responses in a dynamic game. Consider a market consisting of $N$ firms. Each firm $j$ makes a decision $A_{j t}\in\mathcal{A}$ at time periods $t=0,1,2,\dots,\infty$. Let $A_t=(A_{1 t},\dots,A_{N t})$ denote the $N$-tuple of decisions of all firms at period $t$. The profit of the firm $j$ at period $t$, denoted by $\pi_j(A_t,S_t,\nu_{j t})$, depends on the $N$-tuple of decisions $A_t$, the state of the market $S_t\in \mathcal{S}$ at period $t$, and the firm- and time-specific shock $\nu_{j t}\in\mathcal{V}$. Assume that the state of the market $S_t$ follows a Markov process, so that $S_{t+1}$ has the distribution function $P(S_{t+1}|A_t,S_t)$, and that $\nu_{j t}$'s are i.i.d. across firms $j$ and time periods $t$ with the distribution function $G(\nu_{j t})$. In addition, assume that when the firm $j$ is making a decision $A_{j t}$ at period $t$, it observes $S_t$ and $\nu_{j t}$ but does not observe $\nu_{-j t}$, the specific shocks of all its rivals, and that the objective function of the firm $j$ at period $t$ is to maximize
\[
\Ep\left[\sum_{\tau=t}^\infty \beta^{\tau-t}\pi_j(A_\tau,S_\tau,\nu_{j t}) \mid S_t \right],
\]
where $\beta$ is a discount factor. Further, assume that a Markov Perfect Equilibrium (MPE) is played in the market. Specifically, let $\sigma_j:\mathcal{S}\times\mathcal{V}\to\mathcal{A}$ denote the MPE strategy of firm $j$, and let $\sigma:=(\sigma_1,\dots,\sigma_N)$ denote the $N$-tuple of strategies of all firms. Define the value function of the firm $j$ in the state $s\in\mathcal{S}$ given the profile of strategies $\sigma$, $V_j(s,\sigma)$, by the Bellman equation:
\[
V_j(s,\sigma):=\Ep_\nu\Big[\pi_j(\sigma(s,\nu),s,\nu_j)+\beta\int V_j(s',\sigma)d P(s' \mid \sigma(s,\nu),s)\Big],
\]
where $\sigma(s,\nu)=(\sigma_1(s,\nu_1),\dots,\sigma_N(s,\nu_N))$, and expectation is taken with respect to $\nu=(\nu_1,\dots,\nu_N)$ consisting of $N$ i.i.d. random variables $\nu_j$ with the distribution function $G(\nu_j)$. Then the profile of strategies $\sigma$ is an MPE if for any $j=1,\dots, N$ and $\sigma_j':\mathcal{S}\times\mathcal{V}\to\mathcal{A}$, we have
\begin{align*}
V_j(s,\sigma)&\geq V_j(s,\sigma_j',\sigma_{-j})\\
&=\Ep_\nu\Big[\pi_j(\sigma_j'(s,\nu_i),\sigma_{-j}(s,\nu_{-j}),s,\nu_j)\\
&\qquad+\beta\int V_j(s',\sigma_j',\sigma_{-j})d P(s' \mid \sigma_j'(s,\nu_j),\sigma_{-j}(s,\nu_{-j}),s)\Big],
\end{align*}
where $\sigma_{-j}$ is strategies of all rivals of the firm $j$ in the profile $\sigma$. 

For estimation purposes, assume that the functions $\pi_j(A_t,S_t,\nu_{j t})$ and $G(\nu_{j t})$ are known up to a finite dimensional parameter $\theta$, that is we have $\pi_j(A_t,S_t,\nu_{j t})=\pi_j(A_t,S_t,\nu_{j t},\theta)$ and $G(\nu_{j t})=G(\nu_{j t},\theta)$, so that the value function $V_j(s,\sigma)=V_j(s,\sigma,\theta)$ also depends on $\theta$, and the goal is to estimate $\theta$. Assume that the data consist of observations on $n$ similar markets for a short span of periods or observations on one market for $n$ periods. In the former case, assume also that the same MPE is played in all markets.\footnote{In the case of data consisting of observations on one market for $n$ periods, one has to use techniques for dependent data developed in Appendix \ref{sec: dependent data} of the online supplement. It is also conceptually straightforward to extend our techniques to the case when the data consist of observations on many markets for many periods, as happens in some empirical studies. We leave this extension for future work.}

In this model, \cite{BBL07} suggested a computationally tractable two-stage procedure to estimate the structural parameter $\theta$. An important feature of their procedure is that it does not require point identification of the model. The first stage of their procedure consists of estimating transition probability function $P(S_{t+1}|S_t,A_t)$ and policy functions (strategies) $\sigma_j(s,\nu_j)$. Following their presentation, assume that these functions are known up to a finite dimensional parameter $\alpha = (\alpha_1,\alpha_2)$, that is $P(S_{t+1}|S_t,A_t)=P(S_{t+1}|S_t,A_t,\alpha_1)$ and $\sigma_j(s,\nu_j)=\sigma_j(s,\nu_j,\alpha_2)$, and that the first stage yields a $\sqrt{n}$-consistent estimator $\hat{\alpha}_n = (\hat\alpha_{n,1},\hat\alpha_{n,2})$ of $\alpha = (\alpha_1,\alpha_2)$.\footnote{Estimation of $\alpha_1$ is simple; for example, it can be estimated by the maximum likelihood method. Estimation of $\alpha_2$ is more complicated since the functions $\sigma_j(s,\nu_j)$ depend on unobservable $\nu_j$'s and requires additional assumptions. When the set $\mathcal A$ is finite, for example, one can assume that the shock $\nu_j$ is additively separable in the profit function, so that $\pi_j(A_t,S_t,\nu_{j t}) = \tilde\pi_j(A_t,S_t) + \nu_i(A_{j t})$, where the vector $\{\nu_i(A)\}_{A\in\mathcal A}$ consists of i.i.d. random variables, and use the methods of \cite{HM93} to estimate $\alpha_2$; see \cite{BBL07} for details.} Using $\hat{\alpha}_{n,1}$, one can estimate the transition probability function by $P(S_{t+1}|S_t,A_t,\hat{\alpha}_{n,1})$, and then one can calculate the (estimated) value function of the firm $j$ at every state $s\in\mathcal{S}$, $\hat{V}_j(s,\sigma',\theta)$, for any profile of strategies $\sigma'$ and any value of the parameter $\theta$ using forward simulation as described in \cite{BBL07}. Here we have $\hat{V}_j(s,\sigma',\theta)$ instead of $V_j(s,\sigma',\theta)$ because forward simulations are based on the estimated transition probability function $P(S_{t+1}|S_t,A_t,\hat{\alpha}_n)$ instead of the true function $P(S_{t+1}|S_t,A_t,\alpha)$. Then, on the second stage, one can test the equilibrium conditions
$$
V_j(s,\sigma_j,\sigma_{-j},\theta)\geq V_j(s,\sigma_j',\sigma_{-j},\theta)
$$
for all $j=1,\dots, N$, $s\in \mathcal{S}$, and $\sigma_j'\in \Sigma$ for some set of strategies $\Sigma$ by considering inequalities
\begin{equation}\label{eq: equilibrium conditions BBL}
\hat{V}_j(s,\hat\sigma_j,\hat\sigma_{-j},\theta)\geq \hat{V}_j(s,\sigma_j',\hat\sigma_{-j},\theta)
\end{equation}
where $\hat\sigma_j = \sigma_j(\hat{\alpha}_{n,2})$ and $\hat\sigma_{-j} = \sigma_{-j}(\hat{\alpha}_{n,2})$ are the estimated policy functions for the firm $j$ and all of its rivals, respectively. Inequalities (\ref{eq: equilibrium conditions BBL}) can be used to test hypotheses about the parameter $\theta$. The number of inequalities is determined by the number of elements in $\Sigma$. Assuming that $\mathcal{A}$, $\mathcal{S}$, and $\mathcal{V}$ are all finite, we obtain $|\Sigma|=|\mathcal{A}|^{|\mathcal{S}|\cdot|\mathcal{V}|}$, so that the total number of inequalities is $N\cdot |\mathcal{S}|\cdot |\Sigma|$, which is a very large number in all but trivial empirical applications.

Inequalities (\ref{eq: equilibrium conditions BBL}) do not fit directly into our testing framework (\ref{eq: null hypothesis})-(\ref{eq: alternative hypothesis}). One possibility to go around this problem is to use a jackknife procedure.
To explain the procedure, assume that the data consist of observations on $n$ i.i.d markets. Let $\hat V_j^{-i}(s,\sigma',\theta)$ and $\hat\sigma^{-i}$ denote the leave-market-$i$-out estimates of $V_j(s,\sigma',\theta)$ and $\sigma$, respectively. Define
\begin{align*}
\tilde{X}_{i j}(s,\theta)&:=n\hat{V}_j(s,\hat\sigma_j,\hat\sigma_{-j},\theta)-(n-1)\hat{V}_j^{-i}(s,\hat\sigma_j^{-i},\hat\sigma_{-j}^{-i},\theta)
\end{align*}
and
\begin{align*}
\tilde{X}_{i j}'(s,\sigma_j',\theta)&:=n\hat{V}_j(s,\sigma_j',\hat\sigma_{-j},\theta)-(n-1)\hat{V}_j^{-i}(s,\sigma_j',\hat\sigma_{-j}^{-i},\theta).
\end{align*}
Also, define
$$
\hat{X}_{i j}(s,\sigma_j',\theta):=\tilde{X}_{i j}'(s,\sigma_j',\theta)-\tilde{X}_{i j}(s,\theta).
$$
Then under some regularity conditions including smoothness of the value function $V_j(s,\sigma)$, one can show that
\begin{equation}\label{eq: approximate data}
\hat{X}_{i j}(s,\sigma_j',\theta)=X_{i j}(s,\sigma_j',\theta)+o_P(1)
\end{equation}
for some $X_{i j}(s,\sigma_j',\theta)$ satisfying
\begin{equation}\label{eq: mean to test}
\Ep[X_{i j}(s,\sigma_j',\theta)]=V_j(s,\sigma_j',\sigma_{-j},\theta)-V_j(s,\sigma,\theta)\leq 0,
\end{equation}
where $X_{i j}(s,\sigma_j',\theta)$'s are independent across markets $i=1,\dots,n$. We provide some details on the derivation of (\ref{eq: approximate data}) and \eqref{eq: mean to test} in Appendix \ref{sec: details} of the online supplement. Now we can use the results of Appendix \ref{sec: approximate moment inequalities} on testing approximate moment inequalities to do inference about the parameter $\theta$ if we replace $X_{i j}(s,\sigma_j',\theta)$ by the ``data'' $\hat{X}_{i j}(s,\sigma_j',\theta)$ and, in addition, we use $(\hat{V}_j(s,\sigma_j',\hat\sigma_{-j},\theta)-\hat{V}_j(s,\hat\sigma_j,\hat\sigma_{-j},\theta))$ instead of $\hat{\mu}_j=n^{-1}\sum_{i=1}^n\hat{X}_{i j}(s,\sigma_j',\theta)$ in the numerator of our test statistic defined in (\ref{eq: test statistic}).\footnote{Note that one of the conditions of Theorem \ref{thm: approximate moment inequalities} is that \eqref{eq: main condition approximate inequalities} holds with $\hat\mu_{j,0} = n^{-1}\sum_{i=1}X_{i j}(s,\sigma_j',\theta)$ in our case, and since we can only guarantee that $\hat X_{i j}(s,\sigma_j',\theta) - X_{i j}(s,\sigma_j',\theta) = O_P(n^{-1/2})$ as in \eqref{eq: approximate data}, this condition may not be satisfied if we define $\hat\mu_j = n^{-1}\sum_{i=1}^n\hat X_{i j}(s,\sigma_j',\theta)$. This condition is satisfied, however, under mild regularity conditions, if we define $\hat\mu_j = \hat{V}_j(s,\sigma_j',\hat\sigma_{-j},\theta)-\hat{V}_j(s,\hat\sigma_j,\hat\sigma_{-j},\theta)$; see the online supplement for details.} Thus, this example fits into our framework as well.\footnote{The jackknife procedure described above may be computationally intensive in some applications but, on the other hand, the required computations are rather straightforward. In addition, this procedure only involves the first stage estimation, which is typically computationally simple. Moreover, bootstrap procedures developed in this paper do not interact with the jackknife procedure, so that the latter procedure has to be performed only once.}

\section{Test statistic}\label{sec: test statistic}

We begin with preparing some notation.
Recall that $\mu_{j} = \Ep [ X_{1j} ]$. We assume that
\begin{equation}
\Ep[ X_{1j}^{2} ] < \infty, \ \sigma_{j}^{2} := \Var (X_{1j}) > 0, \ j=1,\dots, p. \label{eq: variance cond}
\end{equation}
For $j=1,\dots,p$, let $\hat{\mu}_{j}$ and $\hat{\sigma}_{j}^{2}$ denote the sample mean and variance of $X_{1j},\dots,X_{nj}$, respectively, that is,
\[
\hat{\mu}_{j}=\En[X_{ij}]=\frac{1}{n}\sum_{i=1}^n X_{ij}, \ \hat{\sigma}_{j}^{2} = \En [ (X_{ij} - \hat{\mu}_j])^{2} ] = \frac{1}{n} \sum_{i=1}^{n} (X_{ij} - \hat{\mu}_{j})^{2}.
\]
Alternatively, we can use $\tilde{\sigma}_{j}^{2}=(1/(n-1))\sum_{i=1}^{n}(X_{ij}-\hat{\mu}_{j})^{2}$ instead of $\hat{\sigma}_{j}^{2}$, which does not alter the overall conclusions of the theorems ahead. In all what follows, however, we will use $\hat{\sigma}_{j}^{2}$.

There are several different statistics that can be used for testing the null hypothesis (\ref{eq: null hypothesis}) against the alternative (\ref{eq: alternative hypothesis}).
Among all possible statistics, it is natural to consider statistics that take large values when some of $\hat{\mu}_j$'s are large.
In this paper, we focus on the statistic that takes large values when at least one of $\hat{\mu}_j$'s is large. One can also consider either non-Studentized or Studentized versions of the test statistic.
For a non-Studentized statistic, we mean a function of $\hat{\mu}_1,\dots,\hat{\mu}_p$, and for a Studentized statistic, we mean a function of $\hat{\mu}_1/\hat{\sigma}_1,\dots,\hat{\mu}_p/\hat{\sigma}_p$.
Studentized statistics are often considered preferable. In particular, they are scale-invariant (that is, multiplying $X_{1j},\dots,X_{nj}$  by a scalar value does not change the value of the test statistic), and
they typically spread the power evenly among the different moment inequalities $\mu_{j}\leq 0$.
See \cite{RomanoWolf05} for a detailed comparison of Studentized versus non-Studentized statistics in a related context of multiple hypothesis testing. In our case, Studentization also has an advantage that it allows us to derive an analytical critical value for the test under weak moment conditions. In particular, for our SN critical values, we will only require finiteness (existence) of $\Ep[|X_{1j}|^{3}]$ (see Section \ref{sec: analytical plugin}). As far as MB/EB critical values are concerned, our theory can cover a non-Studentized statistic but Studentization leads to easily interpretable regularity conditions. For these reasons, in this paper we study the Studentized version of the test statistic.

To be specific, we focus on the following test statistic:
\begin{equation}\label{eq: test statistic}
T=\max_{1\leq j\leq p} \frac{\sqrt{n}\hat{\mu}_j}{\hat{\sigma}_j}.
\end{equation}
Large values of $T$ indicate that $H_0$ is likely to be violated, so that it would be natural to consider the test of the form
\begin{equation}
\label{eq: test}
T > c \Rightarrow \text{reject} \ H_{0},
\end{equation}
where $c$ is a critical value suitably chosen in such a way that the test has approximately size $\alpha \in (0,1)$.
We will consider various ways for calculating critical values and  prove their validity.

Rigorously speaking, the test statistic $T$ is not defined when $\hat{\sigma}_{j}^{2}=0$ for some $j=1,\dots,p$. In such cases, we interpret the meaning of ``$T > c$''  in (\ref{eq: test}) as $\sqrt{n}\hat{\mu}_j > c \hat{\sigma}_{j}$ for some $j=1,\dots,p$,
which makes sense even if $\hat{\sigma}_{j}^{2}=0$ for some $j=1,\dots,p$. We will obey such conventions if necessary without further mentioning.

Other types of test statistics are possible. For example, one alternative is the test statistic of the form
\begin{equation}\label{eq: quadratic sum test statistic}
T'=\sum_{j=1}^{p} \left ( \max \{ \sqrt{n}\hat{\mu}_j/\hat{\sigma}_j, 0 \} \right )^{2}.
\end{equation}
The statistic $T'$ has an advantage that it is less sensitive to outliers. However, $T'$ leads to good power only if many inequalities are violated simultaneously. In general, $T'$ is preferable against $T$ if the researcher is interested in detecting deviations when many inequalities are violated simultaneously, and $T$ is preferable against $T'$ if the main interest is in detecting deviations when at least one moment inequality is violated too much. When $p$ is large, as in our motivating examples, the statistic $T$ seems preferable over $T'$ because the critical value for the test based on $T$ grows very slowly with $p$ (at most as $(\log p)^{1/2}$) whereas one can expect that the critical value for the test based on $T'$ grows at least polynomially with $p$.

Another alternative is the quasi likelihood-ratio test statistic of the form
\[
T''=\min_{t\leq 0} n(\hat{\mu}-t)^T\hat{\Sigma}^{-1}(\hat{\mu}-t),
\]
where $\hat{\mu}=(\hat{\mu}_1,\dots,\hat{\mu}_p)^T$, $t = (t_{1},\dots, t_{p})^{T} \leq 0$ means $t_j\leq 0$ for all $j=1,\dots,p$, and $\hat{\Sigma}$ is some $p$ by $p$ symmetric positive definite matrix. This statistic in the context of testing moment inequalities was first studied by \cite{Rosen08} when the number of moment inequalities $p$ is fixed; see also \cite{Wolak91} for the analysis of this statistic in a different context. Typically, one wants to take $\hat{\Sigma}$ as a suitable estimate of the covariance matrix of $X_{1}$, denoted by $\Sigma$. However, when $p$ is larger than $n$, it is not possible to consistently estimate $\Sigma$ without imposing some structure (such as sparsity) on it. Moreover, the results of \cite{BaiSaranadasa96} suggest that the statistic $T'$ or its variants may lead to higher power than $T''$ even when $p$ is smaller than but close to $n$. On the other hand, when $p$ is small relative to $n$, the test statistic $T''$ may lead to more powerful tests than those based on $T$ and $T'$ since it takes into account the correlation structure between the inequalities, like GMM does in the setting of moment equalities. For the rest of the paper, we focus on the statistic $T$ and do not provide critical values for the tests based on $T'$ and $T''$.

%

\section{Critical values}
\label{sec: critical value}

In this section, we study several methods to compute critical values for the test statistic $T$ so that under $H_0$, the probability of rejecting $H_0$ does not exceed size $\alpha$ asymptotically. The methods are essentially ordered by increasing computational complexity, increasing strength of required conditions, but also increasing power. We note, however, that all our methods require only mild conditions on the underlying distributions and are computationally rather simple.

 The basic idea for construction of critical values for $T$ lies in the fact that under $H_{0}$,
\begin{equation}\label{eq: under null}
T\leq \max_{1\leq j\leq p}\sqrt{n}(\hat{\mu}_j-\mu_{j})/\hat{\sigma}_{j},
\end{equation}
where the equality holds when all the moment inequalities are binding, that is, $\mu_{j}=0$ for all $j=1,\dots,p$.
Hence in order to make the test to have size $\alpha$, it is enough to choose the critical value  as (a bound on) the $(1-\alpha)$-quantile of the distribution of $\max_{1\leq j\leq p}\sqrt{n}(\hat{\mu}_j-\mu_{j})/\hat{\sigma}_{j}$.
We consider two approaches to construct such critical values: self-normalized and bootstrap methods. We also consider two- and three-step variants of the methods by incorporating inequality selection.


We will use the following notation. Pick any $\alpha \in (0,1/2)$. Let
\begin{equation}\label{eq: Z}
Z_{ij} = (X_{ij}-\mu_{j})/\sigma_{j}, \ \text{and} \ Z_{i}=(Z_{i1},\dots,Z_{ip})^{T}.
\end{equation}
Observe that $\Ep [ Z_{ij} ] = 0$ and $\Ep [ Z_{ij}^{2} ] = 1$. Define
\[
M_{n,k}= \max_{1 \leq j \leq p} \Big(\Ep [|Z_{1j}|^{k}]\Big)^{1/k}, \ k=3,4, \ \ B_{n} = \Big(\Ep\Big[ \max_{1 \leq j \leq p} Z_{1j}^{4}\Big]\Big)^{1/4}.
\]
($M_{n,k}$ and $B_{n}$ depend on $n$ since $p=p_{n}$ (implicitly) depends on $n$.)
Note that by Jensen's inequality, $B_{n} \geq M_{n,4} \geq M_{n,3} \geq 1$. In addition, if $Z_{i j}$'s are all bounded by a constant $C$ almost surely, we have $C\geq B_n$. These inequalities are useful to get a sense of various conditions on $M_{n,3}$, $M_{n,4}$, and $B_n$ imposed in the theorems below.

\subsection{Self-Normalized methods}\label{sec: analytical methods}

\subsubsection{One-step method}
\label{sec: analytical plugin}
The self-normalized method (abbreviated as the SN method in what follows) we consider is based upon the union bound combined with a moderate deviation inequality for self-normalized sums.  Because of inequality (\ref{eq: under null}), under $H_{0}$,
\begin{equation}\label{eq: SN union bound}
\Pr (T>c) \leq \sum_{j=1}^{p} \Pr ( \sqrt{n}(\hat{\mu}_{j}-\mu_{j})/\hat{\sigma}_{j} > c).
\end{equation}
At a first sight, this bound might look too crude when $p$ is large since, as long as $X_{ij}$'s have polynomial tails, applying, for example, the Markov inequality would only allow us to show that the right-hand side of \eqref{eq: SN union bound} is bounded from above by $\alpha$ when $c$ is growing {\em polynomially} fast with $p$, and using such $c$ would yield a test with low power. However, the Markov inequality is far from being sharp here. Instead, we will exploit the self-normalizing nature of the quantity $ \sqrt{n}(\hat{\mu}_{j}-\mu_{j})/\hat{\sigma}_{j}$ to show that the right-hand side of \eqref{eq: SN union bound} is bounded from above by $\alpha$, up to a vanishing term, even if $c$ is growing {\em logarithmically} fast with $p$. Using such $c$ will in turn yield a test with much better power properties.

For $j=1,\dots,p$, define
\[
U_{j}=\sqrt{n}\En [ Z_{ij} ]/\sqrt{\En [ Z_{ij}^{2}]}.
\]
By simple algebra, we see that
\[
\sqrt{n}(\hat{\mu}_j-\mu_{j})/\hat{\sigma}_j=U_j/\sqrt{1-U_j^2/n},
\]
where the right-hand side is increasing in $U_{j}$ as long as $U_{j} \geq 0$. Hence under $H_0$,
\begin{equation}
\Pr (T>c)
\leq \sum_{j=1}^{p} \Pr \left ( U_{j} > c/\sqrt{1+c^2/n} \right), \ c \geq 0. \label{eq: first step}
\end{equation}
Now, the moderate deviation inequality for self-normalized sums of \cite{JSW03} (see Lemma \ref{lem: moderate deviations} in the online supplement) implies that for moderately large $c \geq 0$,
\[
 \Pr \left ( U_{j} > c/\sqrt{1+c^2/n} \right) \approx \Pr \left( N(0,1) > c/\sqrt{1+c^2/n}\right)
\]
even if $Z_{i j}$ only have $2+\delta$ finite moments for some $\delta>0$.
Therefore, we take the critical value as
\begin{equation}
c^{SN}(\alpha)=\frac{\Phi^{-1}(1-\alpha/p)}{\sqrt{1-\Phi^{-1}(1-\alpha/p)^2/n}},  \label{SN critical value}
\end{equation}
where $\Phi(\cdot)$ is the distribution function of the standard normal distribution, and $\Phi^{-1}(\cdot)$ is its quantile function.
We will call $c^{SN}(\alpha)$ the (one-step) SN critical value with size $\alpha$ as its derivation depends on the moderate deviation inequality for self-normalized sums.
Note that
\[
\Phi^{-1}(1-\alpha/p) \sim \sqrt{\log (p/\alpha)},
\]
so that $c^{SN}(\alpha)$ depends on $p$ only through $\log p$.



The following theorem provides a non-asymptotic bound on the probability that the test statistic $T$ exceeds the SN critical value $c^{SN}(\alpha)$ under $H_0$ and shows that the bound converges to $\alpha$ under mild regularity conditions, thereby validating the SN method.

\begin{theorem}[Validity of one-step SN method]
\label{thm: analytical plugin method}
Suppose that $M_{n,3}\Phi^{-1}(1-\alpha/p) \leq n^{1/6}$.  Then under $H_0$,
\begin{equation}\label{eq: theorem 1 nonasymptotic}
\Pr(T>c^{SN}(\alpha))\leq \alpha \left [ 1 + K n^{-1/2} M_{n,3}^{3} \{ 1+\Phi^{-1}(1-\alpha/p)\}^{3} \right ],
\end{equation}
where $K$ is a universal constant. Hence, if there exist constants $0<c_{1}<1/2$ and $C_{1} >0$ such that
\begin{equation}
M_{n,3}^{3} \log^{3/2} (p/\alpha) \leq C_{1}n^{1/2-c_{1}}, \label{eq: cond1}
\end{equation}
then there exists a positive constant $C$ depending only on $C_1$ such that under $H_{0}$,
\begin{equation}\label{eq: SN test size bound}
\Pr(T>c^{SN}(\alpha)) \leq \alpha + Cn^{-c_{1}}.
\end{equation}
Moreover, this bound holds uniformly over all distributions $\mathcal L_X$ satisfying (\ref{eq: variance cond}) and (\ref{eq: cond1}). In addition, if \eqref{eq: cond1} holds, all components of $X_1$ are independent, $\mu_j = 0$ for all $1\leq j\leq p$, and $p = p_n\to\infty$, then
\begin{equation}\label{eq: third assertion sn critical value}
\Pr(T > c^{SN}(\alpha)) \to 1 - e^{-\alpha}.
\end{equation}
\end{theorem}

\begin{remark}[On conditions of Theorem \ref{thm: analytical plugin method}]
Since condition \eqref{eq: cond1} is abstract, it is instructive to see how this condition looks in particular examples. Suppose, for example, that all $X_{i j}$'s are Gaussian. Then all $Z_{i j}$'s are standard Gaussian, and so $\Ep[|Z_{1 j}|^3] = (8/\pi)^{1/2}$. Hence, it follows that $M_{n,3} = \max_{1\leq j\leq p}(\Ep[|Z_{1 j}|^3])^{1/3} = (8/\pi)^{1/6}$, and condition \eqref{eq: cond1} reduces to $\log^{3/2}(p/\alpha) \leq C_1 n^{1/2-c_1}$ (with a different constant $C_1$). When $\alpha$ is independent of $n$, the condition further reduces to $(\log^3 p) / n \leq C_1 n^{-c_1}$ (with possibly different constants $c_1$ and $C_1$).
\end{remark}

\begin{remark}[Relaxing conditions of Theorem \ref{thm: analytical plugin method}]
The theorem assumes that $\max_{1 \leq j \leq p} \Ep [| X_{1j}|^{3} ] < \infty$ (so that $M_{n,3}<\infty$) but allows this quantity to diverges as $n \to \infty$ (recall $p=p_{n}$). In principle, $M_{n,3}$ that appears in the theorem could be replaced by $\max_{1 \leq j \leq p} (\Ep [ | Z_{1j} |^{2+\nu}])^{1/(2+\nu)}$ for $0<\nu \leq 1$,  which would further weaken moment conditions; however, for the sake of simplicity of presentation, we do not explore this generalization.
\end{remark}

\begin{remark}[On conservativeness of the one-step SN method]
The last asserted claim of Theorem \ref{thm: analytical plugin method}, \eqref{eq: third assertion sn critical value}, shows that when $p$ is large, all components of $X_1$ are independent, and all inequalities satisfy the null and are binding, the one-step SN method is approximately non-conservative. Indeed, the nominal level $\alpha$ is typically small, e.g. 5\% or 10\%, so that $e^{-\alpha} \approx 1 - \alpha$, and the probability of rejecting the null is approximately $\alpha$ in this case. 
\end{remark}

\begin{remark}[Comparison with the classical Bonferroni procedure]
The classical Bonferroni approach to test \eqref{eq: null hypothesis} against  \eqref{eq: alternative hypothesis} would be to compare the statistic $T$ with the Bonferroni critical value $c^{Bon}(\alpha) = \Phi^{-1}(1 - \alpha / p)$. It is straightforward to show using standard techniques that this approach works (controls size) when $p$ is much smaller than $n$ or $X_i$'s are Gaussian. In contrast, our techniques do not require these conditions, which is important because it allows us to test many moment inequalities in a wide variety of settings, without assuming Gaussianity. In addition, using our techniques, it is possible to show that the Bonferroni approach also works under the same conditions as those required for our SN method; see Theorem \ref{thm: bonferroni} in the online supplement.
\end{remark}


\subsubsection{Two-step method}

We now turn  to combine the SN method with inequality selection.
We begin with stating the motivation for inequality selection.

Observe that when $\mu_j<0$ for some $j=1,\dots,p$, inequality (\ref{eq: under null}) becomes strict, so that when there are many $j$ for which $\mu_{j}$ are negative and large in absolute value, the resulting test with one-step SN critical values would tend to be unnecessarily conservative.
Hence it is intuitively clear that, in order to improve the power of the test, it is better to exclude $j$ for which $\mu_{j}$ are below some (negative) threshold when computing critical values. This is the basic idea behind inequality selection.

More formally, let $0 < \beta_{n}<\alpha/2$ be some constant. For generality, we allow $\beta_{n}$ to depend on $n$; in particular, $\beta_{n}$ is allowed to decrease to zero as the sample size $n$ increases. Let $c^{SN}(\beta_{n})$ be the SN critical value with size $\beta_{n}$, and define the set $\hat{J}_{SN} \subset \{ 1,\dots, p \}$ by
\begin{equation}
\hat{J}_{SN}: = \left\{ j \in \{ 1,\dots, p \} : \sqrt{n} \hat{\mu}_{j}/\hat{\sigma}_j > -2 c^{SN}(\beta_{n}) \right\}. \label{eq: SN set}
\end{equation}
Let $\hat{k}$ denote the number of elements in $\hat{J}_{SN}$, that is,
\[
\hat{k} = | \hat{J}_{SN} |.
\]
Then the two-step SN critical value is defined by
\begin{equation}
c^{SN,2S}(\alpha)=
\begin{cases}
\frac{\Phi^{-1}(1-(\alpha-2\beta_{n})/\hat{k})}{\sqrt{1-\Phi^{-1}(1-(\alpha-2\beta_{n})/\hat{k})^2/n}}, &\text{if $\hat{k} \geq 1$}, \\
0, & \text{if $\hat{k} = 0$}.
\end{cases}
\label{SN-MS critical value}
\end{equation}

%

The following theorem establishes validity of this critical value.

\begin{theorem}[Validity of two-step SN method]
\label{thm: analytical rms method}
Suppose that there exist constants $0 < c_{1} <1/2$ and $C_{1} >0$ such that
\begin{equation}\label{eq: cond2}
\begin{split}
M_{n,3}^{3}  \log^{3/2} \left(\frac{p}{\beta_{n}\wedge(\alpha - 2\beta_n)}\right) \leq C_{1}n^{1/2-c_{1}},\\
\text{and } \ B_{n}^{2} \log^{2} (p/\beta_{n}) \leq C_{1} n^{1/2-c_{1}}.
\end{split}
\end{equation}
Then there exist positive constants $c,C$ depending only on $\alpha,c_{1}, C_{1}$ such that under $H_{0}$,
\begin{equation}\label{eq: thm 42 first bound}
\Pr(T>c^{SN,2S}(\alpha))\leq \alpha+Cn^{-c}.
\end{equation}
Moreover, this bound holds uniformly over all distributions $\mathcal L_X$ satisfying (\ref{eq: variance cond}) and (\ref{eq: cond2}). In addition, if all components of $X_1$ are independent, $\mu_j = 0$ for all $1\leq j\leq p$, $p = p_n\to\infty$, and $\beta_n\to 0$, then
\begin{equation}\label{eq: thm42 additional bound}
\Pr(T > c^{SN,2S}(\alpha)) \to 1 - e^{-\alpha}.
\end{equation}
\end{theorem}
\begin{remark}[Comparing conditions of one-step and two-step SN methods]\label{rem: prim conditions two-step SN}
Observe that the condition \eqref{eq: cond2} required for the validity of the two-step SN method in Theorem \ref{thm: analytical rms method} is stronger than the condition \eqref{eq: cond1} required for the validity of the one-step SN method in Theorem \ref{thm: analytical plugin method}. To see the meaning of \eqref{eq: cond2} under primitive conditions, suppose that all $X_{i j}$'s are Gaussian. Then all $Z_{i j}$'s are standard Gaussian, and so $B_n = (\Ep[\max_{1\leq j\leq p}Z_{1 j}^4])^{1/4} \leq C(\log p)^{1/2}$ for some constant $C>0$. Hence, given that $M_{n,3}\leq C$ in this case and $\beta_n < 1$, it follows that condition \eqref{eq: cond2} is implied by $\log^3(p/(\beta_n\wedge(\alpha - 2\beta_n))) \leq C_1 n^{1/2-c_1}$ (with a different constant $C_1$). Hence, if $c n^{-1/C} \leq \beta_n \leq \alpha/2 - c$, it follows that condition \eqref{eq: cond2} holds when $\log^6p/n \leq C_1 n^{-c_1}$ (with different constants $c_1$ and $C_1$).
\end{remark}

\subsection{Bootstrap methods}\label{sec: simulation methods}

In this section, we consider bootstrap methods for calculating critical values. Specifically, we consider Multiplier Bootstrap (MB) and Empirical (nonparametric, or Efron's) Bootstrap (EB) methods.
The methods studied in this section are computationally harder than those in the previous section but they lead to less conservative tests. In particular, we will show that when all the moment inequalities are binding (that is, $\mu_{j}=0$ for all $1 \leq j \leq p$), the asymptotic size of the tests based on these methods coincides with the nominal size.

\subsubsection{One-step method}
We first consider the one-step method.
Recall that, in order to make the test to have size $\alpha$,
it is enough to choose the critical value  as (a bound on) the $(1-\alpha)$-quantile of the distribution of
\[
\max_{1 \leq j \leq p} \sqrt{n}(\hat{\mu}_{j}-\mu_{j})/\hat{\sigma}_{j}.
\]
The SN method finds such a bound by using the union bound and the moderate deviation inequality for self-normalized sums.
However, the SN method may be conservative as it ignores correlation between the coordinates in $X_{i}$.

Alternatively, we consider here a Gaussian approximation. Observe first that under suitable regularity conditions,
\[
\max_{1 \leq j \leq p} \sqrt{n}(\hat{\mu}_{j}-\mu_{j})/\hat{\sigma}_{j} \approx \max_{1 \leq j \leq p} \sqrt{n}(\hat{\mu}_{j}-\mu_{j})/\sigma_{j} = \max_{1 \leq j \leq n} \sqrt{n} \En[ Z_{ij}],
\]
where $Z_{i}=(Z_{i1},\dots,Z_{ip})^{T}$ are defined in (\ref{eq: Z}).
{\em When $p$ is fixed}, the central limit theorem guarantees that as $n \to \infty$,
\[
\sqrt{n} \En [ Z_{i} ] \stackrel{d}{\to} Y, \ \text{with} \ Y=(Y_1,\dots,Y_p)^T \sim N(0,\Ep[ Z_{1}Z_{1}^{T}]),
\]
which, by the continuous mapping theorem, implies that
\[
\max_{1 \leq j \leq p} \sqrt{n} \En [ Z_{ij} ] \stackrel{d}{\to} \max_{1 \leq j \leq p} Y_{j}.
\]
Hence in this case it is enough to take the critical value as the  $(1-\alpha)$-quantile of the distribution of $\max_{1\leq j\leq p}Y_j$.

{\em When $p$ grows with $n$}, however, the concept of convergence in distribution does not apply, and different tools should be used to derive an appropriate critical value for the test. One possible approach is to use a Berry-Esseen theorem that provides a suitable non-asymptotic bound between the distributions of $\sqrt{n} \En [ Z_{i} ]$ and $Y$; see, for example, \cite{Gotze91} and \cite{Bentkus03}. However, such Berry-Esseen bounds require $p$ to be small in comparison with $n$ in order to guarantee that the distribution of $\sqrt{n} \En [ Z_{i} ]$ is  close to that of $Y$.
Another possible approach is to compare the distributions of $\max_{1\leq j\leq p}\sqrt{n}\En[Z_{ij}]$ and $\max_{1\leq j\leq p}Y_j$ directly, avoiding the comparison of distributions of the whole vectors $\sqrt{n} \En [Z_{i}]$ and $Y$.
Our recent work \citep{CCK12, CCK14} shows that, under mild regularity conditions,  the distribution of $\max_{1\leq j\leq p} \sqrt{n} \En[Z_{i j}] $ can be approximated by that of $\max_{1\leq j\leq p}Y_j$  in the sense of Kolmogorov distance {\em even when  $p$ is larger or much larger than $n$}.\footnote{The Kolmogorov distance between the distributions of two random variables $\xi$ and $\eta$ is defined by $\sup_{t \in \R} | \Pr(\xi \leq t) - \Pr(\eta \leq t)|$.} This result implies that we can still use the $(1-\alpha)$-quantile of the distribution of $\max_{1\leq j\leq p}Y_j$ even when $p$ grows with $n$ and is potentially much larger than $n$.\footnote{Some applications of this result can be found in \cite{Chetverikov2011,Chetverikov2012}, \cite{WassermanKolarRinaldo2013}, and \cite{Wasserman13}.}

Still,  the distribution of $\max_{1\leq j\leq p}Y_j$ is typically unknown because the covariance structure of $Y$ is unknown. Hence we will approximate the distribution of $\max_{1\leq j\leq p}Y_j$ by one of the following two bootstrap procedures:

\begin{algorithm}[Multiplier bootstrap]
\
\begin{enumerate}
\item[1.] Generate independent  standard normal random variables $\epsilon_1,\dots,\epsilon_n$  independent of the data $X_{1}^{n} = \{ X_{1},\dots,X_{n} \}$. \\
\item[2.] Construct the multiplier bootstrap test statistic
\begin{equation}
W^{MB}=\max_{1\leq j\leq p}\frac{\sqrt{n}\En[\epsilon_i(X_{ij}-\hat{\mu}_j)]}{\hat{\sigma}_j}. \label{eq: MB stat}
\end{equation}
\item[3.]
Calculate $c^{MB}(\alpha)$ as
\begin{equation}
c^{MB}(\alpha)=\text{conditional $(1-\alpha)$-quantile of $W^{MB}$ given $X_{1}^{n}$}. \label{MB critical value}
\end{equation}
\end{enumerate}
\end{algorithm}

\begin{algorithm}[Empirical bootstrap]
\
\begin{enumerate}
\item[1.] Generate a bootstrap sample $X_1^*,\dots,X_n^*$ as i.i.d. draws from the empirical distribution of $X_1^n=\{X_1,\dots,X_n\}$. \\
\item[2.] Construct the empirical bootstrap test statistic
\begin{equation}
W^{EB}=\max_{1\leq j\leq p}\frac{\sqrt{n}\En[X_{i j}^*-\hat{\mu}_j]}{\hat{\sigma}_j}. \label{eq: EB stat}
\end{equation}
\item[3.]
Calculate $c^{EB}(\alpha)$ as
\begin{equation}
c^{EB}(\alpha)=\text{conditional $(1-\alpha)$-quantile of $W^{EB}$ given $X_{1}^{n}$}. \label{EB critical value}
\end{equation}
\end{enumerate}
\end{algorithm}

We will call $c^{MB}(\alpha)$ and $c^{EB}(\alpha)$ the (one-step) Multiplier Bootstrap (MB) and Empirical Bootstrap (EB) critical values with size $\alpha$.
In practice conditional quantiles of $W^{MB}$ or $W^{EB}$ can be computed with any precision by using simulation.

Intuitively, it is expected that the multiplier bootstrap works well since conditional on the data $X_{1}^{n}$, the vector
\[
\left ( \frac{\sqrt{n}\En[\epsilon_i(X_{ij}-\hat{\mu}_j)]}{\hat{\sigma}_j} \right )_{1 \leq j \leq p}
\]
has the centered normal distribution with covariance matrix
 \begin{equation}\label{eq: empirical covariance matrix}
\En \left [ \frac{(X_{ij}-\hat{\mu}_j)}{\hat{\sigma}_{j}} \frac{(X_{ik}-\hat{\mu}_k)}{\hat{\sigma}_{k}} \right ],   \ 1 \leq j,k \leq p,
\end{equation}
which should be close to the covariance matrix of the vector $Y$. Indeed, by Theorem 2 in \cite{CCK13}, the primary factor for the bound on the Kolmogorov distance between the conditional distribution of $W$ and the distribution of $\max_{1 \leq j \leq p} Y_{j}$ is
\[
\max_{1 \leq j,k \leq p} \left | \En \left [ \frac{(X_{ij}-\hat{\mu}_j)}{\hat{\sigma}_{j}} \frac{(X_{ik}-\hat{\mu}_k)}{\hat{\sigma}_{k}} \right ] - \Ep[ Z_{1j} Z_{1k} ] \right |,
\]
which we show to be small under suitable conditions even when $p \gg n$.

In turn, the empirical bootstrap is expected to work well since conditional on the data $X_1^n$, the maximum of the random vector
$$
\left(\frac{\sqrt{n}\En[X_{i j}^*-\hat{\mu}_j]}{\hat{\sigma}_j}\right)_{1\leq j\leq p}
$$
can be well approximated in distibution by the maximum of a random vector with centered normal distribution with covariance matrix (\ref{eq: empirical covariance matrix}) even when $p\gg n$.

The following theorem formally establishes validity of the MB and EB critical values.

\begin{theorem}
[Validity of one-step MB and EB methods]
\label{thm: simulation plugin method}
Let $c^{B}(\alpha)$ stand either for $c^{MB}(\alpha)$ or $c^{EB}(\alpha)$.
Suppose that there exist constants $0 < c_{1} < 1/2$ and $C_{1}>0$ such that
\begin{equation}
(M_{n,3}^{3} \vee M_{n,4}^{2} \vee B_{n})^{2} \log^{7/2} (pn) \leq C_1n^{1/2-c_1}. \label{eq: cond3}
\end{equation}
Then there exist positive constants $c,C$ depending only on $c_{1},C_{1}$ such that under $H_{0}$,
\begin{equation}\label{eq: simulation plugin control}
\Pr(T>c^{B}(\alpha))\leq \alpha+Cn^{-c}.
\end{equation}
In addition, if $\mu_{j}=0$ for all $1 \leq j \leq p$, then
\begin{equation}\label{eq: simulation plugin control2}
|\Pr(T>c^{B}(\alpha))-\alpha|\leq Cn^{-c}.
\end{equation}
Moreover, both bounds hold uniformly over all distributions $\mathcal L_X$ satisfying (\ref{eq: variance cond}) and (\ref{eq: cond3}).
\end{theorem}
\begin{remark}[High dimension bootstrap CLT]
The result (\ref{eq: simulation plugin control2}) can be understood as a high dimensional bootstrap CLT for maxima of {\em studentized} sample averages. It shows that such maxima can be approximated either by multiplier or empirical bootstrap methods even if maxima are taken over (very) many sample averages. Moreover, the distributional approximation holds with polynomially (in $n$) small error. This result complements a high dimensional bootstrap CLT for {\em non-studentized} sample averages derived in \cite{CCK12} and \cite{CCK14}, and may be of interest in many other settings, well beyond the problem of testing many moment inequalities.
\end{remark}

\begin{remark}[Comparison with \citealp{White00}]
\cite{White00} is relevant to our one-step MB/EB methods in
the sense that \cite{White00} considers a max-type statistic for an inequality testing problem and
applies bootstrap to calibrate critical values. However, \cite{White00} does not consider
Studentization, and more importantly 1) does not allow the number of inequalities
increasing with the sample size, and 2) does not consider inequality selection so that
his test would be conservative (see the next subsection on our two-step MB/EB methods). In fact, \cite{White00} acknowledges the importance
of extending his analysis to the case where the number of inequalities increases with
the sample size, and explicitly states that ``it is natural to consider what happens when
$l$ grows with $T$'' [$l$ is the number of inequalities tested and $T$ is the sample size] but ``
rigorous treatment for our context is beyond our present scope'' \citep[][p.1110-1111]{White00}. Our results on the one-step MB/EB methods address this important question
in a far more general setting where the number of inequalities can be much larger than
the sample size. In addition, our results provided finite sample error bounds that hold uniformly
over a wide class of underlying distributions, while \cite{White00} only derives pointwise
asymptotic results on validity of the test.
\end{remark}

\begin{remark}[Other bootstrap procedures]
There exist many different bootstrap procedures in the literature, each with its own advantages and disadvantages. In this paper, we focused on multiplier and empirical bootstraps, and we leave analysis of more general exchangeably weighted bootstraps, which include many existing bootstrap procedures as a special case (see, for example, \cite{PW93}), in the high dimensional setting for future work.
\end{remark}

\begin{remark}[Comparing conditions of two-step SN method and one-step MB/EB methods]
Observe that the condition \eqref{eq: cond3} required for the validity of the one-step MB/EB methods in Theorem \ref{thm: simulation plugin method} is stronger than the condition \eqref{eq: cond2} required for the validity of the two-step SN method in Theorem \ref{thm: analytical rms method}. To see the meaning of \eqref{eq: cond3} under primitive conditions, suppose that all $X_{i j}$'s are Gaussian. As in Comment \ref{rem: prim conditions two-step SN}, it then follows that $M_{n,3}\leq C$ and $B_n \leq C(\log p)^{1/2}$ for some constant $C$ in this case. Moreover, it is easy to see that $M_{n,4}\leq C$ as well. Therefore, condition \ref{rem: prim conditions two-step SN} holds if $(\log^9 p)/n\leq C_1 n^{-c_1}$ (with possibly different constants $c_1$ and $C_1$).
\end{remark}

\subsubsection{Two-step methods}
We now consider to combine bootstrap methods with inequality selection.
To describe these procedures, let $0 < \beta_{n} < \alpha/2$ be some constant. As in the previous section, we allow $\beta_{n}$  to depend on $n$. Let $c^{MB}(\beta_n)$ and $c^{EB}(\beta_n)$ be one-step MB and EB critical values with size $\beta_{n}$, respectively. Define the sets $\hat{J}_{MB}$ and $\hat{J}_{EB}$ by
\[
\hat{J}_{B}:=\{j\in\{1,\dots,p\}:\sqrt{n}\hat{\mu}_j/\hat{\sigma}_j > -2 c^{B}(\beta_{n})\}
\]
where $B$ stands either for $MB$ or $EB$.
Then the two-step MB and EB critical values, $c^{MB,2S}(\alpha)$ and $c^{EB,2S}(\alpha)$, are defined by the following procedures:

\begin{algorithm}[Multiplier bootstrap with inequality selection]
\
\begin{enumerate}
\item[1.] Generate independent   standard normal  random variables $\epsilon_1,\dots,\epsilon_n$  independent of the data $X_{1}^{n}$.
\item[2.] Construct the multiplier bootstrap test statistic
\[
W_{\hat{J}_{MB}}=
\begin{cases}
\max_{j\in\hat{J}_{MB}}\frac{\sqrt{n}\En[\epsilon_i(X_{ij}-\hat{\mu}_j)]}{\hat{\sigma}_j}, &\text{if }\hat{J}_{MB}\text{ is not empty,}\\
0 &\text{if }\hat{J}_{MB}\text{ is empty}.
\end{cases}
\]
\item[3.]
Calculate $c^{MB,2S}(\alpha)$ as
\begin{equation}
c^{MB,2S}(\alpha)
=\text{conditional $(1-\alpha + 2\beta_n)$-quantile of $W_{\hat{J}_{MB}}$ given $X_{1}^{n}$}. \label{MB-MS critical value}
\end{equation}
\end{enumerate}
\end{algorithm}

\begin{algorithm}[Empirical bootstrap with inequality selection]
\
\begin{enumerate}
\item[1.] Generate a bootstrap sample $X_1^*,\dots,X_n^*$ as i.i.d. draws from the empirical distribution of $X_1^n=\{X_1,\dots,X_n\}$.
\item[2.] Construct the empirical bootstrap test statistic
\[
W_{\hat{J}_{EB}}=
\begin{cases}
\max_{j\in\hat{J}_{EB}}\frac{\sqrt{n}\En[X_{i j}^*-\hat{\mu}_j]}{\hat{\sigma}_j}, &\text{if }\hat{J}_{EB}\text{ is not empty,}\\
0 &\text{if }\hat{J}_{EB}\text{ is empty}.
\end{cases}
\]
\item[3.]
Calculate $c^{EB,2S}(\alpha)$ as
\begin{equation}
c^{EB,2S}(\alpha)
=\text{conditional $(1-\alpha + 2\beta_n)$-quantile of $W_{\hat{J}_{EB}}$ given $X_{1}^{n}$}. \label{EB-MS critical value}
\end{equation}
\end{enumerate}
\end{algorithm}

The following theorem establishes validity of the two-step MB and EB critical values.

\begin{theorem}
[Validity of two-step MB and EB methods]
\label{thm: simulation MS method}
Let $c^{B,2S}(\alpha)$ stand either for $c^{MB,2S}(\alpha)$ or $c^{EB,2S}(\alpha)$.
Suppose that the assumption of Theorem \ref{thm: simulation plugin method} is satisfied.  Moreover, suppose that  $\log (1/\beta_{n}) \leq C_{1}\log n$. 
Then there exist positive constants $c,C$ depending only on $c_{1},C_{1}$ such that under $H_{0}$, 
$$
\Pr(T>c^{B,2S}(\alpha))\leq \alpha+Cn^{-c}.
$$ 
In addition, if $\mu_{j}=0$ for all $1 \leq j \leq p$, then 
$$
\Pr(T>c^{B,2S}(\alpha)) \geq  \alpha - 3 \beta_{n} - Cn^{-c},
$$
so that under an extra assumption that $\beta_n\leq C_1n^{-c_1}$, then 
$$
|\Pr(T>c^{B,2S}(\alpha))-\alpha|\leq Cn^{-c}.
$$
Moreover, all these bounds hold uniformly over all distributions $\mathcal L_X$ satisfying (\ref{eq: variance cond}) and (\ref{eq: cond3}). 
\end{theorem}
\begin{remark}
The selection procedure used in the theorem above is most closely related to those in \cite{ChernozhukovLeeRosen2013} and in \cite{Chetverikov2011}. Other selection procedures were suggested in the literature in the framework when $p$ is fixed. Specifically, \cite{RomanoShaikhWolf2012} derived an inequality selection method based on the construction of rectangular confidence sets for the vector $(\mu_1,\dots,\mu_p)^T$. To extend their method to high dimensional setting considered here, note that by (\ref{eq: simulation plugin control2}), we have that $\mu_j\leq \hat{\mu}_j+\hat{\sigma}_jc^{MB}(\beta_n)/\sqrt{n}$ for all $1 \leq j \leq p$ with probability $1-\beta_n$ asymptotically. Therefore, we can replace (\ref{eq: under null}) with the following probabilistic inequality: under $H_0$,
\[
\Pr\left(T\leq\max_{1\leq j\leq p}\frac{\sqrt{n}(\hat{\mu}_j-\mu_j+\tilde{\mu}_j)}{\hat{\sigma}_j}\right)\geq 1-\beta_n+o(1),
\]
where
\[
\tilde{\mu}_j=\min\left(\hat{\mu}_j+\hat{\sigma}_j c^{MB}(\beta_n)/\sqrt{n},0\right).
\]
This suggests that we could obtain a critical value based on the distribution of the bootstrap test statistic
\[
\hat{W}=\max_{1\leq j\leq p}\frac{\sqrt{n}\En[\epsilon_i(X_{ij}-\hat{\mu}_j)]+\sqrt{n}\tilde{\mu}_j}{\hat{\sigma}_j}.
\]
For brevity, however, we leave analysis of this critical value for future research.
\qed
\end{remark}

\subsection{Hybrid methods}

We have considered the one-step SN, MB, and EB methods and their two-step variants. In fact, we can also consider ``hybrids'' of these methods. For example, we can use the SN method for inequality selection, and apply the MB or EB method for the selected inequalities, which is a computationally more tractable alternative to the two-step MB and EB methods. For convenience of terminology, we will call it the Hybrid (HB) method. To formally define the method, let $0 < \beta_{n} < \alpha/2$ be some constants, and recall the set $\hat{J}_{SN} \subset \{ 1,\dots, p \}$ defined in (\ref{eq: SN set}). Suppose we want to use the MB method on the second step. Then the hybrid MB critical value, $c^{MB,H}(\alpha)$ is defined by the following procedure:

\begin{algorithm}[Multiplier Bootstrap Hybrid method]
\
\begin{enumerate}
\item[1.] Generate independent  standard normal  random variables $\epsilon_1,\dots,\epsilon_n$  independent of the data $X_{1}^{n}$.
\item[2.] Construct the bootstrap test statistic
\[
W_{\hat{J}_{SN}}=
\begin{cases}
\max_{j\in\hat{J}_{SN}}\frac{\sqrt{n}\En[\epsilon_i(X_{ij}-\hat{\mu}_j)]}{\hat{\sigma}_j}, &\text{if }\hat{J}_{SN}\text{ is not empty,}\\
0 &\text{if }\hat{J}_{SN}\text{ is empty}.
\end{cases}
\]
\item[3.]
Calculate $c^{MB,H}(\alpha)$ as
\begin{equation}
c^{MB,H}(\alpha)
=\text{conditional $(1-\alpha + 2\beta_n)$-quantile of $W_{\hat{J}_{SN}}$ given $X_{1}^{n}$}. \label{HB critical value}
\end{equation}
\end{enumerate}
\end{algorithm}
A similar algorithm can be defined for the EB method on the second step, which leads to the hybrid EB critical value $c^{EB,H}(\alpha)$.
The following theorem establishes validity of these critical values.
\begin{theorem}[Validity of hybrid two-step methods]
\label{thm: HB method}
Let $c^{B,H}(\alpha)$ stand either for $c^{MB,H}(\alpha)$ or $c^{EB,H}(\alpha)$. Suppose that there exist constants $0 < c_{1} < 1/2$ and $C_{1} > 0$ such that (\ref{eq: cond3}) is verified. Moreover, suppose that $\log (1/\beta_{n}) \leq C_{1} \log n$. Then all the conclusions of Theorem \ref{thm: simulation MS method} hold with $c^{B,MS}(\alpha)$ replaced by $c^{B,H}(\alpha)$.
\end{theorem}

\subsection{Three-step method}
\label{sec: three-step method}
In empirical studies based on moment inequalities, one typically has inequalities of the form
\begin{equation}\label{eq: alternative setting}
\Ep[g_j(\xi,\theta)]\leq 0\quad \text{ for all }j=1,\dots,p,
\end{equation}
where $\xi$ is a vector of random variables from a distribution denoted by $\mathcal L_{\xi}$, $\theta = (\theta_{1},\dots,\theta_{r})^{T}$ is a vector of parameters in $\mathbb{R}^r$, and $g_1,\dots,g_p$ a set of (known) functions. In these studies, inequalities (\ref{eq: null hypothesis})-(\ref{eq: alternative hypothesis}) arise when one tests the null hypothesis $\theta=\theta_0$ against the alternative $\theta\neq \theta_0$ on the i.i.d. data $\xi_1,\dots,\xi_n$ by setting $X_{i j}:=g_j(\xi_i,\theta_0)$ and $\mu_j:=\Ep[X_{1 j}]$. So far in this section, we showed how to increase power of such tests by employing inequality selection procedures that allow the researcher to drop uninformative inequalities, that is inequalities $j$ with $\mu_j<0$ if $\mu_j$ is not too close to 0. In this subsection, we seek to combine these selection procedures with another selection procedure that is suitable for the model (\ref{eq: alternative setting}) and that can substantially increase local power of the test of $\theta=\theta_0$ by dropping {\em weakly informative} inequalities, that is inequalities $j$ with the function $\theta\mapsto\Ep[g_j(\xi,\theta)]$ being flat or nearly flat around $\theta=\theta_0$. When the tested value $\theta_0$ is close to some $\theta$ satisfying (\ref{eq: alternative setting}), such inequalities can only provide a weak signal of violation of the hypothesis $\theta=\theta_0$ in the sense that they have $\mu_j\approx 0$, and so it is useful to drop them. For brevity of the paper, we only consider weakly informative inequality selection based on the MB and EB methods and note that similar results can be obtained for the self-normalized method. Also, we only consider the case when the functions $\theta\mapsto g_j(\xi,\theta)$ are almost surely continuously differentiable, and leave the extension to non-differentiable functions to future work.

We start with preparing necessary notation. Let $\xi_1,\dots,\xi_n$ be a sample of observations from the distribution of $\xi$. Suppose that we are interested in testing the null hypothesis
\begin{align*}
&H_0: \Ep[g_j(\xi,\theta_0)]\leq 0\quad \text{ for all }j=1,\dots,p,
\intertext{against the alternative}
&H_1: \Ep[g_j(\xi,\theta_0)]> 0\quad \text{ for some }j=1,\dots,p,
\end{align*}
where $\theta_0$ is some value of the parameter $\theta$. Define
\begin{align*}
m_j(\xi,\theta)&:=(m_{j 1}(\xi,\theta),\dots,m_{j r}(\xi,\theta))^T\\
&:=(\partial g_j(\xi,\theta)/\partial \theta_{1},\dots,\partial g_{j}(\xi,\theta)/\partial \theta_{r})^{T}
\end{align*}
Further, let $X_{i j}:=g_j(\xi_i,\theta_0)$, $\mu_j:=\Ep[X_{1 j}]$, $\sigma_j:=(\text{Var}(X_{1 j}))^{1/2}$, $V_{i j l}:=m_{j l}(\xi_i,\theta_0)$, $\mu_{j l}^V:=\Ep[V_{1 j l}]$, and $\sigma_{j l}^V:=(\text{Var}(V_{1 j l}))^{1/2}$. 
We assume that
\begin{align}
&\Ep[X_{1j}^2]<\infty, \, \sigma_j>0, \, j=1,\dots,p,\label{eq: conditions for gradient 1} \\
&\Ep[V_{1 j l}^2]<\infty, \, \sigma_{j l}^V>0, \, j=1,\dots,p, \, l=1,\dots, r.\label{eq: conditions for gradient 2}
\end{align}
In addition, let
\[
\hat{\mu}_j=\En[X_{i j}] \text{ and }\hat{\sigma}_j=\left(\En[(X_{i j}-\hat{\mu}_j)^2]\right)^{1/2}
\]
be estimators of $\mu_j$ and $\sigma_j$, respectively, and let
\[
\hat{\mu}_{j l}^V=\En[V_{i j l}] \text{ and }\hat{\sigma}_{j l}^V=\left(\En[(V_{i j l}-\hat{\mu}_{j l}^V)^2]\right)^{1/2}
\]
be estimators of $\mu_{j l}^V$ and $\sigma_{j l}^V$, respectively.

Weakly informative inequality selection that we derive is based on the bootstrap methods similar to those described in Section \ref{sec: critical value}:
\begin{algorithm}[Multiplier bootstrap for gradient statistic]
\
\begin{enumerate}
\item[1.] Generate independent  standard normal random variables $\epsilon_1,\dots,\epsilon_n$  independent of the data $\xi_{1}^{n} = \{ \xi_{1},\dots,\xi_{n} \}$. \\
\item[2.] Construct the multiplier bootstrap gradient statistic
\begin{equation}
W_{MB}^V=\max_{j,l}\frac{\sqrt{n}|\En[\epsilon_i(V_{i j l}-\hat{\mu}_{j l}^V)]|}{\hat{\sigma}_{j l}^V}. \label{eq: MB gradient stat}
\end{equation}
\item[3.]
For $\gamma\in(0,1)$, calculate $c^{MB,V}(\gamma)$ as
\begin{equation}
c^{MB,V}(\gamma)=\text{conditional $(1-\gamma)$-quantile of $W_{MB}^V$ given $\xi_{1}^{n}$}. \label{MB gradient critical value}
\end{equation}
\end{enumerate}
\end{algorithm}
\begin{algorithm}[Empirical bootstrap for gradient statistic]
\
\begin{enumerate}
\item[1.] Generate a bootstrap sample $V_1^*,\dots,V_n^*$ as i.i.d. draws from the empirical distribution of $V_1^n=\{V_1,\dots,V_n\}$. \\
\item[2.] Construct the empirical bootstrap gradient statistic
\begin{equation}
W_{EB}^V=\max_{j,l}\frac{\sqrt{n}|\En[V_{i j l}^*-\hat{\mu}_{j l}^V]|}{\hat{\sigma}_{j l}^V}. \label{eq: EB gradient stat}
\end{equation}
\item[3.]
For $\gamma\in(0,1)$, calculate $c^{EB,V}(\gamma)$ as
\begin{equation}
c^{EB,V}(\gamma)=\text{conditional $(1-\gamma)$-quantile of $W_{EB}^V$ given $\xi_{1}^{n}$}. \label{EB gradient critical value}
\end{equation}
\end{enumerate}
\end{algorithm}
For some strictly positive constants $c_2$ and $C_2$, let $\varphi_n$ be a sequence of constants satisfying $\varphi_n\log n\geq c_2$, and let $\beta_n$ be a sequence of constants satisfying $0<\beta_n<\alpha/4$ and $\log(1/(\beta_n-\varphi_n))\leq C_2\log n$ where $\alpha$ is the nominal level of the test. Define three estimated sets of inequalities:
\begin{align*}
&\hat{J}_B:=\left\{j\in\{1,\dots,p\}:\sqrt{n}\hat{\mu}_j/\hat{\sigma}_j>-2c^B(\beta_n)\right\},\\
&\hat{J}_{B}':=\left\{j\in\{1,\dots,p\}:\sqrt{n}|\hat{\mu}_{j l}^V/\hat{\sigma}_{j l}^V|>3c^{B,V}(\beta_n-\varphi_n)\text{ for some }l=1,\dots,r\right\},\\
&\hat{J}_{B}'':=\left\{j\in\{1,\dots,p\}:\sqrt{n}|\hat{\mu}_{j l}^V/\hat{\sigma}_{j l}^V|> c^{B,V}(\beta_n+\varphi_n)\text{ for some }l=1,\dots,r\right\},
\end{align*}
where $B$ stands either for $MB$ or $EB$. 

Importantly, the weakly informative inequality selection procedure that we derive requires that both the test statistic and the critical value depend on the estimated sets of inequalities. Let $T^{B}$ and $c^{B,3S}(\alpha)$ denote the test statistic and the critical value for $B=MB$ or $EB$ depending on which bootstrap procedure is used. If the set $\hat{J}_{B}'$ is empty, set the test statistic $T^B=0$ and the critical value $c^{B,3S}(\alpha)=0$. Otherwise, define the test statistic
\[
T^B=\max_{j\in \hat{J}_{B}'}\frac{\sqrt{n}\hat{\mu}_j}{\hat{\sigma}_j},
\]
and define the three-step MB/EB critical values, $c^{B,3S}(\alpha)$ for the test by the same bootstrap procedures as those for $c^{B,2S}(\alpha)$ with $\hat{J}_{B}$ replaced by $\hat{J}_B\cap\hat{J}_{B}''$, and also $2\beta_n$ replaced by $4\beta_n$:
$$
c^{B,2S}(\alpha)
=\text{conditional $(1-\alpha + 4\beta_n)$-quantile of $W_{\hat{J}_{B}\cap\hat J_{B}''}$ given $X_{1}^{n}$},
$$
where $W_{\hat J_B\cap \hat J_B''}$ is either the multiplier or the bootstrap test statistic depending on whether $B = MB$ or $EB$. The test rejects $H_0$ if $T^B>c^{B,3S}(\alpha)$.\footnote{In the definition of the bootstrap test statistic $W_{\hat J_B\cap \hat J_B''}$, the set $\hat J_B''$ is different from $\hat J_B'$, which is used in the definition of the test statistic $T^B$. This is because our proof techniques do not allow us to show the validity of the critical values based on $W_{\hat J_B\cap \hat J_B'}$ since $\hat J_B'$ is random. Instead, our approach consists of finding non-random set $J$ such that with large probability, $\hat J_B' \subset J \subset \hat J_B''$, so that $T^B = \max_{j\in\hat J_B'}\sqrt n\mu_j/\hat\sigma_j\leq \max_{j\in J}\sqrt n\hat\mu_j/\hat\sigma_j$ and $W_{\hat J_B\cap \hat J_B''}\geq W_{\hat J_B\cap J}$ and then showing validity of using $W_{\hat J_B\cap J}$ to approximate the distribution of $\max_{j\in J}\sqrt n\hat\mu_j/\hat\sigma_j$.}


To state the main result of this section, we need the following additional notation. Let
\[
Z_{i j l}^V:=(V_{i j l}-\mu_{j l}^V)/\sigma_{j l}^V.
\]
Observe that $\Ep[Z_{i j l}^V]=0$ and $\Ep[(Z_{i j l}^V)^2]=1$. Let
\[
M_{n,k}^V:=\max_{j,l}\left(\Ep[|Z_{1 j l}^V|^k]\right)^{1/k},\, k=3,4,\, B_n^V:=\left(\Ep\Big[\max_{j,l}(Z_{1 j l}^V)^4\Big]\right)^{1/4}.
\]
We have the following theorem:
\begin{theorem}[Validity of three-step MB and EB methods]\label{thm: weak inequalities}
Let $T^B$ and $c^{B,3S}(\alpha)$ stand either for $T^{MB}$ and $c^{MB,3S}(\alpha)$ or for $T^{EB}$ and $c^{EB,3S}(\alpha)$.
Suppose that there exist constants $0<c_1<1/2$ and $C_1>0$ such that
\begin{equation}\label{eq: three step condition 1}
\left(M_{n,3}^{3} \vee M_{n,4}^{2} \vee B_{n}\right)^{2} \log^{7/2} (p n) \leq C_1n^{1/2-c_1}
\end{equation}
and
\begin{equation}\label{eq: three step condition 2}
\left((M_{n,3}^V)^3\vee (M_{n,4}^V)^2\vee B_n^V\right)^{2} \log^{7/2} (p r n) \leq C_1n^{1/2-c_1}.
\end{equation}
Moreover, suppose that $\log(1/(\beta_n-\varphi_n))\leq C_2\log n$ and $\varphi_n\log n\geq c_2$ for some constants $c_2,C_2>0$. Then there exist positive constants $c,C$ depending only on $c_1$, $C_1$, $c_2$, and $C_2$ such that under $H_0$, 
$$
\Pr(T^B>c^{B,3S}(\alpha))\leq \alpha+C n^{-c}.
$$
In addition, the bound holds uniformly over all distributions $\mathcal L_{\xi}$ satisfying (\ref{eq: conditions for gradient 1}), (\ref{eq: conditions for gradient 2}), (\ref{eq: three step condition 1}), and (\ref{eq: three step condition 2}).
\end{theorem}
\begin{remark}[On the choice of $\varphi_n$]
Inspecting the proof of the theorem shows that the result of the theorem remains valid if we replace condition $\varphi_n \log n\geq c_2$ by a weaker condition $\varphi_n\geq C n^{-c}$ for some constants $c,C$ that can be chosen to depend only on $c_1,C_1$. In practice, however, it is difficult to track the dependence of $c,C$ on $c_1,C_1$. Therefore, in the main text we state the result with the condition $\varphi_n \log n\geq c_2$; in simulations reported in Section \ref{sec: monte carlo}, we set $\varphi_n = \beta_n/2$.
\end{remark}



\section{Power}\label{sec: power}
In this section, we discuss power properties of our tests. Consider the same general setup described in the Introduction and assume that (\ref{eq: variance cond}) holds. Let the test statistic $T$ be defined by (\ref{eq: test statistic}). Pick any $\alpha \in (0,1/2)$ and consider the test of the form
\[
T > \hat{c}(\alpha) \Rightarrow \text{reject $H_{0}$},
\]
where $\hat{c}(\alpha)$ is equal to $c^{SN}(\alpha)$, $c^{SN,2S}(\alpha)$, $c^{MB}(\alpha)$, $c^{MB,2S}(\alpha)$, $c^{EB}(\alpha)$, $c^{EB,2S}(\alpha)$, $c^{MB,H}(\alpha)$, or $c^{EB,H}(\alpha)$. We have the following result on the rate of uniform consistency of this test:

\begin{theorem}[Rate of uniform consistency]
\label{thm: rate of uniform consistency}
Suppose there exist constants $0 < c_{1} < 1/2$ and $C_{1} > 0$ such that
\begin{equation}
M_{n,4}^2 \log^{1/2} p \leq C_{1}n^{1/2-c_{1}} \ \text{and} \ \log^{3/2}p \leq C_{1} n. \label{eq: minimax cond}
\end{equation}
In addition, suppose that $\inf_{n \geq 1}(\alpha - 2 \beta_{n}) \geq c_1\alpha$ whenever inequality selection is used. Then there exist constants $c,C>0$ depending only on $\alpha,c_{1},C_{1}$ such that for every $\epsilon \in (0,1)$, whenever
\[
\max_{1 \leq j \leq p} (\mu_{j}/\sigma_{j}) \geq (1+\epsilon + C \log^{-1/2}p) \sqrt{\frac{2 \log (p/\alpha)}{n}},
\]
we have
\[
\Pr (T>\hat{c}(\alpha)) \geq 1-\frac{C}{\epsilon^{2} \log (p/\alpha)} - Cn^{-c}.
\]
Therefore when $p = p_{n} \to \infty$, for any sequence $\epsilon_n$ satisfying $\epsilon_n\to 0$ and $\epsilon_n\sqrt{\log p_n}\to \infty$, as $n \to \infty$, we have (with keeping $\alpha$ fixed)
\begin{equation}
\inf_{\mu\in\mathcal B_n}\Pr_{\mu}(T>\hat{c}(\alpha)) \geq 1 - o(1), \label{eq: asymptotic minimaxity}
\end{equation}
where 
$$
\mathcal B_n = \Big\{\mu = (\mu_1,\dots,\mu_p): \max_{1\leq j\leq p}(\mu_j/\sigma_j)\geq \overline{r}_n = (1+\epsilon_n)\sqrt{2(\log p_n)/n}\Big\}
$$ 
and $\Pr_{\mu}$ denotes the probability measure for the distribution $\mathcal L_X$ having mean $\mu$. Moreover, the above asymptotic result (\ref{eq: asymptotic minimaxity}) holds uniformly with respect to any sequence of distributions $\mathcal L_X$ satisfying (\ref{eq: variance cond}) and (\ref{eq: minimax cond}).
\end{theorem}

\begin{remark}[Discussion of power properties]
This theorem shows that our tests are uniformly consistent against all alternatives excluding those in a small neighborhood of alternatives that are too close to the null. As long as $p = p_n\to\infty$ as $n\to\infty$, the size of this neighborhood is shrinking at a fast rate $\sqrt{(\log p_n)/n}$. This is a fast rate because even when $p$ is fixed, no test can be uniformly consistent against alternatives whose distance from the null converges to zero faster than $\sqrt{1/n}$. In fact, as we show in a working version of the paper,\footnote{arXiv:1312.7614v4.} when $p = p_n\to\infty$, no test can be uniformly consistent against alternatives whose distance from the null converges to zero faster than $\sqrt{(\log p_n)/n}$, and our tests are minimax optimal.  Here, $\sqrt{\log p_n}$ is a small factor representing the cost we have to pay for dealing with a large number of inequalities. 

Further, the theorem indicates that all of our tests have a fast rate of uniform consistency but it does not reveal that the bootstrap tests have better power properties than those of the SN tests. To explain, suppose for example that all inequalities are the same, that is, $X_{1 j_1} = X_{1 j_2}$ for all $j_1,j_2 = 1,\dots,p$ almost surely. In addition, suppose for concreteness that $\sigma = \sigma_1 = \dots = \sigma_p = 1$. Moreover, suppose that $\mu = \mu_1 = \dots = \mu_p$ is strictly positive but converges to zero as $n\to\infty$, that is, $\mu = \mu^n\downarrow 0$. Then the test statistic $T$ is asymptotically equal to a $N(\sqrt n \mu^n,1)$ random variable and, say, both one-step bootstrap critical values converge in probability to $z_{\alpha}$, the $(1 - \alpha)$ quantile of the $N(0,1)$ distribution. Therefore, the bootstrap tests are consistent against all alternatives such that $\sqrt n \mu^n \to\infty$ as $n\to\infty$. On the other hand, the one-step SN critical value is of order $\sqrt{\log p_n}$, as explained in Section \ref{sec: critical value}, and the one-step SN test is only consistent against alternatives such that $\sqrt n\mu^n/\sqrt{\log p_n} \to\infty$. A similar discussion applies to the two-step tests. This explains the difference in power between the SN and the bootstrap tests.
\end{remark}

\begin{remark}[Comparison with methods for conditional moment inequalities]
As discussed in the Introduction, our methods can also be applied when dealing with a large number of (unconditional) moment inequalities that arise from a small number of conditional moment inequalities. Here we explain how our methods compare with those developed specifically for testing conditional moment inequalities. To fix ideas, suppose that we have one conditional moment inequality,
\begin{equation}\label{eq: conditional moment inequality}
\Ep[m(Y,Z) | Z] \leq 0,
\end{equation}
where $Y$ and $Z$ are random vectors and $m$ is a known function. To transform this inequality into unconditional ones, let $w_{z,h}(Z)\geq 0$ be a positive weighting function indexed by the location point $z\in\mathcal Z_n$ and the bandwidth value $h\in\mathcal H_n$, where both $\mathcal Z_n$ and $\mathcal H_n$ are some large but finite sets. Then it follows from \eqref{eq: conditional moment inequality} that
$$
\Ep[m(Y,Z) w_{z, h}(Z)] \leq 0,\quad \text{for all $z\in\mathcal Z_n$ and $h\in\mathcal H_n$}.
$$
If $(Y_i,Z_i)$, $i=1,\dots,n$, is a random sample from the distribution of the pair $(Y,Z)$, our approach would be to consider the test statistic
$$
T = \max_{z\in\mathcal Z_n; h\in\mathcal H_n} \frac{n^{-1/2}\sum_{i=1}^n m(Y_i,Z_i)w_{z,h}(Z_i)}{\hat V_{z,h}^{1/2}},
$$
where $\hat V_{z,h}$ is an estimator of $V_{z,h}$, the variance of $m(Y,Z)w_{z,h}(Z)$. This is the test statistic used in \cite{ArmstrongChan2012}, up to a minor modification that they use infinite sets $\mathcal Z_n$ and $\mathcal H_n$. Since they couple the test statistic $T$ with the $(1-\alpha)$ quantile of the asymptotic distribution of $T$ when $\Ep[m(Y,Z)|Z] = 0$ almost surely, it follows that the power of their test essentially coincides with that of our one-step bootstrap tests, which can be improved by using our two-step and three-step bootstrap tests.

The approach in \cite{Chetverikov2011}, on the other hand, would be to consider the test statistic
$$
T' = \max_{z\in\mathcal Z_n, h\in\mathcal H_n}\frac{n^{-1/2}\sum_{i=1}^n m(Y_i,Z_i)w_{z,h}(Z_i)}{\hat V_{z,h,c}^{1/2}},
$$
where $\hat V_{z,h,c}$ is an estimator of $V_{z,h,c}$, the variance of $\varepsilon w_{z,h}(Z)$, where $\varepsilon = m(Y,Z) - \Ep[m(Y,Z)|Z]$. Since
\begin{align*}
V_{z,h}
&=\Ep[m(Y,Z)^2w_{z,h}(Z)^2] - \Ep[m(Y,Z)w_{z,h}(Z)]^2\\
&=\Ep[(\Ep[m(Y,Z)|Z] + \varepsilon)^2 w_{z,h}(Z)^2] - \Ep[\Ep[m(Y,Z)|Z]w_{z,h}(Z)]^2\\
&=\text{Var}(\Ep[m(Y,Z)|Z]w_{z,h}(Z)) + V_{z,h,c} \geq V_{z,h,c},
\end{align*}
the same alternatives will lead to larger values of $T'$ than of $T$. It is therefore expected that the tests in \cite{Chetverikov2011} would typically have better power properties than those of the tests developed in our paper.\footnote{The precise comparison here is difficult. Indeed, consider for example the one-step bootstrap critical values developed here and in \cite{Chetverikov2011}. In both cases, the critical values are asymptotically equal to the $(1-\alpha)$ quantile of the maxima of $N(0,1)$ random variables, and are expected to be similar. On the other hand, the correlation structure of the $N(0,1)$ random variables in our paper and in \cite{Chetverikov2011} are different, and so it may be possible that our tests sometimes perform better than those in \cite{Chetverikov2011}.}

Further, it is argued in \cite{ArmstrongChan2012} that their test typically has better power properties than those of the test in \cite{AndrewsShi2013}, and so, given that our methods perform at least as good as the Armstrong-Chan test, we expect that our methods also should often have better power than those in \cite{AndrewsShi2013}, although neither approach dominates the other one. Moreover, it is important to emphasize that the Andrews-Shi test requires somewhat weaker regularity (in particular, moment) conditions than those used in our paper. Further comparisons of different methods, including those in \cite{ChernozhukovLeeRosen2013} and in \cite{LeeSongWhang13, LeeSongWhang13b} can be found in \cite{Chetverikov2011}.


To conclude this comparison, we emphasize that our methods are meant to complement those in the literature on testing conditional moment inequalities since our methods can be used to deal with a large number of (unconditional) moment inequalities that do not arise from the small number of conditional moment inequalities.
\end{remark}

\section{Monte Carlo Experiments}\label{sec: monte carlo}
In this section, we provide results of a Monte Carlo simulation study. The simulation study consists of three parts. The first part demonstrates that the methods developed in this paper have good size control and power properties and also demonstrates power advantages of using bootstrap and multi-step procedures over self-normalized and one-step procedures in a broad variety of abstract settings. These abstract settings are useful because they allow us to vary the key parameters of the data-generating process in a straightforward fashion and see how the performance of our methods depend on these parameters. Importantly, this part of the simulation study shows that the size control is achieved even though we use setups with a large number of moment inequalities. The second part sheds some light on the choice of the tuning parameters for our two- and three-step methods. The third part applies our methods in an example based on the market structure model of \cite{CilibertoTamer2009}.

\subsection{Size and power in abstract settings}
Throughout all the experiments in this subsection, we consider i.i.d. samples of size $n=400$. Depending on the experiment, the number of moment inequalities is $p=200$, $500$, or $1000$. Thus, we consider models where the number of moment inequalities $p$ is comparable, larger, or substantially larger than the sample size $n$. 

All the experiments are based on the following data-generating process:
$$
X_{i j} = \theta(1\{j\leq \gamma_1 p\} + \varepsilon_{i j}) - b1\{\gamma_2 p < j \leq p\} + \varepsilon_{i j}.
$$
Here, $\theta$ is a scalar parameter of interest, $(\gamma_1,\gamma_2,b)$ is a triple of additional parameters governing the data-generating process, and $\varepsilon_i = (\varepsilon_{i1},\dots,\varepsilon_{i p})^T$, $i=1,\dots,n$, is a sequence of i.i.d. random vectors in $\mathbb R^p$. We always set $\gamma_1 = 5\%$ and $\gamma_2 = 10\%$ but we vary $b$ and the distribution of $\varepsilon_i$'s depending on the experimental design.

We consider 8 different experimental designs. In all designs, we assume that for all $i=1,\dots,n$, we have $\varepsilon_i = A^T\epsilon_i$, where the vector $\epsilon_i = (\epsilon_{1 i},\dots,\epsilon_{i p})^T$ consists of i.i.d. zero-mean random variables with variance one, so that the covariance matrix of $\varepsilon_i$'s is $\Sigma = A^T A$. In Designs 1, 2, 5, and 6,
$$
\Sigma_{j k} = 1\{j  = k\} + \rho 1\{j\neq k\},\quad \text{for all }j,k=1,\dots,p.
$$
In Designs 3, 4, 7, and 8,
$$
\Sigma_{j k} = \rho^{|j - k|},\quad \text{for all }j,k=1,\dots,p.
$$
We set $b = 0$ in Designs 1, 3, 5, and 7, and $b=0.8$ in Designs 2, 4, 6, and 8. For each experimental design, we consider $\rho = 0$, $0.5$, and $0.9$, and we generate $\epsilon_{i j}$'s either from Student's $t$ distribution, which we normalize to have variance one, or from the uniform on $[-a,a]$ distribution, where we set $a = \sqrt 3$, so that this distribution also has variance one. In the tables, where the results are presented, we write $\mathcal L(\epsilon) = T$ or $\mathcal L(\epsilon) = U$, depending on whether $\epsilon_{i j}$'s are simulated from Student's $t$ or from the uniform distribution.

Observe that for our data-generating process, 
$$
\mu_j = \Ep[X_{1 j}] = \theta 1\{j\leq \gamma_1 p\} - b\{\gamma_2 p < j \leq p\},\quad\text{for all $j=1,\dots,p$},
$$
so that the null hypothesis \eqref{eq: null hypothesis} holds if and only if $\theta \leq 0$ since we always set $b\geq0$. We therefore consider testing \eqref{eq: null hypothesis} against \eqref{eq: alternative hypothesis} for $\theta = 0$ (Designs 1-4; the null holds) and $\theta = 0.07$ (Designs 5-8; the null does not hold; the value 0.07 is chosen to make sure that most probabilities are bounded away from 0 and 1). Note also that when we set $\theta = 0.07$, only $\gamma_1 = 5\%$ of the inequalities violate the null hypothesis. Moreover, when we set $b = 0.8$, $1 - \gamma_2 = 90\%$ of inequalities satisfy the null and are not binding.

We consider self-normalized (SN), multiplier bootstrap (MB), and empirical bootstrap (EB) critical values. For all three methods, we consider their one- and two-step versions. For the MB and EB methods, we also consider their three-step versions. In all experiments, we set the nominal level of the test $\alpha=5\%$ and for the tests with the inequality selection, we set $\beta=0.1\%$. For the three-step methods, we set $\varphi = \beta / 2$. We present results based on 1000 simulations for each design, and we use $B=1000$ bootstrap samples for each bootstrap procedure.

In addition, to see if the methods developed specifically for testing conditional moment inequalities can be used in our setting (with ``unstructured'' inequalities), we also consider the Andrews-Shi test (note that their approach consists of first transforming the conditional moment inequalities into many unconditional ones and then testing the unconditional moment inequalities but implementing the second step does not require knowing the original structure of the conditional moment inequalities, which makes it possible to apply their test in our setting).\footnote{The tests of \cite{ArmstrongChan2012} and of \cite{Chetverikov2011} can not be implemented in our setting because they require knowledge of the original structure of the conditional moment inequalities. In particular, the critical value for the Armstrong-Chan test depends on the volume of the support of the conditioning variable and the test statistic for Chetverikov's test depends on certain conditional heteroscedasticity functions.} To implement their test, we use the test statistic $T'$ in \eqref{eq: quadratic sum test statistic}, which corresponds to their CvM statistic, and obtain the critical value via a bootstrap procedure as described in Section 9 of \cite{AndrewsShi2013}, which corresponds to their GMS critical value. We follow all their recommendations regarding the choice of the tuning parameters.


Results on the probabilities of rejecting the null in all the experiments are presented in Tables 1-4 in the online supplement. In these tables, we use $B_j$ for $B \in \{SN, MB, EB\}$ and $j\in\{1,2,3\}$ to denote $j$-step $B$ test. We also use $AS$ to denote the Andrews-Shi test.

The first observation to be taken from these tables is that the MB and EB methods give similar results. The second observation is that although the Andrews-Shi test performs well in many settings, it does not control size in some settings; for example, when $p = 1000$ and $\rho = 0$, the AS test rejects the null with probability around 15\% in Design 1 (Table 1), even though the null holds and the nominal level of the test is 5\%.   Therefore, in what follows, we only discuss and compare our SN and bootstrap (MB and EB) methods. 

Tables 1 and 2 give results for Designs 1-4, where $H_0$ holds, and demonstrate that all of our tests have good size control. The largest over-rejection occurs in Design 3 with autocorrelated data, uniform $\epsilon_{i j}$'s, $p=500$, and $\rho=0$, where the one-step EB test rejects the null with probability 7.7\% against the nominal level $\alpha=5\%$ (Table 2). As expected, the self-normalized tests tend to under-reject $H_0$ but the bootstrap tests take the correlation structure of the data into account, and have rejection probability close to nominal level $\alpha=5\%$ in Designs 1 and 3, where inequalities hold as equalities. The most striking difference between the SN and bootstrap tests in this dimension perhaps can be seen in Design 1 with equicorrelated data, uniform $\epsilon_{i j}$'s, $p=1000$, and $\rho=0.9$ where the MB and EB tests reject the null with probability between 4.8\% and 5.2\%, which is very close to the nominal level $\alpha=5\%$, but both the SN tests never reject the null. Observe also that when the correlation in the data is not too large, the SN tests also have size rather close to the nominal level; see results for Design 3 with autocorrelated data and $\rho=0$ or 0.5.


Tables 3 and 4 give results for Designs 5-8, where $\theta = 0.07$ and $H_0$ does not hold, and demonstrate power properties of our tests. Note that we have for all $j=1,\dots,p$ that $\Var(X_{1 j}) = (1 + \theta)^2 = 1.07^2 = 1.1449$. Hence, if we had only one inequality to test ($p=1$), non-trivial testing would only be possibly for $\mu_1$ at least of order $(1.1449/n)^{1/2} = 1.07/20 = 0.0535$. Instead, we have many inequalities ($p$ is large) but we set $\mu_j = 0.07$ for the inequalities that violate the null, which is of the same order as $0.0535$. Note also that in our setting, only $\gamma_1 = 5\%$ of all inequalities violate the null. Therefore, since Tables 3 and 4 show that our methods yield non-trivial rejection probabilities in most cases and sometimes yield the rejection probability close to one, we conclude that our methods have good power properties. The one-step and two-step SN tests have rejection probabilities close to those for the corresponding bootstrap tests when $\rho=0$ or even when $\rho=0.5$ for Designs 7 and 8 with autocorrelated data. Further, the one-step and two-step bootstrap tests substantially improve upon the corresponding SN tests in cases with large correlation in the data; see, for example, results for Design 5 with equicorrelated data, $\epsilon_{i j}$ having Student's t-distribution, $p=1000$ and $\rho=0.5$, where the SN tests reject $H_0$ with probability around 20\% and the corresponding bootstrap tests reject $H_0$ with probability around 40\%. Finally, selection procedures yield important power improvements. For example, for Design 8 with autocorrelated data, $\epsilon_{i j}$ having Student's t-distribution, $p=1000$ and $\rho=0.5$, the one-step MB method reject the null with probability around 40\% but the two-step method reject with probability around 90\%. Similarly, In Design 7 with autocorrelated data, $\epsilon_{i j}$ having the uniform distribution, $p = 200$ and $\rho = 0$, the two-step EB method rejects the null with probability around 50\% and the three-step EB method rejects with probability around 80\%.

\subsection{Selecting tuning parameters}
In this subsection, we carry out a small simulation study to develop a rule of thumb for selecting the tuning parameters for our methods. Since the bootstrap methods are more powerful than the SN methods, we do not consider the SN methods here. Also, since the MB and EB methods give similar results, we focus on the MB methods only. Thus, in this subsection, we only discuss the two-step and three-step MB methods but note that the same discussion applies to the corresponding EB methods.

We consider the same data-generating process as that in Design 5 in the previous subsection with $\rho = 0$, $\epsilon_{i j}$'s having uniform distribution, and $\theta = 0.07$. Instead of setting $b = 0$, however, we vary $b$ from $0.05$ to $0.8$ to see how it affects the choice of the tuning parameters.  We consider both the two-step and the three-step MB methods with $\alpha = 5\%$ and $\beta$ varying from $0.1\%$ to $1.0\%$. For the three-step MB method, we set $\varphi = \beta / 2$. Depending the simulation, we set $p = 200$ or $1000$. As in the previous subsection, we present results based on 1000 simulations for each setting, and we use $B=1000$ bootstrap samples for each bootstrap procedure. In unreported simulations, we also tried to vary $\rho$ and to use Student's distribution for $\epsilon_{i j}$'s and found results similar to those reported below. Results for the two-step and three-step MB methods are presented in Tables 5 and 6, respectively, in the online supplement.

Before looking at the simulation results, we provide some intuition regarding the choice of the tuning parameters. First, we discuss the two-step MB method, which requires selecting the tuning parameter $\beta$. Observe that increasing $\beta$ has two effects on the power of the method. One effect is that holding $\hat J_{MB}$ fixed, increasing $\beta$ leads to higher values of $c^{MB,2S}(\alpha)$ since $c^{MB,2S}(\alpha)$ is defined as the $(1 - \alpha + 2\beta)$-quantile of the conditional distribution of $W_{\hat J_{MB}}$ given $X_1^n$; see \eqref{MB-MS critical value}. The other effect is that increasing $\beta$ shrinks the set $\hat J_{MB}$, which is defined as the set of all $j$'s such that $\sqrt n\hat\mu_j/\hat\sigma_j > -2c^{MB}(\beta)$. This in turn leads to smaller values of $c^{MB,2S}(\alpha)$. Since the test statistic $T$ does not depend on $\beta$, the first effect decreases the power of the method and the second one increases it. Selecting $\beta$ therefore requires balancing these two effects. 


Further, observe that the second effect is negligible when all inequalities satisfying the null are binding or nearly binding since these inequalities will be in the set $\hat J_{MB}$ even for large values of $\beta$. Similarly, the second effect is negligible when all inequalities satisfying the null are far away from being binding since these inequalities will be out of the set $\hat J_{MB}$ even for small values of $\beta$. Thus, the second effect is non-negligible, so that it might be useful to use large values of $\beta$, only when there are inequalities under the null that are not too close and not too far away from being binding.

Our simulation results support the discussion above. Indeed, as follows from Table 5, for $p = 200$, the power of the two-step MB method is a decreasing function of $\beta$ when $b < 0.40$ and when $b > 0.55$. Therefore, the second effect is strong enough to create a non-monotonicity in the power function only in a small range of the values of $b$. Even in these cases, however, the second effect is not strong enough, so that setting $\beta = 0.1\%$ yields almost the same power as the power we would obtain by selecting $\beta$ optimally. Similar discussion also applies when $p = 1000$. Hence, the simulation results in Table 5 suggest that setting $\beta = 0.1\%$ is a good rule of thumb.\footnote{Note also that it is almost never useful to set $\beta<0.1\%$ since in this case, holding $\hat J_{MB}$ fixed, we would obtain essentially the same critical value $c^{MB,2S}(\alpha)$ as the one given by $\beta = 0.1\%$, but the substantial cost of setting $\beta < 0.1\%$ is that it can significantly increase the set $\hat J_{MB}$, relative to the set we obtain by setting $\beta = 0.1\%$.}

Next, consider the three-step MB method. The problem of selecting the tuning parameters is now much more complicated because we now have to choose two parameters, $\beta$ and $\varphi$. Regarding the choice of $\varphi$, for given value of $\beta$, selecting $\varphi$ exhibits a trade off between good power and size control: choosing larger values of $\varphi$ improves the size control but undermines the power of the test. Since there are no universally accepted rules in the literature on striking the balance between power and size control, and since our results (Theorem \ref{thm: weak inequalities}) require that $\varphi$ is not too close to zero and not too close to $\beta$, we simply set $\varphi = \beta/2$. Regarding the choice of $\beta$, although the situation is now more difficult relative to what we had with the two-step method because now both the test statistic and the critical value depends on $\beta$, the overall trade off is similar to what we had before. In particular, the simulation results in Table 6 reveal that the power of the three-step MB method is always a decreasing function of $\beta$. We therefore, again, conclude that setting $\beta = 0.1\%$ is a good rule of thumb.

\subsection{An application to market structure model}\label{sec: monte carlo market structure}
In this subsection, we show how our methods apply in an economic model setting. Specifically, we consider the market structure model from Section \ref{sec: motivating examples}. For a given market, three firms ($m = 3$) are simultaneously deciding whether to enter the market or not. For $j=1,\dots,3$, let $D_j = 1$ if the firm $j$ enters the market and $D_j = 0$ otherwise. If the firm $j$ enters the market, its profit is given by
$$
\pi_j = \sum_{l\neq j}\delta_{l j} D_l + \varepsilon + \zeta_j,
$$
where $\varepsilon$ is the market size shock that is common to all three firms, and $\zeta_j$ is an idiosyncratic shock representing specific conditions of the firm $j$ in the market. If the firm $j$ does not enter the market, $\pi_j = 0$. The objective of each firm is to maximize its profit given the decisions of other firms. 

We assume that $\varepsilon$, $\zeta_1$, $\zeta_2$, and $\zeta_3$ are i.i.d. standard normal random variables. The parameter $\delta_{l j}$ represents the effect of the presence of the firm $l$ in the market on the firm $j$. To simplify the setting, we assume that $\delta_{l j} = \delta_{j l}$ for all $j, l = 1,\dots,3$ with $j\neq l$, so that the firms have symmetric effects on each other. With this assumption, we use the following reparameterization of the model:
$$
\theta_1 = \delta_{1 2}, \quad \theta_2 = \delta_{1 3}, \quad \theta_3 = \delta_{2 3}.
$$
The random variables $\varepsilon$, $\zeta_1$, $\zeta_2$, and $\zeta_3$ are observed by the firms when they make their decisions but are not observed by the researcher. For simplicity, we also assume away any variation $X$ that is observed by the researcher.

We assume that the parameters $\theta_1$, $\theta_2$, and $\theta_3$ are all negative, so that the game always has a Nash equilibrium in pure strategies, and we focus on such equilibria. When there is only one equilibrium, we assume that the outcome of the game $D = (D_1,D_2,D_3)$ is determined by this equilibrium. When there are several equilibria, we assume that the outcome is determined by a randomly selected equilibrium, where all equilibria have the same probability of being chosen.

We consider inference on the parameters $\theta_1$, $\theta_2$, and $\theta_3$ using the data on market outcomes for $n$ i.i.d. markets. If the researcher knew that the outcome of the game were determined by a randomly selected equilibrium whenever there are several equilibria, the model would be point identified, and there would be only one value of the parameters consistent with the distribution of the outcomes. However, since the researcher typically has no reasons to believe that a particular equilibrium selection mechanism is used, we consider inference approaches from the literature on partial identification, which are agnostic about the equilibrium selection mechanism. 

Specifically, we consider two types of bounds: the \cite{CilibertoTamer2009} bounds and the \cite{GH11} bounds. The Ciliberto-Tamer (CT) bounds, which are described in Section \ref{sec: motivating examples}, give $2\cdot 2^m = 2\cdot 2^3 = 16$ inequalities:
\begin{equation}\label{eq: CT bounds}
P_1(d,\theta) \leq \Ep[1\{D = d\}] \leq P_2(d,\theta),\text{ for all }d\in \mathcal{D},
\end{equation}
where $\mathcal D = \{0,1\}^m = \{0,1\}^3$ is the set of all possible outcomes, $P_1(d,\theta)$ is the probability that the outcome $d$ is the unique equilibrium of the game, and $P_2(d,\theta)$ is the probability that the outcome $d$ is an equilibrium of the game. Since the probabilities $P_1(d,\theta)$ and $P_2(d,\theta)$ are hard to calculate exactly, we approximate them numerically using 100000 simulations of the game.

To describe the Galichon-Henry (GH) bounds, for each set of outcomes $A\subset\mathcal D$, let $\mathcal L(A,\theta)$ be the probability of observing an outcome in $A$ under the assumption that whenever the game has several equilibria, some of which are in $A$ and others are not, an equilibrium from $A$ is selected. Then the GH bounds give inequalities
\begin{equation}\label{eq: GH bounds}
\Ep[1\{D \in A\}] \leq \mathcal L(A,\theta),\text{ for all }A\subset \mathcal{D}.
\end{equation}
Thus, for each set $A$, we get one inequality, and so in total we obtain $2^{|\mathcal D|} = 2^{2^3} = 2^8 = 256$ inequalities. Note, however, that when $A = \emptyset$, the empty set, or $A = \mathcal D$, we obtain inequalities that always hold, and so we can disregard them. Thus, we have $256 - 2 = 254$ inequalities.

The major advantage of the GH bounds is that they are tight and yield the sharp identified set for $\theta = (\theta_1,\theta_2,\theta_3)$, that is, it is never possible, without further assumptions, to find a value of $\theta$ that would satisfy the inequalities \eqref{eq: GH bounds} but would be inconsistent with the distribution of the outcomes of the game.  The CT bounds do not necessarily have this property, and it may be possible to find a value of $\theta$ that would satisfy \eqref{eq: CT bounds} but would not satisfy \eqref{eq: GH bounds}. On the other hand, even though the GH bounds are useful for {\em the identification analysis}, since they produce a lot of inequalities even in simple models (254 in our case, which is a large number, and our game has only three firms), it was previously not possible to use them for {\em inference} on $\theta$. This is, however, possible using our methods. We are therefore interested to see, via simulations, how the GH bounds work for inference and also to compare the inference based on the GH bounds with that based on the CT bounds.  

For our simulations, we consider samples of size $n = 1000$, $2000$, and $5000$, which are comparable with the sample size in \cite{CilibertoTamer2009}, $n = 2742$. We always set $\theta_1 = \theta_2 = -0.6$ and $\theta_3 = -1.3$, and we consider testing the null hypothesis $H_0: \theta = \theta_0$ for different values of $\Delta\theta = \theta_0 - \theta$. To investigate size control of our methods, we use $\Delta\theta = (0,0,0)$, and to investigate their power, we use $\Delta\theta = (0.25,0,0)$, $(-0.25,0,0)$, $(0,0.25,0)$, $(0,-0.25,0)$, $(0,0,0.25)$, and $(0,0,-0.25)$. We consider the one-step and two-step versions of the SN, MB, and EB methods. In addition, we consider the three-step versions of the MB and EB methods. Note, however, that the market structure model studied here violates the conditions required for our three-step methods. In particular, we require in Section \ref{sec: three-step method} that the gradients (with respect to the parameters) of the moment functions have non-vanishing variance, $\sigma_{j l}^V > 0$, but the corresponding gradients here are non-stochastic and so have variance zero. Therefore, as a way to drop weakly informative inequalities in the three-step methods, we drop all the inequalities that have $|\mu_{ j l}^V| \leq 1/\sqrt n$ for all $l = 1,2,3$ in the notation of Section \ref{sec: three-step method}. We tried replacing $1/\sqrt n$ by $0.5/\sqrt n$ and $2/\sqrt n$ but obtained similar results. For all methods, we set $\alpha = 5\%$ and whenever needed, $\beta = 0.1\%$. For all bootstrap methods, we use 500 bootstrap samples, and for each simulation design, we repeat the experiment 1000 times to obtain rejection probabilities. The results of our simulation study are presented in Table 7 in the online supplement.

Table 7 shows that all of our methods have good size control. In particular, when $\Delta\theta = (0,0,0)$, the rejection probabilities do not exceed $3.8\%$. Also, the GH bounds give somewhat more conservative results in comparison with the CT bounds. Regarding the power, it is important to note that since the market structure model is partially identified, our methods have relatively low power against some alternatives (for example, $\Delta\theta = (0,0,-0.25)$) even when $n = 5000$ (no methods may have power against $\theta_0$ in the sharp identified set). The MB and EB methods give similar results, and the bootstrap methods are more powerful than the SN methods, especially in the case of the GH bounds; for example, when $\Delta\theta = (0,0.25,0)$ and $n = 5000$, the two-step MB method based on the GH bounds rejects the null with probability $53\%$ whereas the corresponding two-step SN method rejects the null with probability $36\%$. Three-step methods give results similar to those for the two-step methods. 

Further, it is intuitively clear that in comparison with the CT bounds, the GH bounds may be much more powerful against those $\theta_0$ that satisfy or nearly satisfy \eqref{eq: CT bounds} but do not satisfy \eqref{eq: GH bounds}. This can be seen for $\Delta\theta = (-0.25,0,0)$ and $n=5000$, where the two-step MB method based on the GH bounds rejects the null with probability $99\%$ but the same method based on the CT bounds rejects the null with probability only $70\%$ (in fact, as was reported in the previous version of the paper, when we set $\Delta\theta = (-0.2,0,0)$ and $n = 5000$, the two-step MB method rejects the null with probability $87\%$ when the GH bounds are used and only $18\%$ when the CT bounds are used). This is an important advantage of the GH bounds. On the other hand, whenever $\theta_0$ does not satisfy \eqref{eq: CT bounds}, the methods based on the CT bounds may be more powerful because they use a smaller set of inequalities, and the critical values for our methods are increasing with the number of moment inequalities used. However, the simulation results reveal that the methods based on the GH bounds, even though sometimes less powerful, are always comparable with those based on the CT bounds. When the two-step MB method is used, perhaps the largest difference in power occurs for $\Delta\theta = (0,-0.25,0)$ and $n = 5000$, where the CT and GH bounds yield the rejection probabilities $48\%$ and $34\%$, respectively.


\newpage

\begin{center}
\textbf{{\Large Online Supplement to ``Testing many moment inequalities''}}\\
\text{}\\
\end{center}

\begin{center}
{by V. Chernozhukov, D. Chetverikov, and K. Kato}\\
\text{}\\
\textit{MIT, UCLA, and Cornell University}\\
\text{}\\
\end{center}

\appendix

\section{Honest confidence regions for identifiable parameters in partially identified models}
\label{sec: confidence region}
In this section,
we consider the problem of constructing confidence regions for identifiable parameters in partially identified models defined by moment inequalities. Let $\xi_{1},\dots,\xi_{n}$ be i.i.d. random variables taking values in a measurable space $(S,\mathcal{S})$ with common distribution $P$; let $\Theta$ be a parameter space which is a Borel measurable subset of a metric space (usually a Euclidean space), and let
$g: S \times \Theta \to \R^{p}, \ (\xi, \theta) \mapsto g(\xi,\theta) = (g_{1}(\xi,\theta),\dots,g_{p}(\xi,\theta))^{T}$,
 be a jointly Borel measurable map. We consider the partially identified model where the identified set $\Theta_{0}(P)$ is given by
\[
\Theta_{0}(P) = \{ \theta \in \Theta : \Ep_{P}[ g_{j}(\xi_{1},\theta) ] \leq 0 \text{ for all} \ j=1,\dots,p \}.
\]
Here $\Ep_{P}$ means that the expectation is taken with respect to $P$ (similarly $\Pr_{P}$ means that the probability is taken with respect to $P$).  We consider the problem of constructing confidence regions $\mathcal{C}_{n}(\alpha) = \mathcal{C}_{n}(\alpha; \xi_{1},\dots,\xi_{n}) \subset \Theta$ such that for some constant $c,C > 0$, for all $n \geq 1$,
\begin{equation}
\inf_{P \in \mathcal{P}_{n}} \inf_{\theta \in \Theta_{0}(P)} \Pr_{P} ( \theta \in \mathcal{C}_{n} (\alpha)) \geq 1-\alpha - Cn^{-c}, \label{eq: honest CR}
\end{equation}
while allowing for $p > n$ (indeed we allow $p$ to be much larger than $n$), where $0 < \alpha < 1/2$ and $\mathcal{P}_{n}$ is a suitable sequence of classes  of distributions on $(S,\mathcal{S})$.
We call confidence regions $\mathcal{C}_{n}(\alpha)$ for which (\ref{eq: honest CR}) is verified {\em asymptotically honest to $\mathcal{P}_{n}$ with a polynomial rate}, where the term is inspired by \cite{Li89} and \cite{CCK13b}.

We first state the required restriction on the class of distributions $\mathcal{P}_{n}$. We assume that for every $P \in \mathcal{P}_{n}$,
\begin{align}
\begin{split}
\Theta_{0} (P) \neq \emptyset, \ \text{and} \ &\Ep_{P}[g_{j}^{2}(\xi_{1},\theta)] < \infty, \ \sigma_{j}^{2}(\theta,P) := \Var_{P}(g_{j}(\xi_{1},\theta)) > 0, \\
&\quad  \text{ for all } j=1,\dots,p, \text{ and all } \theta \in \Theta_{0}(P). \label{eq: variance cond2}
\end{split}
\end{align}

We construct confidence regions based upon duality between hypothesis testing and construction of confidence regions. For any given $\theta \in \Theta$, consider the statistic
$T(\theta) = \max_{1 \leq j \leq p} \sqrt{n} \hat{\mu}_{j}(\theta)/ \hat{\sigma}_{j}(\theta)$,
where $\hat{\mu}_{j}(\theta) = \En [g_{j}(\xi_{i},\theta)], \ \hat{\sigma}^{2}_{j}(\theta) = \En [(g_{j}(\xi_{i},\theta) - \hat{\mu}_{j}(\theta))^{2}]$.
This statistic is a test statistic for the problem of testing 
\begin{align*}
&H_{\theta}: \mu_{j}(\theta,P) \leq 0, \text{ for all }j=1,\dots,p,
\intertext{against the alternative}
&H_{\theta}': \mu_{j}(\theta,P) > 0, \text{ for some }j=1,\dots,p,
\end{align*}
where $\mu_{j}(\theta,P) := \Ep_{P}[g_{j}(\xi_{1},\theta)]$.
Pick any $\alpha \in (0,1/2)$. We consider the confidence region of the form
\begin{equation}
\mathcal{C}_{n} (\alpha) = \{ \theta \in \Theta : T (\theta) \leq c(\alpha,\theta) \}, \label{eq: confidence region}
\end{equation}
where $c(\alpha,\theta)$ is a critical value such that $\mathcal{C}_{n}(\alpha)$ contains $\theta$ with probability (approximately) at least $1-\alpha$ whenever $\theta \in \Theta_{0}(P)$.

Recall $c^{SN}(\alpha)$ defined in (\ref{SN critical value}), and let
$c^{SN,2S}(\alpha,\theta)$, $c^{MB}(\alpha,\theta)$, $c^{MB,2S}(\alpha,\theta)$, $c^{EB}(\alpha,\theta)$, $c^{EB,2S}(\alpha,\theta)$, $c^{MB,H}(\alpha,\theta)$, and $c^{EB,H}(\alpha,\theta)$
be the two-step SN, one-step MB, two-step MB, one-step EB, two-step EB, MB hybrid, and EB hybrid critical values defined in Section \ref{sec: critical value} with $X_{i} =(X_{i1},\dots,X_{ip})^{T}$ replaced by $g(\xi_{i},\theta)=(g_{1}(\xi_{i},\theta),\dots,g_{p}(\xi_{i},\theta))^{T}$. Moreover, let $\mathcal{C}_{n}^{SN}(\alpha)$ be the confidence region (\ref{eq: confidence region}) with $c(\alpha,\theta)=c^{SN}(\alpha)$; define
\[
\mathcal{C}_{n}^{SN,2S}(\alpha), \mathcal{C}_{n}^{MB}(\alpha), \mathcal{C}_{n}^{MB,2S}(\alpha),  \mathcal{C}_{n}^{EB}(\alpha), \mathcal{C}_{n}^{EB,2S}(\alpha), \mathcal{C}_{n}^{MB,H}(\alpha), \mathcal{C}_{n}^{EB,H}(\alpha)
\]
analogously. Finally, define
\begin{align*}
M_{n,k}(\theta,P) &:= \max_{1 \leq j \leq p} (\Ep_{P}[|(g_{j}(\xi_{1},\theta)-\mu_{j}(\theta,P))/\sigma_{j}(\theta,P)|^{k}])^{1/k}, \ k=3,4, \\
B_{n}(\theta,P) &:= \left (\Ep_{P}\left[ \max_{1 \leq j \leq p}|(g_{j}(\xi_{1},\theta)-\mu_{j}(\theta,P))/\sigma_{j}(\theta,P)|^{4} \right ] \right)^{1/4}.
\end{align*}
Let $0 < c_{1} < 1/2,C_{1} > 0$ be given constants.
The following theorem is the main result of this section.

\begin{theorem}
\label{thm: confidence region}
Let $\mathcal{P}_{n}^{SN}$ be the class of distributions $P$ on $(S,\mathcal{S})$ for which (\ref{eq: variance cond2}) and (\ref{eq: cond1}) are verified with $M_{n,3}$ replaced by $M_{n,3}(\theta,P)$ for all $\theta \in \Theta_{0}(P)$; let $\mathcal{P}_{n}^{SN,2S}$ be the class of distributions $P$ on $(S,\mathcal{S})$ for which (\ref{eq: variance cond2}) and (\ref{eq: cond2}) are verified with $M_{n,3},B_{n}$ replaced by (respectively) $M_{n,3}(\theta,P), B_{n}(\theta,P)$ for all $\theta \in \Theta_{0}(P)$; and let $\mathcal{P}_{n}^{B}$ be the class of distributions $P$ on $(S,\mathcal{S})$ for which (\ref{eq: variance cond2}) and (\ref{eq: cond3}) are verified with $M_{n,k},B_{n}$ replaced by (respectively) $M_{n,k}(\theta,P), B_{n}(\theta,P)$ for all $\theta \in \Theta_{0}(P)$.\footnote{For example, $\mathcal{P}_{n}^{SN}=\{ P: \ \text{(\ref{eq: variance cond2}) is verified, and} \ M_{n,3}^{3} (\theta, P) \log^{3/2} (p/\alpha) \leq C_{1}n^{1/2-c_{1}}, \forall \theta \in \Theta_{0}(P) \}$.}
Moreover, suppose that $\log (1/\beta_{n}) \leq C_{1} \log n$ whenever inequality selection is used.
Then there exist positive constants $c,C$ depending only on $\alpha,c_{1},C_{1}$ such that
\[
\inf_{P \in \mathcal{P}_{n}} \inf_{\theta \in \Theta_{0}(P)} \Pr_{P}(\theta \in \mathcal{C}_{n}(\alpha)) \geq 1-\alpha - C n^{-c}
\]
where  $(\mathcal{P}_{n},\mathcal{C}_{n})$ is one of the pairs $(\mathcal{P}_{n}^{SN}, \mathcal{C}_{n}^{SN})$, $(\mathcal{P}_{n}^{SN,2S}, \mathcal{C}_{n}^{SN,2S})$, $(\mathcal{P}_{n}^{B}, \mathcal{C}_{n}^{MB})$, $(\mathcal{P}_{n}^{B}, \mathcal{C}_{n}^{MB,2S})$, $(\mathcal{P}_{n}^{B}, \mathcal{C}_{n}^{EB})$, $(\mathcal{P}_{n}^{B}, \mathcal{C}_{n}^{EB,2S})$, $(\mathcal{P}_{n}^{B}, \mathcal{C}_{n}^{MB,H})$ or $(\mathcal{P}_{n}^{B}, \mathcal{C}_{n}^{EB,H})$.
\end{theorem}

\begin{remark}[Computationally attractive procedure]
In many applications, the parameter $\theta$ is relatively high-dimensional, and it may be computationally difficult to construct a confidence set \eqref{eq: confidence region}. In these cases, an asymptotically honest and computationally attractive procedure would be to first construct a preliminary confidence set in \eqref{eq: confidence region} by using the one-step SN method and then to eliminate the values in the preliminary confidence set that are rejected by a two-step or a three-step bootstrap method. This procedure is computationally attractive because the one-step SN critical value does not depend on $\theta$ and has to be calculated only once, so that constructing the preliminary confidence set is simple, and computationally more intense two-step or three-step bootstrap critical value does not have to be calculated for all values of $\theta \in \Theta$ but only for those in the preliminary confidence set.
\end{remark}

\section{Extensions}
\label{sec: some extensions}

\subsection{Dependent data}
\label{sec: dependent data}

In this section we consider the case where the random vectors $X_1,\dots,X_n$ are dependent. In particular, we assume $\beta$-mixing conditions. To avoid technical complications, we focus here on the non-Studentized version of $T$:
\[
\check{T} = \max_{1 \leq j \leq p} \sqrt{n} \hat{\mu}_{j}.
\]
We consider a version of the multiplier bootstrap, namely the block multiplier bootstrap, to calculate critical values for $\check{T}$, where a certain blocking technique is used to account for dependency among $X_{1},\dots,X_{n}$.\footnote{We refer to \cite{Lahiri03} as a general reference on resampling methods for dependent data.}

Our results in this section complement the set of results in  \cite{ZhangCheng2014} who, independently from us and around the same time, considered the case of the functionally-dependent time series data (the concept of functional dependence was introduced in \cite{Wu2005} and is different from $\beta$-mixing). Both our paper and \cite{ZhangCheng2014}  extend Gaussian approximation and bootstrap results of \cite{CCK12} to the case of dependent data but under different dependence conditions (that do not nest each other). The results obtained in these two papers are strongly complementary and, taken together, cover a wide variety of dependent data processes, thereby considerably expanding the applicability of the proposed tests.

 Let $X_{1},\dots,X_{n}$ be possibly dependent random vectors in $\R^{p}$ with identical distribution (that is, $X_{i} \stackrel{d}{=} X_{1}, \text{ for all }i=1,\dots,n$), defined on the probability space $(\Omega,\mathcal{A},\Pr)$. We follow the basic notation introduced in Section \ref{sec: test statistic}. For the sake of simplicity, assume that there exists a constant $D_n \geq 1$ such that $|X_{ij} - \mu_j | \leq D_n$ a.s. for $1 \leq i \leq n, 1 \leq j \leq p$.

For any integer $1 \leq q \leq n$, define
\begin{align*}
&\overline{\sigma}^{2} (q) := \max_{1 \leq j \leq p} \max_{I} \Var \left (q^{-1/2} \sum_{i \in I} X_{ij} \right), \\
&\underline{\sigma}^{2} (q) := \min_{1 \leq j \leq p} \min_{I} \Var \left (q^{-1/2} \sum_{i \in I} X_{ij} \right),
\end{align*}
where $\max_{I}$ and $\min_{I}$ are taken over all $I \subset \{1,\dots,n\}$ of the form $I=\{ i+1,\dots,i+q \}$. For any sub $\sigma$-fields $\mathcal{A}_{1},\mathcal{A}_{2} \subset \mathcal{A}$, define
\begin{align*}
\beta (\mathcal{A}_{1},\mathcal{A}_{2}) := \frac{1}{2} \sup \Bigg \{ &\sum_{i} \sum_{j}  \Pr (A_{i} \cap B_{j}) - \Pr(A_{i})\Pr(B_{j}) | : \\
&\quad \text{$\{ A_{i} \}$ is any finite partition of $\Omega$ in $\mathcal{A}_{1}$}, \\
&\quad \text{$\{ B_{j} \}$ is any finite  partition of $\Omega$ in $\mathcal{A}_{2}$} \Bigg \}.
\end{align*}
Define the $k$th $\beta$-mixing coefficient for $X_{1}^{n} = \{ X_{1},\dots,X_{n} \}$ by
\[
b_{k} = b_{k} (X_{1}^{n}) = \max_{1 \leq  l \leq n-k} \beta (\sigma (X_{1},\dots,X_{l}),\sigma(X_{l+k},\dots,X_{n})), \ 1 \leq k \leq n-1,
\]
where $\sigma (X_{i}, i \in I)$ with $I \subset \{ 1,\dots, n \}$ is the $\sigma$-field generated by $X_{i}, i \in I$.\footnote{We refer to \cite{FY03}, Section 2.6, as a general reference on mixing.}

We employ Bernstein's ``small-block and large-block'' technique and decompose the sequence $\{1,\dots,n \}$ into ``large'' and ``small'' blocks.
Let $q > r$ be positive integers with $q+r \leq n/2$ ($q,r$ depend on $n$: $q=q_n,r=r_n$, and asymptotically we require $q_n \to \infty, q_n = o(n), r_n \to \infty$, and $r_n = o(q_n)$), and let $I_{1} = \{ 1,\dots,q \}, J_{1} = \{ q+1,\dots,q+r \}, \dots,
I_{m} = \{ (m-1)(q+r) + 1,\dots, (m-1)(q+r) + q \}, J_{m} = \{ (m-1)(q+r) + q+1,\dots,m(q+r) \}, J_{m+1} = \{ m(q+r),\dots, n \}$,
where $m = m_n =  [n/(q+r)]$ (the integer part of $n/(q+r)$). The $q$ and $r$ are the lengths of large and small blocks, respectively, and $m$ is the number of blocks. 

Then the block multiplier bootstrap is described as follows: generate independent standard normal random variables $\epsilon_1,\dots,\epsilon_m$, independent of $X_{1}^{n}$. Let
\[
\check{W} = \max_{1 \leq j \leq p} \frac{1}{\sqrt{mq}} \sum_{l=1}^{m} \epsilon_{l} \sum_{i \in I_{l}} (X_{ij} - \hat{\mu}_{j}),
\]
and consider
\[
\hat{c}^{BMB}(\alpha) =  \text{conditional $(1-\alpha)$-quantile of $\check{W}$ given $X_{1}^{n}$},
\]
which we call the BMB (Block Multiplier Bootstrap) critical value.
\begin{theorem}[Validity of BMB method]
\label{thm: validity of BMB}
Work under the setting described above. Suppose that there exist constants $0 < c_1 \leq C_1$ and $0 < c_2 < 1/4$ such that
$c_{1} \leq \underline{\sigma}^{2}(q) \leq \overline{\sigma}^{2}( r ) \vee \overline{\sigma}^{2}(q)  \leq C_{1}, \max \{ m b_{r}, (r/q) \log^{2} p  \} \leq C_{1} n^{-c_2}$, and $q D_n \log^{5/2} (pn)  \leq C_{1}n^{1/2-c_{2}}$.
Then there exist positive constants $c,C$ depending only on $c_1,c_2,C_1$ such that under $H_0$, 
$$
\Pr (\check{T} > \hat{c}^{BMB}(\alpha)) \leq \alpha + Cn^{-c}.
$$
In addition, if $\mu_j = 0$ for all $1 \leq j \leq p$, then 
$$
| \Pr (\check{T} > \hat{c}^{BMB}(\alpha)) - \alpha | \leq C n^{-c}.
$$
\end{theorem}

\begin{remark}[Connection to tapered block bootstrap]
The BMB method can be considered as a variant of the tapered block bootstrap \citep[see][]{PP01,PP02,Andrews04} applied to non-overlapping blocks with a rectangular tapering function.
The difference is that in the original tapered block bootstrap the multipliers are multinomially distributed, while in the BMB the multipliers are independent standard normal.
\end{remark}

\subsection{Approximate moment inequalities}
\label{sec: approximate moment inequalities}
As shown in a dynamic model of imperfect competition example in Section \ref{sec: dynamic game}, in some applications, random vectors $X_1,\dots,X_n$ satisfying inequalities (\ref{eq: null hypothesis}) with $\mu_j=\Ep[X_{1j}]$ are not observed. Instead, the data consist of random vectors $\hat{X}_1,\dots,\hat{X}_n$ that approximate vectors $X_1,\dots,X_n$. In that example, the approximation error arises from the need to linearize original inequalities. Another possibility leading to a nontrivial approximation error is that where the data contain estimated parameters. In this section, we derive a set of conditions that suffice for the same results as those obtained in Section \ref{sec: critical value} when we use the data $\hat{X}_1,\dots,\hat{X}_n$ as if we were using exact vectors $X_1,\dots,X_n$. For brevity, we only consider two-step MB/EB methods.

We use the following notation. Let $\hat{\mu}_{j,0}:=\En[X_{i j}]$ and $\hat{\sigma}_{j,0}^2:=\En[(X_{i j}-\hat{\mu}_{j,0})^2]$ denote (infeasible) estimators of $\mu_j=\Ep[X_{1j}]$ and $\sigma_j^2=\text{Var}(X_{1j})$. In addition, assume that we have estimates $\hat{\mu}_j$ that appropriately approximate $\hat{\mu}_{j,0}$ for $j=1,\dots,p$. In the context of Section \ref{sec: dynamic game}, for example, these estimates would take the form $\hat{V}_j(s,\sigma_j',\hat\sigma_{-j},\theta)-\hat{V}_j(s,\hat\sigma_j,\hat\sigma_{-j},\theta)$. Moreover, let $\hat{\sigma}_j^2:=\En[(\hat X_{i j}-\hat{\mu}_j)^2]$ be a (feasible) estimator of $\sigma_j^2$.

Define the test statistic $T$ by (\ref{eq: test statistic}); that is, $T=\max_{1\leq j\leq p}\sqrt{n}\hat{\mu}_j/\hat{\sigma}_j$. Define the critical value $c^{B,2S}(\alpha)$ for $B=MB$ or $EB$ by the same algorithms as those used in Section \ref{sec: critical value} with $X_{i j}$ replaced by $\hat{X}_{i j}$ for all $i$ and $j$ (and using $\hat{\mu}_j$ and $\hat{\sigma}_j^2$ as defined in this section). We have the following theorem:

\begin{theorem}[Validity of two-step MB/EB methods for approximate inequalities]\label{thm: approximate moment inequalities}
Let $c^{B,2S}(\alpha)$ stand either for $c^{MB,2S}(\alpha)$ or $c^{EB,2S}(\alpha)$. Suppose that the assumption of Theorem \ref{thm: simulation plugin method} is satisfied. Moreover, suppose that $\log(1/\beta_n)\leq C_1\log n$. In addition, suppose that there exists a sequence $\zeta_{n1}$ satisfying $\zeta_{n1}\log p\leq C_1 n^{-c_1}$ and such that 
\begin{equation}\label{eq: main condition approximate inequalities}
\Pr\left(\max_{1\leq j\leq p}\sqrt{n}|\hat{\mu}_j-\hat{\mu}_{j,0}|>\zeta_{n1}\right)\leq C_1n^{-c_1}
\end{equation}
and 
$$
\Pr\left(\max_{1\leq j\leq p}(\En[(\hat{X}_{i j}-X_{i j})^2])^{1/2}>\zeta_{n1}\right)\leq C_1n^{-c_1}.
$$
Moreover, if the EB method is used, suppose that 
$$
\Pr\left(\sqrt{\log p}\max_{i,j}|\hat{X}_{i j}-X_{i j}|>\sqrt{n}\zeta_{n,1}\right)\leq C_1 n^{-c_1}.
$$
Finally, assume that $\sigma_{j}\geq c_1$ for all $1\leq j \leq p$. Then all the conclusions of Theorem \ref{thm: simulation MS method} hold with $T$, $c^{MB,2S}(\alpha)$, and $c^{EB,2S}(\alpha)$ defined in this section.
\end{theorem}

\begin{remark}[Data with estimated parameters]
When Theorem \ref{thm: approximate moment inequalities} is applied to data with estimated parameters, verifying \eqref{eq: main condition approximate inequalities} typically requires imposing further conditions, even when $p$ is small. For example, suppose that we observe a random sample $(V_i,Y_i)$, $i=1,\dots,n$, from the distribution of $(V,Y)$, where both $V$ and $Y$ are scalar random variables. Suppose further that $\theta = \Ep[Y]$ and that we are interested in testing whether $\Ep[f(V,\theta)] \leq 0$ for some known function $f:\mathbb R^2\to\mathbb R$. To map this problem into the setting of Theorem \ref{thm: approximate moment inequalities}, denote $X_i = f(V_i,\theta)$ and $\hat X_i = f(V_i,\hat\theta)$, where $\hat\theta = n^{-1}\sum_{i=1}^n Y_i$. Moreover, let $\hat\mu = n^{-1}\sum_{i=1}^n \hat X_i$ and $\hat \mu_0 = n^{-1}\sum_{i=1}^n X_i$. Finally, denote by $f'(V,\theta)$ the derivative of $f(V,\theta)$ with respect to $\theta$. Then, under mild regularity conditions, $\sqrt n(\hat\mu - \hat\mu_0) = n^{-1/2}\sum_{i=1}^n f'(V_i,\theta)(\hat\theta - \theta) + o_P(1)$, and so \eqref{eq: main condition approximate inequalities} can be verified only if $\Ep[f'(V,\theta)] = 0$.
\end{remark}

\section{Details on equations (\ref{eq: approximate data}) and (\ref{eq: mean to test}) in the main text}\label{sec: details}
In this section, we continue discussion of the ``Dynamic model of imperfect competition'' example presented in Section \ref{sec: motivating examples}. In particular, we explain how to construct $X_{i j}(s,\sigma_j',\theta)$'s that satisfy
\begin{equation}\label{eq: approximate data 2}
\hat{X}_{i j}(s,\sigma_j',\theta)=X_{i j}(s,\sigma_j',\theta)+o_P(1)
\end{equation}
and
\begin{equation}\label{eq: mean to test 2}
\Ep[X_{i j}(s,\sigma_j',\theta)]=V_j(s,\sigma_j',\sigma_{-j},\theta)-V_j(s,\sigma,\theta),
\end{equation}
which are needed to apply results in Appendix \ref{sec: approximate moment inequalities}. We also show that setting 
$$
\hat\mu_j := \hat{V}_j(s,\sigma_j',\hat\sigma_{-j},\theta)-\hat{V}_j(s,\hat\sigma_j,\hat\sigma_{-j},\theta)
$$ 
gives 
\begin{equation}\label{eq: verification third main condition}
\sqrt n|\hat\mu_j - \hat\mu_{j,0}| = o_P(n^{-1/2})
\end{equation}
with $\hat\mu_{j,0} = n^{-1}\sum_{i=1}^n X_{i j}(s,\sigma_j',\theta)$, which is also needed to apply results in Appendix \ref{sec: approximate moment inequalities}. We continue to assume that the data consist of observations on $n$ i.i.d. markets.

To construct $X_{i j}(s,\sigma_j',\theta)$'s, assume the following linear expansions:
\begin{align}
& \sqrt{n}(\hat{V}_j(s,\hat\sigma_j,\hat\sigma_{-j},\theta)-V_j(s,\sigma_j,\sigma_{-j},\theta))\nonumber\\
&\qquad =\frac{1}{\sqrt{n}}\sum_{k=1}^n\psi_{k j}(s,\theta)+o_P(n^{-1/2})\label{eq: linear expansion 1}
\end{align}
and
\begin{align}
&\sqrt{n}(\hat{V}_j(s,\sigma_j',\hat\sigma_{-j},\theta)-V_j(s,\sigma_j',\sigma_{-j},\theta))\nonumber\\
&\qquad=\frac{1}{\sqrt{n}}\sum_{k=1}^n\psi_{k j}'(s,\sigma_j',\theta)+o_P(n^{-1/2}),\label{eq: linear expansion 2}
\end{align}
where $\psi_{k j}$ and $\psi_{k j}'$ are influence functions depending only on the data for the market $k$ and satisfying
\begin{equation}\label{eq: influence functions}
\Ep[\psi_{k j}(s,\theta)]=0\ \text{ and }\ \Ep[\psi_{k j}'(s,\sigma_j',\theta)]=0.
\end{equation}
These are standard expansions that hold in many settings, so for brevity, we do not discuss the regularity conditions behind them. Then, considering leave-market-$i$-out estimates $\hat V_j^{-i}(s,\sigma',\theta)$ and $\sigma^{-i}$ as in the main text, we obtain
\begin{align*}
& \sqrt{n-1}(\hat{V}_j^{-i}(s,\hat\sigma_j^{-i},\hat\sigma_{-j}^{-i},\theta)-V_j(s,\sigma_j,\sigma_{-j},\theta))\\
&\qquad=\frac{1}{\sqrt{n-1}}\sum_{k=1;\, k\neq i}^n\psi_{k j}(s,\theta)+o_P(n^{-1/2})
\end{align*}
and
\begin{align*}
& \sqrt{n-1}(\hat{V}_j^{-i}(s,\sigma_j',\hat\sigma_{-j}^{-i},\theta)-V_j(s,\sigma_j',\sigma_{-j},\theta))\\
&\qquad=\frac{1}{\sqrt{n-1}}\sum_{k=1;\, k\neq i}^n\psi_{k j}'(s,\sigma_j',\theta)+o_P(n^{-1/2}).
\end{align*}
Hence, we have for all $i=1,\dots,n$,
\begin{align*}
\tilde{X}_{i j}(s,\theta)&:=n\hat{V}_j(s,\hat\sigma_j,\hat\sigma_{-j},\theta)-(n-1)\hat{V}_j^{-i}(s,\hat\sigma_j^{-i},\hat\sigma_{-j}^{-i},\theta)\\
&=V_j(s,\sigma_j,\sigma_{-j},\theta)+\psi_{i j}(s,\theta)+o_P(1)
\end{align*}
and
\begin{align*}
\tilde{X}_{i j}'(s,\sigma_j',\theta)&:=n\hat{V}_j(s,\sigma_j',\hat\sigma_{-j},\theta)-(n-1)\hat{V}_j^{-i}(s,\sigma_j',\hat\sigma_{-j}^{-i},\theta)\\
&=V_j(s,\sigma_j',\sigma_{-j},\theta)+\psi_{i j}'(s,\sigma_j',\theta)+o_P(1).
\end{align*}
Therefore, defining
\begin{align*}
X_{i j}(s,\sigma_j',\theta):=&V_j(s,\sigma_j',\sigma_{-j},\theta)-V_j(s,\sigma_j,\sigma_{-j},\theta)+\psi_{i j}'(s,\sigma_j',\theta)-\psi_{i j}(s,\theta),
\end{align*}
we obtain
$$
\hat{X}_{i j}(s,\sigma_j',\theta)=\tilde{X}_{i j}'(s,\sigma_j',\theta)-\tilde{X}_{i j}(s,\theta)=X_{i j}(s,\sigma_j',\theta)+o_P(1).
$$
Combining these equalities with (\ref{eq: influence functions}) implies (\ref{eq: approximate data}) and (\ref{eq: mean to test}) from the main text. Moreover, observe that it follows from \eqref{eq: linear expansion 1} and \eqref{eq: linear expansion 2} that
\begin{align*}
\hat\mu_j &= \hat{V}_j(s,\sigma_j',\hat\sigma_{-j},\theta)-\hat{V}_j(s,\hat\sigma_j,\hat\sigma_{-j},\theta)\\
& = V_j(s,\sigma_j',\sigma_{-j},\theta)-V_j(s,\sigma_j,\sigma_{-j},\theta) \\
&\qquad + \frac{1}{n}\sum_{i=1}^n(\psi_{i j}'(s,\sigma_j',\theta) - \psi_{i j}(s,\theta)) + o_P(n^{-1})\\
& = \frac{1}{n}\sum_{i=1}^n X_{i j}(s,\sigma_j',\theta) + o_P(n^{-1}) = \hat\mu_{j,0} + o_P(n^{-1}),
\end{align*}
and so \eqref{eq: verification third main condition} holds. Finally, observe that by imposing further regularity conditions on the terms $o_P(n^{-1/2})$ in \eqref{eq: linear expansion 1} and \eqref{eq: linear expansion 2}, it is rather standard to make sure that \eqref{eq: approximate data 2} holds uniformly over $i$ and $j$ and that \eqref{eq: verification third main condition} holds uniformly over $j$, which are the needed to apply results in Appendix \ref{sec: approximate moment inequalities}.

\section{Proofs}

In what follows, let $\phi (\cdot)$ denote the density  function of the standard normal distribution, and
let $\overline{\Phi}(\cdot) = 1-\Phi(\cdot)$ where recall that $\Phi(\cdot)$ is the distribution function of the standard normal distribution.

\subsection{Technical tools}
We state here some technical tools used to prove the theorems.
The following lemma states a moderate deviation inequality for self-normalized sums.

\begin{lemma}\label{lem: moderate deviations}
Let $\xi_{1},\dots, \xi_{n}$ be independent centered random variables with $\Ep [ \xi_{i}^{2} ] = 1$ and $\Ep [| \xi_{i} |^{2+\nu} ] < \infty$ for all $1 \leq i \leq n$ where $0 < \nu \leq 1$.
Let $S_{n}=\sum_{i=1}^{n} \xi_{i}, V_{n}^2=\sum_{i=1}^{n} \xi_{i}^2$, and $D_{n,\nu}=(n^{-1}\sum_{i=1}^{n} \Ep [ |\xi_{i}|^{2+\nu}])^{1/(2+\nu)}$. Then uniformly in $0\leq x\leq n^{\frac{\nu}{2(2+\nu)}}/D_{n,\nu}$,
\[
\left | \frac{\Pr(S_n/V_n\geq x)}{\overline{\Phi}(x)} - 1 \right | \leq  K n^{-\nu/2} D_{n,\nu}^{2+\nu} (1+x)^{2+\nu},
\]
where $K$ is a universal constant.
\end{lemma}
\begin{proof}
See Theorem 7.4 in \cite{Shao_book} or the original paper, \cite{JSW03}.
\end{proof}

The following lemma states a Fuk-Nagaev type inequality, which is a deviation inequality for the maximum of the sum of random vectors from its expectation.

\begin{lemma}[A Fuk-Nagaev type inequality]
\label{lem: fuk-nagaev}
Let $X_{1},\dots,X_{n}$ be independent random vectors in $\R^{p}$.
Define $\sigma^{2} := \max_{1 \leq j \leq p} \sum_{i=1}^{n} \Ep [X_{ij}^{2}]$. Then for every $s > 1$ and $t > 0$,
\begin{multline*}
\Pr \left  ( \max_{1 \leq j \leq p} \Big | \sum_{i=1}^{n} (X_{ij} - \Ep [ X_{ij} ]) \Big | \geq 2\Ep \Big [ \max_{1 \leq j \leq p} \Big | \sum_{i=1}^{n} (X_{ij} - \Ep [ X_{ij} ]) \Big | \Big  ]+t \right )
\\ \leq e^{-t^{2}/(3\sigma^{2})} + \frac{K_{s}}{t^s} \sum_{i=1}^{n} \Ep \left[ \max_{1 \leq j \leq p} | X_{ij} |^{s} \right],
\end{multline*}
where $K_{s}$ is a constant depending only on $s$.
\end{lemma}

\begin{proof}
See Theorem 3.1 in \cite{EL08}. Note that \cite{EL08} assumed that $s > 2$ but their proof applies to the case where $s > 1$.
More precisely, we apply Theorem 3.1 in \cite{EL08} with $(B,\| \cdot \|)=(\R^{p}, | \cdot |_{\infty})$ where $| x |_{\infty} = \max_{1 \leq j \leq p} | x_{j} |$ for $x = (x_{1},\dots,x_{p})^{T}$, and $\eta = \delta = 1$. The unit ball of the dual of $(\R^{p}, | \cdot |_{\infty})$ is the set of linear functions $\{ x =(x_{1},\dots,x_{p})^{T} \mapsto \sum_{j=1}^{p} \lambda_{j} x_{j} : \sum_{j=1}^{p} | \lambda_{j} | \leq 1 \}$, and for $\lambda_{1},\dots,\lambda_{p}$ with $\sum_{j=1}^{p} | \lambda_{j} | \leq 1$, by Jensen's inequality,
\begin{align*}
&\textstyle \sum_{i=1}^{n} \Ep \left[ (\sum_{j=1}^{p} \lambda_{j} X_{ij})^{2} \right]
=\sum_{i=1}^{n} \Ep \left[ (\sum_{j=1}^{p} | \lambda_{j} | \sign (\lambda_{j}) X_{ij})^{2} \right] \\
&\quad\leq \textstyle \sum_{j=1}^{p} |\lambda_{j}| \sum_{i=1}^{n} \Ep [ X_{ij}^{2} ]  \leq \max_{1 \leq j \leq p} \sum_{i=1}^{n} \Ep [ X_{ij}^{2} ] =\sigma^{2},
\end{align*}
where $\sign (\lambda_{j})$ is the sign of $\lambda_{j}$.
Hence in this case $\Lambda^{2}_{n}$ in Theorem 3.1 of \cite{EL08} is bounded by (and indeed equal to) $\sigma^{2}$.
\end{proof}

In order to use Lemma \ref{lem: fuk-nagaev}, we need a suitable bound on  the expectation of the maximum.
The following lemma is useful for that purpose.

\begin{lemma}
\label{lem: maximal ineq}
Let $X_{1},\dots,X_{n}$ be independent random vectors in $\R^{p}$ with $p \geq 2$. Define $M := \max_{1 \leq i \leq n} \max_{1 \leq j \leq p} | X_{ij} |$ and $\sigma^{2} := \max_{1 \leq j \leq p} \sum_{i=1}^{n} \Ep [ X_{ij}^{2} ]$.
Then
\begin{equation*}
\Ep \left [\max_{1 \leq j \leq p} \Big  | \sum_{i=1}^{n} (X_{ij} - \Ep [X_{ij}]) \Big  | \right ] \leq K (\sigma \sqrt{\log p} + \sqrt{\Ep [ M^{2} ]} \log p),
\end{equation*}
where $K$ is a universal constant.
\end{lemma}

\begin{proof}
See Lemma 8 in \cite{CCK13}.
\end{proof}

For bounding $\Ep [ M^{2} ]$, we will frequently use the following inequality: let $\xi_{1},\dots,\xi_{n}$ be arbitrary random variables with $\Ep [ | \xi_{i} |^{s} ] < \infty$ for all $1 \leq i \leq n$ for some $s \geq 1$. Then
\begin{align*}
\Ep [\max_{1 \leq i \leq n} | \xi_{i} | ] &\leq (\Ep[ \max_{1 \leq i \leq n} | \xi_{i} |^{s}])^{1/s} \\
&\leq ({\textstyle \sum}_{i=1}^{n} \Ep [ | \xi_{i} |^{s} ])^{1/s} \leq n^{1/s} \max_{1 \leq i \leq n} (\Ep[ | \xi_{i} |^{s}])^{1/s}.
\end{align*}
For centered normal random variables $\xi_{1},\dots,\xi_{n}$ with $\sigma^{2} = \max_{1 \leq i \leq n} \Ep[ \xi_{i}^{2} ]
$, we have
\[
\Ep\left[ \max_{1 \leq j \leq p} \xi_{i} \right] \leq \sqrt{2 \sigma^{2} \log p}.
\]
See, for example, Proposition 1.1.3 in \cite{Talagrand2003}.

\begin{lemma}\label{lem: critical value bound}
Let $(Y_1,\dots,Y_p)^{T}$ be a normal random vector with $\Ep [ Y_{j} ]= 0$ and $\Ep[ Y_{j}^{2} ]= 1$ for all $1 \leq j \leq p$.
(i) For $\alpha\in(0,1)$, let $c_0(\alpha)$ denote the $(1-\alpha)$-quantile of the distribution of $\max_{1\leq j\leq p}Y_j$. Then
$c_0(\alpha)\leq \sqrt{2\log p}+\sqrt{2 \log(1/ \alpha)}$.
(ii) For every $t \in \R$ and $\epsilon > 0$,
$\Pr ( | \max_{1 \leq j \leq p} Y_{j} - t | \leq \epsilon ) \leq 4 \epsilon (\sqrt{2 \log p} +1)$.
\end{lemma}
\begin{proof}
Part (ii) follows from Theorem 3 in \cite{CCK13} together with the fact that
\begin{equation}\label{eq: gaussian maximal inequality}
\Ep\left[\max_{1\leq j\leq p}Y_j\right]\leq \sqrt{2\log p}.
\end{equation}
For part (i),
by the Borell-Sudakov-Tsirelson inequality (see Theorem A.2.1 in \cite{VW96}), for every $r > 0$,
\[
\Pr\left(\max_{1\leq j\leq p}Y_j\geq \Ep\Big[\max_{1\leq j\leq p}Y_j\Big]+r\right)\leq e^{-r^2/2},
\]
by which we have
\begin{equation}\label{eq: critical value bound}
c_0(\alpha)\leq \Ep\left[\max_{1\leq j\leq p}Y_j\right]+\sqrt{2\log (1/\alpha)}.
\end{equation}
Combining (\ref{eq: critical value bound}) and (\ref{eq: gaussian maximal inequality}) leads to the desired result.
\end{proof}

\subsection{On Bonferroni approach}
We state and prove here a result on validity of the Bonferroni approach for testing \eqref{eq: null hypothesis} against \eqref{eq: alternative hypothesis}.
\begin{theorem}[Validity of Bonferroni method]\label{thm: bonferroni}
If there exist constants $0<c_{1}<1/2$ and $C_{1} >0$ such that
\begin{equation}
M_{n,3}^{3} \log^{3/2} (p/\alpha) \leq C_{1}n^{1/2-c_{1}}, \label{eq: cond bonf}
\end{equation}
then there exists a positive constant $C$ depending only on $C_1$ such that under $H_{0}$,
\begin{equation}\label{eq: SN test size bound}
\Pr(T>c^{Bon}(\alpha)) \leq \alpha + Cn^{-c_{1}},
\end{equation}
where $c^{Bon}(\alpha) = \Phi^{-1}(1 - \alpha/p)$. Moreover, this bound holds uniformly over all distributions $\mathcal L_X$ satisfying (\ref{eq: variance cond}) and (\ref{eq: cond bonf}).
\end{theorem}

\begin{proof}
For brevity of notation, denote $c_0 = c^{Bon}(\alpha) = \Phi^{-1}(1 - \alpha/p)$. Then by \eqref{eq: first step}, under the null,
\begin{align*}
&\Pr(T > c^{Bon}(\alpha))
 = \Pr(T > c_0) \leq \sum_{j=1}^p \Pr\left( U_j > c_0/\sqrt{1 + c_0^2/n} \right)\\
&\qquad \leq \sum_{j=1}^p \Big(1 + K n^{-1/2}M_{n,3}^3(1 + \Phi^{-1}(1 - \alpha/p))^3\Big)\bar\Phi\left(c_0 / \sqrt{1 + c_0^2 / n}\right)
\end{align*}
for some absolute constant $K>0$, where the last inequality follows by applying Lemma \ref{lem: moderate deviations} with $S_n/V_n = U_j$, $x = c_0/\sqrt{1 + c_0^2/n}$, $\nu = 1$, and $D_{n,1} = M_{n,3}$. Hence, like in the proof of Theorem \ref{thm: analytical plugin method},
\begin{equation}\label{eq: bonf der 1}
\Pr(T > c^{Bon}(\alpha)) \leq p(1 + C' n^{-c_1})\bar\Phi\left(c_0 / \sqrt{1 + c_0^2/n}\right)
\end{equation}
for some constant $C'$ depending only on $C_1$. Further,
\begin{align*}
\bar\Phi\left(c_0 / \sqrt{1 + c_0^2/n}\right) 
&\leq \bar\Phi(c_0) + \phi(c_0)\left(c_0 - c_0/\sqrt{1 + c_0^2/n} \right)\\
&\leq \bar\Phi(c_0) + c_0\phi(c_0)\left( \sqrt{1 + c_0^2/n} -1 \right)\\
&\leq \bar\Phi(c_0) + c_0^3\phi(c_0)/n,
\end{align*}
where $\phi$ is the pdf of the standard normal distribution. Also, it follows from Proposition 2.2.1 in \cite{Dudley02} that $\phi(c_0) \leq K' c_0\bar\Phi(c_0)$ for some absolute constant $K'$. Moreover, by the proof of Theorem \ref{thm: analytical plugin method},
$$
c_0 = c^{Bon}(\alpha) = \Phi^{-1}(1 - \alpha/p) \leq \sqrt{2\log(p / \alpha)}.
$$
Hence,
$$
\bar\Phi\left(c_0 / \sqrt{1 + c_0^2/n}\right) \leq \bar\Phi(c_0)(1 + K'c_0^4/n) \leq \bar\Phi(c_0)\Big(1 + 4\log^2(p/\alpha)/n\Big).
$$
Thus, given that $M_{n,3}\geq 1$ and that $\log(p/\alpha) \geq \log 4 > 1$, it follows from \eqref{eq: cond bonf} that for some constant $C''$ depending only on $C_1$,
$$
\bar\Phi\left(c_0/\sqrt{1 + c_0^2/n}\right) \leq \bar\Phi(c_0)(1 + C'' n^{-2c_1}) = (\alpha/p)(1 + C'' n^{-2c_1}).
$$
Combining this bound with \eqref{eq: bonf der 1} gives the first assertion. The second assertion follows from the first one because the constant $C$ depends only on $C_1$. This completes the proof of the theorem.
\end{proof}

\subsection{Proof of Theorem \ref{thm: analytical plugin method}}
Combining \eqref{eq: first step} with \eqref{SN critical value} shows that under the null,
$$
\Pr(T > c^{SN}(\alpha)) \leq \sum_{j=1}^p \Pr(U_j > \Phi^{-1}(1 - \alpha/p)).
$$
The first assertion thus follows immediately by applying Lemma \ref{lem: moderate deviations} to bound $\Pr(U_j > \Phi^{-1}(1 - \alpha/p))$ with $S_n/V_n = U_j$, $x = \Phi^{-1}(1 - \alpha/p)$, $\nu = 1$, and $D_{n,1} = M_{n,3}$.

To prove the second assertion, we first note the well known fact that $1-\Phi (t) \leq e^{-t^2/2}$ for $t > 0$, by which we have $\Phi^{-1}(1-\alpha/p) \leq \sqrt{2\log (p/\alpha)}$.\footnote{The inequality $1-\Phi (t) \leq e^{-t^2/2}$ for $t > 0$ can be proved by using Markov's inequality, $\Pr ( \xi > t ) \leq e^{-\lambda t} \Ep [ e^{\lambda \xi}]$ for $\lambda > 0$ with $\xi \sim N(0,1)$, and optimizing the bound with respect to $\lambda > 0$; there is a sharper inequality, namely $1-\Phi (t) \leq e^{-t^2/2}/2$ for $t > 0$ \citep[see, for example, Proposition 2.1 in][]{Dudley02}, but we do not need this sharp inequality in this paper.} 
Further, since we are assuming $p \geq 2$, $2\log (p/\alpha) \geq 1$ and thus $1+\Phi^{-1}(1-\alpha/p) \leq 2\sqrt{2\log (p/\alpha)}$. 
Hence if $M_{n,3}^{3} \log^{3/2}(p/\alpha) \leq C_{1}n^{1/2-c_{1}}$, it is straightforward to verify that $\alpha Kn^{-1/2} M_{n,3}^{3} \{ 1+\Phi^{-1}(1-\alpha/p) \}^{3}$ is bounded by $Cn^{-c_{1}}$ for some constant $C$ depending only on $C_{1}$, which gives the second assertion. The third assertion follows immediately from the second one since the constant $C$ in \eqref{eq: SN test size bound} depends only on $C_1$.

To prove the last assertion, \eqref{eq: third assertion sn critical value}, we have
\begin{align}
\Pr(T > c^{SN}(\alpha))
&=\Pr\Big( \max_{1\leq j\leq p}\frac{\sqrt n(\hat\mu_j - \mu_j)}{\hat\sigma_j} > c^{SN}(\alpha) \Big)\nonumber\\
&=1 - \prod_{1\leq j\leq p}\Pr\Big( \frac{\sqrt n(\hat \mu_j - \mu_j)}{\hat\sigma_j} \leq c^{SN}(\alpha) \Big)\nonumber\\
&=1 - \prod_{1\leq j\leq p}\Big(1 - \Pr\Big( \frac{\sqrt n(\hat \mu_j - \mu_j)}{\hat\sigma_j} > c^{SN}(\alpha) \Big)\Big)\nonumber\\
&=1 - \prod_{1\leq j\leq p}\Big(1 - \Pr\Big(U_j > \Phi^{-1}(1 - \alpha/p)\Big)\Big),\label{eq: bound under independence}
\end{align}
where the first line follows from the fact that $\mu_j = 0$ for all $j = 1,\dots,p$, the second from independence of components of $X_1$, the third from the formula for probabilities of complements, and the fourth from the definitions of $U_j$'s and $c^{SN}(\alpha)$. Now using the same arguments as those in the proof of the first two assertions, the expression in \eqref{eq: bound under independence} is bounded from below by
\begin{align*}
&1 - \prod_{1\leq j\leq p}\Big(1 - (1 - C n^{-c_1})\alpha/p \Big)=1 - \Big(1 - (1 - C n^{-c_1})\alpha/p\Big)^p\to 1 - e^{-\alpha}
\end{align*}
and from above by
\begin{align*}
&1 - \prod_{1\leq j\leq p}\Big(1 - (1 + C n^{-c_1})\alpha/p \Big)=1 - \Big(1 - (1 + C n^{-c_1})\alpha/p\Big)^p\to 1 - e^{-\alpha}
\end{align*}
since $p = p_n\to\infty$. This gives \eqref{eq: third assertion sn critical value} and completes the proof of the theorem.
\qed

\subsection{Proof of Theorem \ref{thm: analytical rms method}}
We first prove the following technical lemma. Recall that $B_{n} =( \Ep [ \max_{1 \leq j \leq p} Z_{1j}^{4} ])^{1/4}$.
\begin{lemma}
\label{lem: variance}
For every $0 < c < 1$,
\[
\Pr \left ( \max_{1 \leq j \leq p} | \hat{\sigma}_{j}/\sigma_{j} -1 | >  K (n^{-(1-c)/2} B_{n}^{2} \log p + n^{-3/2} B_{n}^{2} \log^{2} p) \right ) \leq K' n^{-c},
\]
where $K,K'$ are universal constants.
\end{lemma}

\begin{proof}
Here $K_{1},K_{2},\dots$ denote universal positive constants.
Note that for $a > 0$, $| \sqrt{a} - 1 | = | a-1 |/(\sqrt{a} + 1) \leq | a-1 |$,
so that for $r>0$,
\[
\Pr \left ( \max_{1 \leq j \leq p} | \hat{\sigma}_{j}/\sigma_{j}-1 | > r \right ) \leq \Pr \left ( \max_{1 \leq j \leq p} | \hat{\sigma}^{2}_{j}/\sigma_{j}^{2}- 1| > r \right ).
\]
Using the expression $\hat{\sigma}_{j}^{2}/\sigma_{j}^{2} -1 = (\En [ Z_{ij}^{2} ] - 1) -( \En [ Z_{ij} ])^{2}$,
we have
\begin{align*}
&\Pr \left ( \max_{1 \leq j \leq p} | \hat{\sigma}_{j}^{2}/\sigma_{j}^{2}- 1| > r  \right) \\
&\quad \leq \Pr \left ( \max_{1 \leq j \leq p} |\En [ Z_{ij}^{2} ] - 1| > r/2 \right ) + \Pr \left ( \max_{1 \leq j \leq p} | \En [ Z_{ij}] | > \sqrt{r/2}  \right ).
\end{align*}
We wish to bound the two terms on the right-hand side by using the Fuk-Nagaev inequality (Lemma \ref{lem: fuk-nagaev}) combined with the maximal inequality in Lemma \ref{lem: maximal ineq}.

By Lemma \ref{lem: maximal ineq} (with the crude bounds $\Ep [ Z_{1j}^{4} ] \leq B_{n}^{4}$ and $\Ep [ \max_{i,j} Z_{ij}^{4} ] \leq nB_{n}^{4}$), we have
\[
\Ep \left [  \max_{1 \leq j \leq p} |\En [ Z_{i j}^{2} ] - 1|  \right ] \leq K_{1} B_{n}^{2} (\log p)/\sqrt{n},
\]
so that by Lemma \ref{lem: fuk-nagaev}, for every $t > 0$,
\[
\Pr \left ( \max_{1 \leq j \leq p} |\En [ Z_{ij}^{2} ] - 1| >  \frac{2K_{1}  B_{n}^{2} \log p}{\sqrt{n}} + t \right ) \leq e^{-nt^{2}/(3B_{n}^{4})} + K_{2}t^{-2} n^{-1} B_{n}^{4}.
\]
Taking $t = n^{-(1-c)/2}B_{n}^{2}$ with $0 < c <1$, the right-hand side becomes
$e^{-n^{c}/3} + K_{2} n^{-c} \leq K_{3} n^{-c}$. Hence we have
\begin{equation}
\Pr \left ( \max_{1 \leq j \leq p} |\En [ Z_{ij}^{2} ] - 1| >  K_{4} n^{-(1-c)/2} B_{n}^{2} (\log p) \right ) \leq K_{3} n^{-c}. \label{var-step1}
\end{equation}

Similarly, using Lemma \ref{lem: maximal ineq}, we have
\begin{equation}\label{eq: exp of maximum of emp proc}
\Ep \left [  \max_{1 \leq j \leq p} |\En [ Z_{ij} ] |  \right ] \leq K_{5} ( n^{-1/2}  \sqrt{\log p} + n^{-3/4} B_{n} \log p),
\end{equation}
so that by Lemma \ref{lem: fuk-nagaev}, for every $t > 0$,
\begin{align*}
&\Pr \left ( \max_{1 \leq j \leq p} | \En [ Z_{ij}] | > 2K_{5} ( n^{-1/2} \sqrt{\log p} + n^{-3/4} B_{n} \log p)+ t \right ) \\
&\quad \leq e^{-nt^{2}/3} + K_{6} t^{-4} n^{-3} B_{n}^{4}.
\end{align*}
Taking $t=n^{-1/4}B_{n}$, the right-hand side becomes
$e^{-n^{1/2}B_{n}/3} + K_{6} n^{-2} \leq K_{7} n^{-2}$.
Hence we have
\begin{equation}
\Pr \left ( \max_{1 \leq j \leq p} | \En [ Z_{ij}] | >  K_{8}  ( n^{-1/4} B_{n} \sqrt{\log p} + n^{-3/4} B_{n} \log p) \right ) \leq K_{7}n^{-2}. \label{var-step2}
\end{equation}
Combining (\ref{var-step1}) and (\ref{var-step2}) leads to the desired result.
\end{proof}

\begin{proof}[Proof of Theorem \ref{thm: analytical rms method}]
Here $c,C$ denote generic positive constants depending only on $\alpha,c_{1},C_{1}$; their values may change from place to place. Define
\begin{equation}
J_{1}=\{j \in \{ 1,\dots,p \} : \sqrt{n}\mu_{j}/\sigma_j>-c^{SN}(\beta_{n})\}, \ J_{1}^{c} = \{ 1,\dots, p\} \backslash J_{1}.  \label{eq: true SN set}
\end{equation}
For $k\geq 1$, let
\[
c^{SN,2S}(\alpha,k)=\frac{\Phi^{-1}(1-(\alpha-2\beta_{n})/k)}{\sqrt{1-\Phi^{-1}(1-(\alpha-2\beta_{n})/k)^2/n}}.
\]
Note that $c^{SN,2S}(\alpha)=c^{SN,2S}(\alpha,\hat{k})$ when $\hat{k} \geq 1$. We divide the proof into several steps.

\medskip

{\bf Step 1}. We wish to prove that with probability larger than $1-\beta_{n} - Cn^{-c}$,
$\hat{\mu}_{j} \leq 0$ for all $j \in J_{1}^{c}$.

\medskip

Observe that
\[
\hat{\mu}_{j} > 0\text{ for some } j \in J_{1}^{c} \Rightarrow \max_{1\leq j\leq p}\sqrt{n}(\hat{\mu}_{j} - \mu_{j})/\sigma_j >  c^{SN}(\beta_{n}),
\]
so that it is enough to prove that
\begin{equation}
\Pr \left (\max_{1 \leq j \leq p} \sqrt{n}(\hat{\mu}_{j} - \mu_{j})/\sigma_j > c^{SN}(\beta_{n})  \right ) \leq \beta_{n} + Cn^{-c}.
\label{step1}
\end{equation}
Since whenever  $\sigma_{j}/\hat{\sigma}_{j} -1\geq -r$ for some $0 < r < 1$,
\[
\sigma_{j} = \hat{\sigma}_{j} ( 1+ (\sigma_{j}/\hat{\sigma}_{j}-1)) \geq \hat{\sigma}_{j} (1-r),
\]
the left-hand side of (\ref{step1}) is bounded by
\begin{align}
&\Pr\left( \max_{1 \leq j \leq p} \sqrt{n} (\hat{\mu}_{j}-\mu_{j})/\hat{\sigma}_{j} > (1-r) c^{SN}(\beta_{n}) \right)  \label{eq: ineq1} \\
&\quad + \Pr \left ( \max_{1 \leq j \leq p} | (\sigma_{j}/\hat{\sigma}_{j})-1 | > r \right ),  \label{ineq2}
\end{align}
where $0 < r < 1$ is arbitrary.

Take $r=r_{n} = n^{-(1-c_{1})/2}B_{n}^{2} \log p$. Then $r_{n} < 1$ for large $n$, and since
\[
|a-1| \leq \frac{r}{r+1}  \Rightarrow | a^{-1} -1 | \leq r,
\]
we see that by Lemma \ref{lem: variance},
the probability in (\ref{ineq2}) is bounded by $Cn^{-c}$.

Consider the probability in (\ref{eq: ineq1}). It is not difficult to see that
\begin{align}
&\Pr\left( \max_{1 \leq j \leq p} \sqrt{n} (\hat{\mu}_{j}-\mu_{j})/\hat{\sigma}_{j} > (1-r) c^{SN}(\beta_{n}) \right) \notag\\
&\quad \leq \Pr \left (\max_{1 \leq j \leq p} U_{j} > (1-r) \Phi^{-1}(1-\beta_{n}/p) \right ) \notag\\
& \quad\leq \sum_{j=1}^{p}  \Pr \left ( U_{j} > (1-r) \Phi^{-1}(1-\beta_{n}/p) \right ).\label{eq: ineq3}
\end{align}
Note that $(1-r)\Phi^{-1}(1-\beta_{n}/p) \leq \sqrt{2 \log (p/\beta_{n})} \leq n^{1/6}/M_{n,3}$ for large $n$. Hence,
by  Lemma \ref{lem: moderate deviations}, the sum in (\ref{eq: ineq3}) is bounded by
\begin{align*}
&p\overline{\Phi} \left ((1-r)\Phi^{-1}(1-\beta_{n}/p) \right ) \left [ 1+ n^{-1/2}C M_{n,3}^{3} \left \{ 1+(1-r)\Phi^{-1}(1-\beta_{n}/p)  \right \}^{3} \right ] \\
&\quad\leq p\overline{\Phi} \left ((1-r)\Phi^{-1}(1-\beta_{n}/p) \right )  \left [ 1+ n^{-1/2}C M_{n,3}^{3} \{1+\Phi^{-1}(1-\beta_{n}/p) \}^{3} \right ].
\end{align*}
Observe that $n^{-1/2} M_{n,3}^{3} \{1+\Phi^{-1}(1-\beta_{n}/p) \}^{3} \leq C n^{-c_{1}}$. Moreover, putting $\xi = \Phi^{-1}(1-\beta_{n}/p)$, we have by Taylor's expansion for some $r'\in[0,r]$,
\begin{align*}
 p \overline{\Phi} \left ((1-r)\xi \right ) &= \beta_{n} + r p \xi \phi \left ((1-r')\xi \right ) \leq \beta_{n} + r p\xi \phi \left ((1-r)\xi \right ).
 \end{align*}
 Using the inequality $(1-r)^{2} \xi^{2} = \xi^{2} + r^{2}\xi^{2} - 2r\xi^{2} \geq \xi^{2} - 2r\xi^{2}$,
 we have $ \phi \left ((1-r)\xi \right ) \leq e^{ r \xi^{2} } \phi (\xi)$.
 Since $\beta_{n} < \alpha/2 < 1/4$ and $p \geq 2$, we have $\xi \geq \Phi^{-1}(1-1/8) > 1$,
so that by Proposition 2.1 in \cite{Dudley02},  we have $\phi (\xi) \leq 2 \xi (1-\Phi (\xi)) =2 \xi \beta_{n}/p$.\footnote{Note that the second part of Proposition 2.1 in \cite{Dudley02} asserts that  $\phi(t)/t \leq \Pr ( | N(0,1) | > t) = 2 (1-\Phi (t))$ when $t \geq 1$, so that $\phi (t) \leq 2 t(1-\Phi(t))$.} Hence
\[
 p \overline{\Phi} \left ((1-r)\xi \right ) \leq \beta_{n} (1+2r\xi^{2} e^{r \xi^{2}}).
\]
Recall that we have taken $r=r_{n}= n^{-(1-c_{1})/2}B_{n}^{2} \log p$, so that
\[
r\xi^{2}  \leq 2 n^{-(1-c_{1})/2} B_{n}^{2} \log^{2}( p / \beta_{n}) \leq C n^{-c_{1}/2}.
\]
Therefore, the probability in (\ref{eq: ineq1}) is bounded by $\beta_{n} + Cn^{-c}$ for large $n$.
The conclusion of Step 1 is verified for large $n$ and hence for all $n$ by adjusting the constant $C$.

\medskip

{\bf Step 2}. We wish to prove that with probability larger than $1-\beta_{n} - Cn^{-c}$, $\hat{J}_{SN}  \supset J_{1}$.

\medskip

Observe that
\begin{equation}\label{ineq}
\Pr (\hat{J}_{SN} \not \supset J_{1}) \leq \Pr\left(\max_{1\leq j\leq p}\left[\sqrt{n}(\mu_{j}-\hat{\mu}_{j})-(2\hat{\sigma}_{j}-\sigma_{j})c^{SN}(\beta_{n})\right] > 0 \right).
\end{equation}
Since whenever $1-\sigma_{j}/\hat{\sigma}_{j} \geq -r$ for some $0 < r < 1$,
\[
2\hat{\sigma}_j-\sigma_j = \hat{\sigma}_{j} (1 + (1-\sigma_{j}/\hat{\sigma}_{j})) \geq \hat{\sigma}_{j} ( 1-r ),
\]
the right-hand side on (\ref{ineq}) is bounded by
\begin{align*}
&\Pr\left( \max_{1 \leq j \leq p} \sqrt{n} (\mu_{j}-\hat{\mu}_{j})/\hat{\sigma}_{j} > (1-r) c^{SN}(\beta_{n}) \right)  \notag \\
&\quad + \Pr \left ( \max_{1 \leq j \leq p} | (\sigma_{j}/\hat{\sigma}_{j})-1 | > r \right ),
\end{align*}
where $0 < r < 1$ is arbitrary.
By the proof of Step 1, we see that the sum of these terms is bounded by $\beta_{n}+Cn^{-c}$ with suitable $r$,
which leads to the conclusion of Step 2.

\medskip

{\bf Step 3}. We are now in position to prove \eqref{eq: thm 42 first bound}. Consider first the case where $J_{1} = \emptyset$. Then by Step 1, with probability larger than $1-\beta_{n}-C n^{-c}$, $T \leq 0$, so that
$$
\Pr (T > c^{SN,2S}(\alpha) ) \leq \beta_{n} + C n^{-c} \leq \alpha + C n^{-c}.
$$
Suppose now that $| J_{1} | \geq 1$.
Observe that
\[
\{ T > c^{SN,2S}(\alpha) \} \cap \left\{ \max_{j\in J_1^c}\hat{\mu}_{j} \leq 0 \right\}
\subset \left\{ \max_{j \in J_{1}} \sqrt{n}\hat{\mu}_{j}/\hat{\sigma}_{j} >  c^{SN,2S}(\alpha)  \right\}.
\]
Moreover, as $c^{SN,2S}(\alpha,k)$ is non-decreasing in $k$,
\begin{multline*}
\left\{ \max_{j \in J_{1}} \sqrt{n}\hat{\mu}_{j}/\hat{\sigma}_{j} >  c^{SN,2S}(\alpha)  \right\} \cap \{ \hat{J}_{SN} \supset J_{1} \} \\
\subset \left\{ \max_{j \in J_{1}} \sqrt{n}\hat{\mu}_{j}/\hat{\sigma}_{j} >  c^{SN,2S}(\alpha, | J_{1} |)  \right\}.
\end{multline*}
Therefore, by Steps 1 and 2, we have
\begin{align}
&\Pr(T>c^{SN,2S}(\alpha)) \notag \\
&\quad \leq \Pr\left(\max_{j\in J_{1}}\sqrt{n}\hat{\mu}_{j}/\hat{\sigma}_{j} > c^{SN,2S}(\alpha,|J_{1}|)\right)+2\beta_{n}+C n^{-c} \notag \\
&\quad \leq \Pr\left(\max_{j\in J_{1}}\sqrt{n}(\hat{\mu}_{j}-\mu_{j})/ \hat{\sigma}_{j} > c^{SN,2S}(\alpha,|J_{1}|)\right)+2\beta_{n}+C n^{-c}. \label{step3-1}
\end{align}
By Theorem \ref{thm: analytical plugin method}, we see that
\begin{equation}
\Pr\left(\max_{j\in J_{1}}\sqrt{n}(\hat{\mu}_{j}-\mu_{j})/\hat{\sigma}_{j} > c^{SN,2S}(\alpha,|J_{1}|)\right) \leq \alpha - 2\beta_{n} + Cn^{-c}, \label{step3-2}
\end{equation}
where the condition \eqref{eq: cond1} of Theorem \ref{thm: analytical plugin method},
$$
M_{n,3}^{3} \log^{3/2} (p/\alpha) \leq C_{1}n^{1/2-c_{1}},
$$
is now replaced by
$$
M_{n,3}^{3} \log^{3/2} (p/(\alpha - 2\beta_n)) \leq C_{1}n^{1/2-c_{1}},
$$
which is assumed in \eqref{eq: cond2}. Combining (\ref{step3-1}) and (\ref{step3-2}) gives \eqref{eq: thm 42 first bound}.

\medskip

{\bf Step 4}. Finally, we prove \eqref{eq: thm42 additional bound}. Since $\mu_j = 0$ for all $j=1,\dots,p$, it follows that $J_1 = \{1,\dots,p\}$, and so by Step 2, $\hat k = p$ and $c^{SN,2S}(\alpha) = c^{SN,2S}(\alpha,p) = c^{SN}(\alpha - 2\beta_n)$ with probability larger than $1 - \beta_n - C n^{-c} = 1 - o(1)$ since $\beta_n\to0$. Therefore,
\begin{align*}
\Pr(T > c^{SN,2S}(\alpha)) 
&= \Pr\left(\max_{1\leq j\leq p}\sqrt n(\hat\mu_j - \mu_j)/\hat\sigma_j > c^{SN,2S}(\alpha)\right)\\
&= \Pr\left(\max_{1\leq j\leq p}\sqrt n(\hat\mu_j - \mu_j)/\hat\sigma_j > c^{SN}(\alpha - 2\beta_n)\right) + o(1)\\
&= 1 - e^{-(\alpha - 2\beta_n)} + o(1) \to 1 - e^{-\alpha}
\end{align*}
as in the proof of Theorem \ref{thm: analytical plugin method}. This completes the proof of the theorem.
\end{proof}

\subsection{Proof of Theorem \ref{thm: simulation plugin method}}
Here $c,C$ denote generic positive constants depending only on $c_{1},C_{1}$; their values may change from place to place. Let $W$ stand for $W^{MB}$ or $W^{EB}$, depending on which bootstrap procedure is used.
Define
\[
\bar{T}:=\max_{1\leq j\leq p}\frac{\sqrt{n}(\hat{\mu}_j-\mu_{j})}{\hat{\sigma}_j}, \ \text{and} \ T_0:=\max_{1\leq j\leq p}\frac{\sqrt{n}(\hat{\mu}_j-\mu_{j})}{\sigma_j}.
\]
In addition, define
\[
\bar{W}^{MB}:=\max_{1\leq j\leq p}\frac{\sqrt{n}\En[\epsilon_{i}(X_{ij}-\hat{\mu}_j)]}{\sigma_j}, \ \bar{W}^{EB}:=\max_{1\leq j\leq p}\frac{\sqrt{n}\En[(X_{i j}^*-\hat{\mu}_j)]}{\sigma_j},
\]
and let $\bar{W}$ stand for $\bar{W}^{MB}$ or $\bar{W}^{EB}$ depending on which bootstrap procedure is used.
Further, let
\[
(Y_1,\dots,Y_p)^T\sim N(0, \Ep[Z_1 Z_1^T])
\]
and for $\gamma\in(0,1)$, denote by $c_0(\gamma)$ the $(1-\gamma)$-quantile of the distribution of $\max_{1\leq j\leq p}Y_j$. Finally, define
\begin{align*}
&\rho_n:=\sup_{t\in\mathbb{R}}\left|\Pr(T_0\leq t)-\Pr\left(\max_{1\leq j\leq p}Y_j\leq t\right)\right|,\\
&\rho_n^B:=\sup_{t\in\mathbb{R}}\left|\Pr(\bar{W}\leq t\mid X_1^n)-\Pr\left(\max_{1\leq j\leq p}Y_j\leq t\right)\right|.
\end{align*}
Observe that under the present assumptions, we may apply Proposition 2.1 in \cite{CCK14} so that we have
\begin{equation}\label{eq: rho bound induction}
\rho_n\leq C n^{-c}; 
\end{equation}
while applying Corollary 4.2 and Proposition 4.3 in \cite{CCK14} to the MB and EB procedures, respectively, we have
for some  $\nu_n:=C n^{-c}$,
\begin{equation}\label{eq: rho bootstrap bound}
\Pr(\rho_n^B < \nu_n)\geq 1-C n^{-c}.
\end{equation}
We divide the rest of  the proof into three steps. Step 1 establishes a relation between $c^B(\cdot)$ and $c_0(\cdot)$. Step 2 proves the assertion of the theorem. Step 3 provides auxiliary calculations. In particular, Step 3 shows that for some $\zeta_{n1}$ and $\zeta_{n2}$ satisfying $\zeta_{n1}\sqrt{\log p}+\zeta_{n2}\leq C n^{-c}$, we have
\begin{align}
&\Pr(|\bar{T}-T_0|>\zeta_{n1})\leq C n^{-c},\label{eq: T approximation 1}\\
&\Pr(\Pr(|W-\bar{W}|>\zeta_{n1}\mid X_1^n)>\zeta_{n2})\leq C n^{-c}. \label{eq: W approximation 1}
\end{align}

\medskip

{\bf Step 1}.  We wish to prove that
\begin{align}
&\Pr(c^B(\alpha)\geq c_0(\alpha+\zeta_{n2}+\nu_n+8\zeta_{n1}\sqrt{\log p}))\geq 1 - C n^{-c},\label{eq: critical value lower bound 1}\\
&\Pr(c^B(\alpha)\leq c_0(\alpha-\zeta_{n2}-\nu_n-8\zeta_{n1}\sqrt{\log p}))\geq 1 - C n^{-c}.\label{eq: critical value upper bound 1}
\end{align}

To establish (\ref{eq: critical value lower bound 1}), observe that for any $t\in \R$,
\begin{align}
&\Pr(W\leq t\mid X_1^n)\leq \Pr(\bar{W}\leq t + \zeta_{n1} \mid X_1^n) + \Pr(|W-\bar{W}|>\zeta_{n1}\mid X_1^n)\label{eq: raw 1a}\\
&\quad  \leq \Pr\left(\max_{1\leq j\leq p}Y_j\leq t + \zeta_{n1}\right) + \rho_n^B + \Pr(|W-\bar{W}|>\zeta_{n1}\mid X_1^n). \label{eq: raw 1b}
\end{align}
By Lemma \ref{lem: critical value bound}, for any $\gamma\in(0,1-8\zeta_{n1}\sqrt{\log p})$ (note that $1-8\zeta_{n1}\sqrt{\log p}>0$ for sufficiently large $n$),
\begin{align*}
&\Pr\left(\max_{1\leq j\leq p}Y_j\leq c_0(\gamma+8\zeta_{n1}\sqrt{\log p})+\zeta_{n1}\right)\\
&\quad \leq \Pr\left(\max_{1\leq j\leq p}Y_j\leq c_0(\gamma+8\zeta_{n1}\sqrt{\log p})\right) + 2\zeta_{n1}(\sqrt{2\log p}+1)\\
&\quad \leq \Pr\left(\max_{1\leq j\leq p}Y_j\leq c_0(\gamma+8\zeta_{n1}\sqrt{\log p})\right) + 8\zeta_{n1}\sqrt{\log p}\\
&\quad = 1 - \gamma - 8\zeta_{n1}\sqrt{\log p} + 8\zeta_{n1}\sqrt{\log p} = 1-\gamma,
\end{align*}
where the third line follows from $p\geq 2$, so that $\sqrt{2\log p}\geq 1$, and the fourth line from the fact that the distribution of $\max_{1\leq j\leq p}Y_j$ has no point masses. 
Hence
\begin{equation}\label{eq: comparison of critical values}
c_0(\gamma+8\zeta_{n1}\sqrt{\log p})+\zeta_{n1}\leq c_0(\gamma).
\end{equation}
Therefore, setting $t=c_0(\alpha+\zeta_{n2}+\nu_n+8\zeta_{n1}\sqrt{\log p})$ in (\ref{eq: raw 1a})-(\ref{eq: raw 1b}), we obtain 
\begin{align*}
&\Pr(W\leq c_0(\alpha+\zeta_{n2}+\nu_n+8\zeta_{n1}\sqrt{\log p})\mid X_1^n)\\
&\quad \leq 1-\alpha-\zeta_{n2}-\nu_n+\rho_n^B+\Pr(|W-\bar{W}|>\zeta_{n1}\mid X_1^n)<1-\alpha
\end{align*}
on the event that $\rho_n^B < \nu_n$ and $\Pr(|W-\bar{W}|>\zeta_{n1}\mid X_1^n)\leq \zeta_{n2}$, which holds with probability larger than $1-C n^{-c}$  by (\ref{eq: rho bootstrap bound}) and (\ref{eq: W approximation 1}). This implies (\ref{eq: critical value lower bound 1}). By a similar argument, we can establish that (\ref{eq: critical value upper bound 1}) holds as well. This completes Step 1.

\medskip

{\bf Step 2}. Here we prove the asserted claims. 
Observe that under $H_0$,
\begin{align*}
&\Pr(T>c^B(\alpha))\leq \Pr(\bar{T}>c^B(\alpha))\\
&\quad \leq\Pr(T_0>c^B(\alpha)-\zeta_{n1})+\Pr(|\bar{T}-T_0|>\zeta_{n1})\\
&\quad \leq \Pr(T_0>c_0(\alpha+\zeta_{n2}+\nu_n+8\zeta_{n1}\sqrt{\log p})-\zeta_{n1})+C n^{-c}\\
&\quad \leq \Pr(T_0>c_0(\alpha+\zeta_{n2}+\nu_n+16\zeta_{n1}\sqrt{\log p}))+C n^{-c}\\
&\quad \leq \Pr(\max_{1\leq j\leq p}Y_j>c_0(\alpha+\zeta_{n2}+\nu_n+16\zeta_{n1}\sqrt{\log p}))+\rho_n+C n^{-c}\\
&\quad = \alpha+\zeta_{n2}+\nu_n+16\zeta_{n1}\sqrt{\log p}+\rho_n+C n^{-c}\leq \alpha + C n^{-c},
\end{align*}
where the third line follows from (\ref{eq: T approximation 1}) and (\ref{eq: critical value lower bound 1}), the fourth line from (\ref{eq: comparison of critical values}), and the last line from (\ref{eq: rho bound induction}) and construction of $\nu_n$, $\zeta_{n1}$, and $\zeta_{n2}$. Hence, (\ref{eq: simulation plugin control}) follows. To prove (\ref{eq: simulation plugin control2}), 
observe that when $\mu_j=0$ for all $1\leq j \leq p$, $T=\bar{T}$, and so
\begin{align*}
&\Pr(T>c^B(\alpha))= \Pr(\bar{T}>c^B(\alpha))\\
&\quad \geq\Pr(T_0>c^B(\alpha)+\zeta_{n1})-\Pr(|\bar{T}-T_0|>\zeta_{n1})\\
&\quad \geq \Pr(T_0>c_0(\alpha-\zeta_{n2}-\nu_n-8\zeta_{n1}\sqrt{\log p})+\zeta_{n1})-C n^{-c}\\
&\quad \geq \Pr(T_0>c_0(\alpha-\zeta_{n2}-\nu_n-16\zeta_{n1}\sqrt{\log p}))-C n^{-c}\\
&\quad \geq \Pr(\max_{1\leq j\leq p}Y_j>c_0(\alpha-\zeta_{n2}-\nu_n-16\zeta_{n1}\sqrt{\log p}))-\rho_n-C n^{-c}\\
&\quad = \alpha-\zeta_{n2}-\nu_n-16\zeta_{n1}\sqrt{\log p}-\rho_n-C n^{-c}\geq \alpha - C n^{-c},
\end{align*}
where the third line follows from (\ref{eq: T approximation 1}) and (\ref{eq: critical value upper bound 1}), the fourth line from (\ref{eq: comparison of critical values}), and the equality in the last line from the fact that the distribution of $\max_{1\leq j\leq p}Y_j$ has no point masses. Hence (\ref{eq: simulation plugin control2}) follows. This completes Step 2.

\medskip

{\bf Step 3}. We wish to prove (\ref{eq: T approximation 1}) and (\ref{eq: W approximation 1}).
We wish to verify these conditions with
\[
\zeta_{n1}:=n^{-(1-c_{1})/2} B_{n}^{2} \log^{3/2} p,  \ \text{and} \  \zeta_{n2}:=C'n^{-c'},
\]
where $c', C'$ are suitable positive constants that depend only on $c_{1},C_{1}$.
We note that because of the assumption that $B_{n}^{2} \log^{7/2} (p n) \leq C_{1}n ^{1/2-c_{1}}$, these choices satisfy $\zeta_{n1}\sqrt{\log p}+\zeta_{n2}\leq Cn^{-c}$.

We first verify (\ref{eq: T approximation 1}). Observe that
\[
| \bar{T}-T_{0} | \leq \max_{1 \leq j \leq p} | (\sigma_{j}/\hat{\sigma}_{j}) - 1 | \times \max_{1 \leq j \leq p} | \sqrt{n} \En[ Z_{ij}] |.
\]
By Lemma \ref{lem: variance} and the simple fact that $|a-1| \leq r/(r+1) \Rightarrow |a^{-1}-1| \leq r$ ($r > 0$), we have
\begin{equation}\label{eq: sigma estimation bound}
\Pr \left ( \max_{1 \leq j \leq p} | (\sigma_{j}/\hat{\sigma}_{j}) - 1 | > n^{-1/2+c_{1}/4} B_{n}^{2}\log p \right ) \leq Cn^{-c}.
\end{equation}
Moreover, by Markov's inequality and (\ref{eq: exp of maximum of emp proc}),
\[
\Pr \left( \max_{1 \leq j \leq p} | \sqrt{n} \En[ Z_{ij}] | > n^{c_{1}/4}  \sqrt{\log p} \right) \leq Cn^{-c}.
\]
Hence (\ref{eq: T approximation 1}) is verified (note that $n^{-1/2+c_{1}/4} B_{n}^{2}(\log p) \times n^{c_{1}/4} \sqrt{\log p}=\zeta_{n1}$).

To verify (\ref{eq: W approximation 1}), let $A_{n}$ be the event such that
\[
A_{n} := \left \{ \max_{1 \leq j \leq p} | (\hat{\sigma}_{j}/\sigma_{j}) - 1 | \leq (n^{-1/2+c_{1}/4} B_{n}^{2}\log p) \wedge (1/4) \right \}.
\]
We have seen that $\Pr (A_{n}) > 1-Cn^{-c}$. We consider MB and EB procedures separately. 

Consider the MB procedure first, so that $W=W^{MB}$ and $\bar{W}=\bar{W}^{MB}$.
Observe that
\[
| W^{MB} - \bar{W}^{MB} | \leq \max_{1 \leq j \leq p} | (\hat{\sigma}_{j}/\sigma_{j}) - 1 | \times | W^{MB} |.
\]
Conditional on the data $X_{1}^{n}$, the vector $(\sqrt{n} \En[ \epsilon_{i}(X_{ij}-\hat{\mu}_{j})/\hat{\sigma}_{j} ])_{1 \leq j \leq p}$ is normal with mean zero and all the diagonal elements of the covariance matrix are one. Hence $\Ep [ | W^{MB} | \mid X_{1}^{n} ] \leq  \sqrt{2\log (2p)}$,
so that by Markov's inequality, on the event $A_{n}$,
\[
\Pr ( | W^{MB}-\bar{W}^{MB} | > \zeta_{n1} \mid X_{1}^{n} ) \leq (1/\zeta_{n1}) \max_{1 \leq j \leq p} | (\hat{\sigma}_{j}/\sigma_{j})-1 | \times \Ep[|W^{MB}| \mid X_{1}^{n}],
\]
which is bounded by $Cn^{-c_{1}/4}$, so that (\ref{eq: W approximation 1}) for the MB procedure is verified.

Now consider the EB procedure. On the event $A_{n} \cap \{ \Pr(|W^{MB}-\bar{W}^{MB}|>\zeta_{n1}\mid X_1^n)\leq \zeta_{n2} \} \cap \{ \rho_n^{MB}<\nu_n \} \cap \{ \rho_n^{EB}<\nu_n \}$, which holds with probability larger than $1- C n^{-c}$,
\begin{align*}
&\Pr(|W^{EB}-\bar{W}^{EB}|>\zeta_{n1}\mid X_1^n)\\
&\quad \leq \Pr(\max_{1\leq j\leq p}|(\sigma_j/\hat{\sigma}_j)-1|\times|\bar{W}^{EB}|>\zeta_{n1}\mid X_1^n)\\
&\quad \leq \Pr(\max_{1\leq j\leq p}|(\sigma_j/\hat{\sigma}_j)-1|\times|\bar{W}^{MB}|>\zeta_{n1}\mid X_1^n)+\rho_n^{EB}+\rho_n^{MB}\\
&\quad \leq \Pr(\max_{1\leq j\leq p}|(\hat{\sigma}_j/\sigma_j)-1|\times|W^{MB}|>\zeta_{n1}/4\mid X_1^n)+\rho_n^{EB}+\rho_n^{MB}\leq C n^{-c},
\end{align*}
so that (\ref{eq: W approximation 1}) for the EB procedure is verified.
This completes the proof.
\qed

\subsection{Proof of Theorem \ref{thm: simulation MS method}}
Here $c,C$ denote generic positive constants depending only on $c_{1},C_{1}$; their values may change from place to place. Let $\hat{J}_B$ stand either for $\hat{J}_{MB}$ or $\hat{J}_{EB}$ depending on which bootstrap procedure is used.
Let
\[
(Y_{1},\dots,Y_{p})^T\sim N(0,\Ep[Z_{1}Z_{1}^T]).
\]
For $\gamma\in(0,1)$, denote by $c_0(\gamma)$  the $(1-\gamma)$-quantile of the distribution of $\max_{1\leq j\leq p}Y_j$. Recall that in the proof of Theorem \ref{thm: simulation plugin method}, we established that with probability larger than $1 - C n^{-c}$, $c^B(\alpha)\geq c_0(\alpha+\bar\varphi_n)$ and $c^B(\alpha)\leq c_0(\alpha-\bar\varphi_n)$ for some $0<\bar\varphi_n\leq C n^{-c}$; see (\ref{eq: critical value lower bound 1}) and (\ref{eq: critical value upper bound 1}).  Define
\[
J_{2} :=\{j\in\{1,\dots,p\}:\sqrt{n}\mu_j/\sigma_j>-c_0(\beta_n+\bar\varphi_n)\}, \, J_{2}^c=\{1,\dots,p\}\backslash J_{2}.
\]
We divide the proof into several steps.

\medskip

{\bf Step 1}. 
We wish to prove that with probability larger than $1-\beta_{n} - Cn^{-c}$, $\hat{\mu}_j\leq 0$ for all $j\in J_2^c$.

\medskip
Like in the proof of Theorem \ref{thm: analytical rms method}, observe that
\[
\hat{\mu}_{j} > 0\text{ for some }j \in J_{2}^{c} \Rightarrow \max_{1\leq j\leq p}\sqrt{n}(\hat{\mu}_{j} - \mu_{j})/\sigma_j >  c_0(\beta_{n}+\bar\varphi_n),
\]
so that it is enough to prove that
\[
\Pr \left (\max_{1 \leq j \leq p}\frac{ \sqrt{n}(\hat{\mu}_{j} - \mu_{j})}{\sigma_j} > c_0(\beta_{n}+\bar\varphi_n)  \right ) \leq \beta_{n} + Cn^{-c}.
\]
But this follows from Proposition 2.1 in \cite{CCK14} (and the fact that $\bar\varphi_{n}\leq Cn^{-c}$). This concludes Step 1.
\medskip

{\bf Step 2}. We wish to prove that with probability larger than $1-\beta_{n} - Cn^{-c}$, $\hat{J}_{B}\supset J_{2}$.

\medskip
Like in the proof of Theorem \ref{thm: analytical rms method}, observe that
\begin{align*}
&\Pr (\hat{J}_{B} \not \supset J_{2}) \\
&\quad \leq \Pr \left( \max_{1\leq j\leq p}\left[\sqrt{n}(\mu_{j}-\hat{\mu}_{j})-(2\hat{\sigma}_{j}c^{B}(\beta_n)-\sigma_{j}c_0(\beta_n+\bar\varphi_n) ) \right] > 0 \right ).
\end{align*}
Since whenever $c^{B}(\beta_{n}) \geq c_{0}(\beta_{n}+\bar\varphi_{n})$ and $\hat{\sigma}_{j}/\sigma_{j} - 1 \geq -r/2$ for some $r > 0$,
\begin{align*}
&2\hat{\sigma}_{j}c^{B}(\beta_n)-\sigma_{j}c_0(\beta_n+\bar\varphi_n) \geq (2 \hat{\sigma}_{j} - \sigma_{j}) c_{0}(\beta_{n}+\bar\varphi_{n}) \\
&\quad =\sigma_{j}(1+ 2(\hat{\sigma}_{j}/\sigma_{j} - 1)) c_{0}(\beta_{n}+\bar\varphi_{n}) \geq (1-r) \sigma_{j} c_{0}(\beta_{n}+\bar\varphi_{n}),
\end{align*}
we have
\begin{align}
\Pr (\hat{J}_{B} \not \supset J_{2}) &\leq \Pr \left (\max_{1 \leq j \leq p} \frac{\sqrt{n}(\mu_{j}-\hat{\mu}_{j})}{\sigma_{j}} > (1-r) c_{0}(\beta_{n}+\bar\varphi_{n}) \right ) \label{eq: ineq4}\\
&\quad+ \Pr \left ( c^{B}(\beta_{n}) < c_{0}(\beta_{n}+\bar\varphi_{n}) \right ) + \Pr \left  ( \max_{1 \leq j \leq p} | (\hat{\sigma}_{j}/\sigma_{j}) - 1 | > r/2 \right ).\notag
\end{align}
By Proposition 2.1 in \cite{CCK14}, the probability on the right-hand side of (\ref{eq: ineq4}) is bounded by
\[
\Pr \left(\max_{1 \leq j \leq p} Y_{j} > (1-r) c_{0}(\beta_{n}+\bar\varphi_{n})\right) + Cn^{-c}.
\]
Moreover, by Lemma \ref{lem: critical value bound},
\begin{multline*}
\Pr \left(\max_{1 \leq j \leq p} Y_{j} > (1-r) c_{0}(\beta_{n}+\bar\varphi_{n})\right) \\
\leq \beta_{n} + \bar\varphi_{n} + 2r \left ( \sqrt{2 \log p} + 1) (\sqrt{2 \log p} + \sqrt{2\log (1/(\beta_{n}+\bar\varphi_{n}))} \right ),
\end{multline*}
which is bounded by $\beta_{n} + \bar\varphi_{n} + C r\log (p n)$.
Thus,
\[
\Pr (\hat{J}_{B} \not \supset J_{2}) \leq \beta_{n} + \Pr \left  ( \max_{1 \leq j \leq p} | (\hat{\sigma}_{j}/\sigma_{j}) - 1 | > r/2 \right ) +  C (r \log (pn) + n^{-c}).
\]
Choosing $r=r_{n}=n^{-(1-c_{1})/2} B_{n}^{2} \log p$, we see that, by Lemma \ref{lem: variance}, the second term on the right-hand side of the inequality above is bounded by $Cn^{-c}$, and
\[
r \log (pn) \leq n^{-(1-c_{1})/2} B_{n}^{2} \log^{2}(pn) \leq C_{1}n^{-c_{1}/2},
\]
because of the assumption that $B_{n}^{2} \log^{7/2}(pn) \leq C_{1}n^{1/2-c_{1}}$. This leads to the conclusion of Step 2.
\medskip

{\bf Step 3}. We  are now in position to finish the proof of the theorem. Assume first that $J_{2}=\emptyset$. Then by Step 1 we have that $T\leq 0$ with probability larger than $1-\beta_n-Cn^{-c}$. But as $c^{B,2S}(\alpha)\geq 0$ (recall that $\alpha<1/2$), we have $\Pr (T > c^{B,2S}(\alpha) ) \leq \beta_{n} + Cn^{-c} \leq \alpha + Cn^{-c}$. Now consider the case where $J_{2} \neq \emptyset$. Define $c^{B,2S}(\alpha,J_{2})$ by the same bootstrap procedure as $c^{B,2S}(\alpha)$ with $\hat{J}_{B}$ replaced by $J_{2}$. Note that $c^{B,2S}(\alpha) \geq c^{B,2S}(\alpha,J_{2})$ on the event $\hat{J}_{B} \supset J_{2}$. Therefore, arguing as in Step 3 of the proof of Theorem \ref{thm: analytical rms method},
\begin{align*}
&\Pr(T>c^{B,2S}(\alpha)) \leq \Pr\left(\max_{j \in J_{2}}\sqrt{n}\hat{\mu}_j/\hat{\sigma}_j>c^{B,2S}(\alpha)\right)+\beta_n+Cn^{-c} \\
&\qquad \leq \Pr\left(\max_{j \in J_{2}}\sqrt{n}\hat{\mu}_{j}/\hat{\sigma}_{j}>c^{B,2S}(\alpha,J_{2})\right)+2\beta_{n}+Cn^{-c} \\
&\qquad \leq \Pr\left(\max_{j \in J_{2}}\sqrt{n}(\hat{\mu}_{j}-\mu_{j})/\hat{\sigma}_{j}>c^{B,2S}(\alpha,J_{2})\right)+2\beta_{n}+Cn^{-c} \\
&\qquad \leq \alpha-2\beta_n+2\beta_n+Cn^{-c}=\alpha + Cn^{-c}.
\end{align*}
This gives the first assertion of the theorem.

Moreover,  when $\mu_j=0$ for all $1 \leq j \leq p$, we have $J_{2}=\{1,\dots,p\}$. Hence by Step 2, $c^{B,2S}(\alpha)=c^{B,2S}(\alpha,J_{2})$ with probability larger than $1-\beta_n-Cn^{-c}$. Therefore,
\begin{align*}
&\Pr(T>c^{B,2S}(\alpha))=\Pr\left(\max_{1\leq j\leq p}\sqrt{n}(\hat{\mu}_{j}-\mu_{j})/\hat{\sigma}_{j}>c^{B,2S}(\alpha)\right)\\
&\qquad \geq \Pr\left(\max_{1\leq j\leq p}\sqrt{n}(\hat{\mu}_{j}-\mu_{j})/\hat{\sigma}_{j}>c^{B,2S}(\alpha,J_{2})\right)-\beta_{n}-Cn^{-c}\\
&\qquad \geq \alpha-3\beta_n-Cn^{-c}.
\end{align*}
This gives the second assertion of the theorem. Finally, the last assertion follows trivially.
This completes the proof of the theorem.
\qed

\subsection{Proof of Theorem \ref{thm: HB method}}

Recall the set $J_{1} \subset \{ 1,\dots, p \}$ defined in (\ref{eq: true SN set}). By Steps 1 and 2 in the proof of Theorem \ref{thm: analytical rms method}, we see that
\begin{align*}
&\Pr (\hat{\mu}_{j} \leq 0\text{ for all } j \in J_{1}^{c}) > 1-\beta_{n}-C n^{-c},\\
&\Pr (\hat{J}_{SN} \supset J_{1}) > 1-\beta_{n}-Cn^{-c},
\end{align*}
where $c,C$ are some positive constants depending only on $c_{1},C_{1}$. The rest of the proof is completely analogous to Step 3 in the proof of Theorem \ref{thm: simulation MS method} and hence omitted. \qed

\subsection{Proof of Theorem \ref{thm: weak inequalities}}
Here $c,C$ denote generic positive constants depending only on $c_1,C_1,c_2,C_2$; their values may change from place to place.
Define
\begin{align*}
&J_2:=\left\{j\in\{1,\dots,p\}:\sqrt{n}\mu_{j}/\sigma_j>-c_0(\beta_n+\bar\varphi_n)\right\},\, J_2^c:=\{1,\dots,p\}\backslash J_2,\\
&J_3:=\left\{j\in\{1,\dots,p\}:\sqrt{n}|\mu_{j l}^V/\sigma_{j l}^V|>2c_0^V(\beta_n)\text{ for some }l=1,\dots,r\right\}
\end{align*}
where $c_0(\beta_n+\bar\varphi_n)$ is defined as in the proof of Theorem \ref{thm: simulation MS method} and $c_0^V(\beta_n)$ is the $(1-\beta_n)$-quantile of the distribution of $\max_{j,l}Y_{j l}^V$ where $\{Y_{j l}^V,1\leq j\leq p,1\leq l\leq r\}$ is a sequence of Gaussian random variables with mean zero and covariance $\Ep[Y_{j l}^V Y_{j' l'}^V]=\Ep[Z_{1 j l}^V Z_{1 j' l'}^V]$. 

By the same arguments as those used in Steps 1 and 2 of the proof of Theorem \ref{thm: simulation MS method}, we have 
\begin{align*}
&\Pr(J_2\subset \hat{J}_B)\geq 1 - \beta_n - C n^{-c},\\
&\Pr(J_3\subset \hat{J}_B'')\geq 1- \beta_n - C n^{-c},\\
&\Pr(\hat{J}_B'\subset J_3)\geq 1 - \beta_n - C n^{-c},\\
&\Pr(\hat{\mu}_j\leq 0,\text{ for all }j \in J_2^c)\geq 1-\beta_n-C n^{-c}.
\end{align*}
Define $c^{B,3S}(\alpha,J_2\cap J_3)$ by the same bootstrap procedure as $c^{B,3S}(\alpha)$ with $\hat{J}_{B}\cap\hat{J}_{B}''$ replaced by $J_2\cap J_3$. Then inequalities above imply that $c^{B,3S}(\alpha,J_2\cap J_3)\leq c^{B,3S}(\alpha)$ with probability larger than $1 - 2\beta_n - C n^{-c}$. 
Therefore, by an argument similar to that used in Step 3 of the proof of Theorem  \ref{thm: simulation MS method}, with maximum over empty set understood as 0, we have
\begin{align*}
&\Pr(T>c^{B,3S}(\alpha))\leq \Pr\left(\max_{j\in J_2\cap \hat{J}_{B}'}\sqrt{n}\hat{\mu}_j/\hat{\sigma}_j>c^{B,3S}(\alpha)\right)+\beta_n+C n^{-c}\\
&\quad \leq \Pr\left(\max_{j\in J_2\cap \hat{J}_{B}'}\sqrt{n}\hat{\mu}_j/\hat{\sigma}_j>c^{B,3S}(\alpha, J_2\cap J_3)\right)+3\beta_n+C n^{-c}\\
&\quad \leq \Pr\left(\max_{j\in J_2\cap J_3}\sqrt{n}\hat{\mu}_j/\hat{\sigma}_j>c^{B,3S}(\alpha, J_2\cap J_3)\right)+4\beta_n+C n^{-c}\\
&\quad \leq \alpha-4\beta_n+4\beta_n+C n^{-c}=\alpha + C n^{-c}.
\end{align*}
This completes the proof of the theorem.
\qed

\subsection{Proof of Theorem \ref{thm: rate of uniform consistency}}
To prove this theorem, we will apply the following lemma:
\begin{lemma}
\label{lem: general power}
In the setting of Theorem \ref{thm: rate of uniform consistency}, for every $\underline{\epsilon}\geq 0$, there exist $\epsilon >0$ and $\delta \in (0,1)$ such that whenever
 \[
\max_{1 \leq j \leq p} (\mu_{j}/\sigma_{j}) \geq (1+\delta)(1+\epsilon + \underline{\epsilon}) \sqrt{\frac{2 \log (p/\alpha)}{n}},
\]
we have
\begin{align*}
 \Pr (T > \hat{c}(\alpha)) \geq &1-\frac{1}{2(1-\delta)^{2}\epsilon^{2} \log (p/\alpha)} \\
&- \max_{1 \leq j \leq p}\Pr( | \hat{\sigma}_{j}/\sigma_{j} - 1| > \delta)-\Pr\left(\hat{c}(\alpha)>(1+\underline{\epsilon}) \sqrt{2\log (p/\alpha)}\right).
\end{align*}
\end{lemma}

\noindent
{\em Proof.}
Let $j^{*} \in \{ 1,\dots, p \}$ be any index such that $\mu_{j^{*}}/\sigma_{j^{*}} = \max_{1 \leq j \leq p} (\mu_{j}/\sigma_{j})$. Let $A_{n,1}$ and $A_{n,2}$ be the events that $| \hat{\sigma}_{j^{*}}/\sigma_{j^{*}} - 1 | \leq \delta$ and $\hat{c}(\alpha)\leq (1+\underline{\epsilon})\sqrt{2\log(p/\alpha)}$, respectively. Then on the event $A_{n,1}\cap A_{n,2}$,
\begin{align*}
T &\geq \sqrt{n}\hat{\mu}_{j^{*}}/\hat{\sigma}_{j^{*}} = \sqrt{n}\mu_{j^{*}}/\hat{\sigma}_{j^{*}} + \sqrt{n}(\hat{\mu}_{j^{*}} - \mu_{j^{*}})/\hat{\sigma}_{j^{*}} \\
&\geq (1/(1+\delta)) \cdot \sqrt{n}\mu_{j^{*}}/\sigma_{j^{*}} + \sqrt{n}(\hat{\mu}_{j^{*}} - \mu_{j^{*}})/\hat{\sigma}_{j^{*}} \\
&\geq (1+\epsilon + \underline{\epsilon}) \sqrt{2 \log (p/\alpha)} +  \sqrt{n}(\hat{\mu}_{j^{*}} - \mu_{j^{*}})/\hat{\sigma}_{j^{*}},
\end{align*}
so that
\[
\sqrt{n}(\hat{\mu}_{j^{*}} - \mu_{j^{*}})/\hat{\sigma}_{j^{*}} > -\epsilon \sqrt{2 \log (p/\alpha)}\quad \Rightarrow\quad T > \hat{c}(\alpha).
\]
Hence we have
\begin{align*}
&\Pr (T > \hat{c}(\alpha) ) \geq \Pr \left ( \left \{ T > \hat{c}(\alpha) \right \} \cap A_{n,1}\cap A_{n,2} \right) \\
&\quad \geq \Pr \left ( \left \{ \sqrt{n}(\hat{\mu}_{j^{*}} - \mu_{j^{*}})/\hat{\sigma}_{j^{*}} > -\epsilon \sqrt{2 \log (p/\alpha)} \right \} \cap A_{n,1}\cap A_{n,2} \right) \\
&\quad \geq  \Pr \left ( \left \{\sqrt{n}(\hat{\mu}_{j^{*}}-\mu_{j^{*}})/\sigma_{j^{*}} > - (1-\delta)\epsilon \sqrt{2 \log (p/\alpha)} \right \} \cap A_{n,1}\cap A_{n,2} \right) \\
&\quad \geq \Pr \left ( \sqrt{n}(\hat{\mu}_{j^{*}}-\mu_{j^{*}})/\sigma_{j^{*}} > - (1-\delta)\epsilon \sqrt{2 \log (p/\alpha)}  \right )  - \Pr (A_{n,1}) - \Pr(A_{n,2}).
\end{align*}
By Markov's inequality, we have
\begin{align*}
&\Pr \left ( \sqrt{n}(\hat{\mu}_{j^{*}}-\mu_{j^{*}})/\sigma_{j^{*}} > - (1-\delta)\epsilon \sqrt{2 \log (p/\alpha)}  \right ) \\
&\quad =1-\Pr \left ( \sqrt{n}(\mu_{j^{*}}-\hat{\mu}_{j^{*}})/\sigma_{j^{*}} \geq  (1-\delta)\epsilon \sqrt{2 \log (p/\alpha)}  \right )  \\
&\quad \geq 1-  \frac{1}{2(1-\delta)^{2}\epsilon^{2}\log (p/\alpha)}.
\end{align*}
This completes the proof. \qed

\medskip

\noindent
Getting back to the proof of Theorem \ref{thm: rate of uniform consistency}, let $c,C$ denote generic positive constants depending only on $\alpha,c_{1},C_{1}$ but such that their values may change from place to place. Note that since $M_{n,4}^2 \log^{1/2} p\leq C_{1}n^{1/2-c_{1}}$, by Markov's inequality, there exists  $\delta_n \leq \min \{ C \log^{-1/2} p, 1/2 \}$ such that
\[
\max_{1\leq j\leq p} \Pr\left(|\hat{\sigma}_j/\sigma_j-1|>\delta_n \right) \leq C n^{-c}.
\]
Hence, by Lemma \ref{lem: general power}, we only have to verify that
\begin{equation}\label{eq: bound on critical value wp1}
\Pr(\hat{c}(\alpha) > (1+C\log^{-1/2}p) \sqrt{2\log (p/\alpha)})\leq C n^{-c}.
\end{equation}
To this end, since $\alpha - 2\beta_{n} \geq c_1\alpha$, we note that
\[
c^{SN,2S}(\alpha) \leq c^{SN}(c_1\alpha), \ c^{B,2S}(\alpha) \vee c^{B,H}(\alpha) \leq c^{B}(c_1\alpha)
\]
where $B=MB$ or $EB$,
so that it suffices to verify (\ref{eq: bound on critical value wp1}) with $\hat{c}(\alpha)=c^{SN}(\alpha)$, $c^{MB}(\alpha)$, and $c^{EB}(\alpha)$.

For $\hat{c}(\alpha)=c^{SN}(\alpha)$, since $\Phi^{-1}(1-p/\alpha) \leq \sqrt{2 \log(p/\alpha)}$ and $\log^{3/2} p \leq C_{1}n$, it is straightforward to see that (\ref{eq: bound on critical value wp1}) is verified.
For $\hat{c}(\alpha)=c^{MB}(\alpha)$, it follows from Lemma \ref{lem: critical value bound} that $c^{MB}(\alpha)\leq \sqrt{2\log p}+\sqrt{2\log(1/\alpha)}$, so that (\ref{eq: bound on critical value wp1}) can be verified by simple algebra.

Now consider $\hat{c}(\alpha)=c^{EB}(\alpha)$. It is established in Step 1 of the proof of Theorem \ref{thm: simulation plugin method} that there exists a sequence $\bar\varphi_n\geq 0$ such that $\bar\varphi_n\leq C n^{-c}$ and $\Pr(c^{EB}(\alpha)>c_0(\alpha-\bar\varphi_n))\leq C n^{-c}$ where $c_0(\alpha-\bar\varphi_n)$ is the $(1-\alpha+\bar\varphi_n)$th quantile of the distribution of $\max_{1\leq j\leq p}Y_j$ and $(Y_1,\dots,Y_p)^T$ is a normal vector with mean zero and all diagonal elements of the covariance matrix equal to one. By Lemma \ref{lem: critical value bound},
$$
c_0(\alpha-\bar\varphi_n)\leq \sqrt{2\log p}+\sqrt{2\log(1/(\alpha-\bar\varphi_n))}.
$$
In addition, simple algebra shows that
$$
(1+C\log^{-1/2}p)\sqrt{2\log(p/\alpha)}>\sqrt{2\log p}+\sqrt{2\log(1/(\alpha-\bar\varphi_n))}
$$
if $C$ is chosen sufficiently large (and depending on $\alpha$). Combining these inequalities gives (\ref{eq: bound on critical value wp1}). This completes the proof.
\qed

\subsection{Proof of Theorem \ref{thm: confidence region}}

The theorem readily follows from Theorems \ref{thm: analytical plugin method}-\ref{thm: HB method}. \qed

\subsection{Proof of Theorem \ref{thm: validity of BMB}}
Here $c,c',C,C'$ denote generic positive constants depending only on $c_1,c_2,C_1$; their values may change from place to place.
It suffices to show that $| \Pr (\check{T} \leq \hat{c}^{BMB}(\alpha)) -  \alpha | \leq Cn^{-c}$ when $\mu_{j} = 0, 1 \leq \forall j \leq p$.
Suppose that $\mu_{j} = 0, 1 \leq \forall j \leq p$.
We use the extensions of the results in \cite{CCK12} to dependent data proved in Appendix \ref{appendix: dependence} ahead.
Note that since $\log (pn) \leq C \sqrt{q}$ (which follows from $(r/q) \log^2 p \leq C_1 n^{-c_2}$), $\sqrt{q} D_n \log^{7/2} (pn) \leq C q D_{n}\log^{5/2} (pn) \leq C' n^{1/2-c_2}$, so that by Theorem \ref{thm: high dim CLT under dependence} in Appendix \ref{appendix: dependence},
\begin{equation}
\sup_{t \in \R} | \Pr (\check{T} \leq t) - \Pr ( \max_{1 \leq j \leq p} \check{Y}_{j} \leq t) | \leq C n^{-c}, \label{eq: step1-dep}
\end{equation}
where $\check{Y} = (\check{Y}_1,\dots,\check{Y}_{p})^{T}$ is a centered normal random vector with covariance matrix $\Ep [ \check{Y} \check{Y}^{T} ] = (1/(mq)) \sum_{l=1}^{m} \Ep [(\sum_{i \in I_{l}} X_{i})(\sum_{i \in I_{l}} X_{i})^{T}]$.
Note that $c_1 \leq \underline{\sigma}^{2}(q) \leq \Ep [\check{Y}_{j}^{2}] \leq \overline{\sigma}^{2}(q)\leq C_1, 1 \leq \forall j \leq p$.

Let $\check{W}_{0} = \max_{1 \leq j \leq p} (1/\sqrt{mq}) \sum_{l=1}^{m} \epsilon_{l} \sum_{i \in I_{l}} X_{ij}$. Then by Theorem \ref{thm: validity of BMB prelim}, with probability larger than $1-C n^{-c}$,
\[
\sup_{t \in \R} | \Pr ( \check{W}_{0} \leq t \mid X_{1}^{n} ) - \Pr ( \max_{1 \leq j \leq p} \check{Y}_{j} \leq t) | \leq C' n^{-c'}.
\]
Observe that $|\check{W} - \check{W}_{0}| \leq \max_{1 \leq j \leq p} | \sqrt{n} \hat{\mu}_{j} | \cdot | m^{-1} \sum_{l=1}^{m} \epsilon_l |$.
Here since $q \leq C n^{1/2-c_2}$, we have $m \geq n/(4q) \geq C^{-1} n^{1/2-c_2}$, so that by Markov's inequality, $\Pr ( | m^{-1} \sum_{l=1}^{m} \epsilon_l | > C n^{-1/4+5c_2/8} ) \leq  n^{-c_2/8}$.
On the other hand, by applying Theorem \ref{thm: high dim CLT under dependence} to $(X_{i1},\dots,X_{ip},-X_{i1},\dots,-X_{ip})^{T}$,
we have
\[
\sup_{t \in \R} | \Pr ( \max_{1 \leq j \leq p} | \sqrt{n} \hat{\mu}_{j} | \leq t) - \Pr ( \max_{1 \leq j \leq p} |\check{Y}_{j}| \leq t) | \leq C n^{-c}.
\]
Since $\Ep [ \max_{1 \leq j \leq p} | \check{Y}_{j} | ] \leq C\sqrt{\log p}$, we conclude that
\[
\Pr ( \max_{1 \leq j \leq p} | \sqrt{n} \hat{\mu}_{j} | > C n^{c_2/8} \sqrt{\log p}) \leq C'n^{-c}.
\]
Hence with probability larger than $1-Cn^{-c}$,
\[
\Pr ( | \check{W} - \check{W}_{0} | > \zeta_n \mid X_{1}^{n} ) \leq n^{-c'},
\]
where $\zeta_n = C' n^{-1/4 + 3c_2/4} \sqrt{\log p}$. Note that since $q D_n \log^{5/2} (pn) \leq C_1 n^{1/2-c_2}$,
$n^{-1/4+c_2/2} \log p \leq C q^{-1/2} \leq C' n^{-c_2/2}$ (the second inequality follows from $(r/q) \log^2 p \leq C_1 n^{-c_2}$ so that $q^{-1} \leq C n^{-c_2}$),
and hence $\zeta_n \sqrt{\log p} \leq C n^{-c_2/4}$. Using the anti-concentration property of $\max_{1 \leq j \leq p} \check{Y}_{j}$ (see Step 3 in the proof of Theorem \ref{thm: high dim CLT under dependence}), we conclude that with probability larger than $1-Cn^{-c}$,
\[
\sup_{t \in \R} | \Pr (\check{W} \leq t \mid X_{1}^{n}) - \Pr (\max_{1 \leq j \leq p} \check{Y}_{j} \leq t) | \leq C' n^{-c'}.
\]
The desired assertion follows from combining this inequality with (\ref{eq: step1-dep}). \qed

\subsection{Proof of Theorem \ref{thm: approximate moment inequalities}}
Here $c,C$ denote generic positive constants depending only on $c_1,C_1$; their values may change from place to place.
Define
\begin{align*}
&\bar{T}:=\max_{1\leq j\leq p}\frac{\sqrt{n}(\hat{\mu}_{j}-\mu_j)}{\hat{\sigma}_j}, \ T_0:=\max_{1\leq j\leq p}\frac{\sqrt{n}(\hat{\mu}_{j,0}-\mu_j)}{\sigma_j},\\
&W^{MB}:=\max_{1\leq j\leq p}\frac{\sqrt{n}\En[\epsilon_i(\hat{X}_{i j}-\hat{\mu}_j)]}{\hat{\sigma}_j}, \ \bar{W}^{MB}:=\max_{1\leq j\leq p}\frac{\sqrt{n}\En[\epsilon_i(X_{i j}-\hat{\mu}_{j,0})]}{\sigma_j},\\
&W^{EB}:=\max_{1\leq j\leq p}\frac{\sqrt{n}\En[\hat{X}_{i j}^*-\hat{\mu}_j]}{\hat{\sigma}_j}, \ \bar{W}^{EB}:=\max_{1\leq j\leq p}\frac{\sqrt{n}\En[X_{i j}^*-\hat{\mu}_{j,0}]}{\sigma_j},
\end{align*}
where $\hat{X}_1^*,\dots,\hat{X}_n^*$ is an empirical bootstrap sample from $\hat{X}_1,\dots,\hat{X}_n$, and $X_1^*,\dots,X_n^*$ is an empirical bootstrap sample from $X_1,\dots,X_n$. Observe that the critical values $c^{MB,2S}(\alpha)$ and $c^{EB,2S}(\alpha)$ are based on the bootstrap statistics $W^{MB}$ and $W^{EB}$.  

We divide the proof into several steps. 
In Steps 1, 2, and 3, we prove that 
\begin{align}
&\Pr\left(|\bar{T}-T_0|>\zeta_{n1}'\right)\leq C n^{-c}, \label{eq: T approximation res}\\
&\Pr(\Pr(|W^{MB}-\bar{W}^{MB}|>\zeta_{n1}'\mid X_{1}^n)> C n^{-c})\leq C n^{-c},\label{eq: WMB approximation res}\\
&\Pr(\Pr(|W^{EB}-\bar{W}^{EB}|>\zeta_{n1}'\mid X_{1}^n)> C n^{-c})\leq C n^{-c},\label{eq: WEB approximation res}
\end{align}
respectively, for some $\zeta_{n1}'$ satisfying $\zeta_{n1}'\sqrt{\log p}\leq C n^{-c}$.
In Step 4, we prove an auxiliary result that
\begin{equation}\label{eq: sigma approximation res}
\Pr\left(\max_{1\leq j\leq p}|1-\hat{\sigma}_j/\hat{\sigma}_{j,0}|>C\zeta_{n1}\right)\leq C n^{-c}.
\end{equation}
Given results (\ref{eq: T approximation res})-(\ref{eq: WEB approximation res}), the conclusions of the theorem follow by repeating the arguments used in the proofs of Theorems \ref{thm: simulation plugin method} and \ref{thm: simulation MS method}.

In the proof, we will frequently use the following implications of Lemma \ref{lem: variance} (recall that $\hat{\sigma}_j$ in Lemma \ref{lem: variance} is denoted as $\hat{\sigma}_{j,0}$ in this proof):
\begin{align}
&\Pr\left(\max_{1\leq j\leq p}(\sigma_j/\hat{\sigma}_{j,0})^2>2\right)\leq C n^{-c},\label{eq: sigma upper bound res}\\
&\Pr\left(\max_{1\leq j\leq p}(\hat{\sigma}_{j,0}/\sigma_j)^2>2\right)\leq C n^{-c}.\label{eq: sigma lower bound res}
\end{align} 

\medskip

{\bf Step 1}. Here we wish to prove (\ref{eq: T approximation res}).
Define $T_0':=\max_{1\leq j\leq p}\sqrt{n}(\hat{\mu}_{j,0}-\mu_j)/\hat{\sigma}_j$. Observe that
\begin{align*}
|\bar{T}-T_0'|
&\leq \max_{1\leq j\leq p}\left|\frac{\sqrt{n}(\hat{\mu}_j-\hat{\mu}_{j,0})}{\hat{\sigma}_j}\right|\leq C\max_{1\leq j\leq p}\left|\frac{\sqrt{n}(\hat{\mu}_j-\hat{\mu}_{j,0})}{\sigma_j}\right|\\
&\leq C\max_{1\leq j\leq p}|\sqrt{n}(\hat{\mu}_j-\hat{\mu}_{j,0})|\leq C\zeta_{n1}
\end{align*}
with probability larger than $1-C n^{-c}$ where the second inequality in the first line follows from (\ref{eq: sigma approximation res}) and (\ref{eq: sigma upper bound res}) and the second line follows from assumptions. Also, 
\[
|T_0'-T_0|\leq \max_{1\leq j\leq p}|\sigma_j/\hat{\sigma}_j-1|\times \max_{1\leq j\leq p}|\sqrt{n}\En[Z_{i j}]|,
\]
where $Z_{i j}=(X_{i j}-\mu_j)/\sigma_j$. As shown in Step 3 of the proof of Theorem \ref{thm: simulation plugin method},
\[
\Pr\left(\max_{1\leq j\leq p}|\sqrt{n}\En[Z_{i j}]|>n^{c_1/4}\sqrt{\log p}\right)\leq C n^{-c}.
\]
In addition, using an elementary inequality $|ab-1|\leq |a||b-1|+|a-1|$ with $a=\sigma_j/\hat{\sigma}_{j,0}$ and $b=\hat{\sigma}_{j,0}/\hat{\sigma}_j$, we obtain from (\ref{eq: sigma estimation bound}) in the proof of Theorem \ref{thm: simulation plugin method}, (\ref{eq: sigma approximation res}), and (\ref{eq: sigma upper bound res}) that
\[
\Pr\left(\max_{1\leq j\leq p}|\sigma_j/\hat{\sigma}_j-1|>C(n^{-1/2+c_1/4}B_n^2\log p+\zeta_{n 1})\right)\leq C n^{-c}
\]
(remember that $\hat{\sigma}_j$ in the proof of Theorem \ref{thm: simulation plugin method} corresponds to $\hat{\sigma}_{j,0}$ here). Therefore, the claim of this step holds with $\zeta_{n1}':=C(n^{-1/2+c_1/2}B_n^2(\log p)^{3/2}+\zeta_{n1}n^{c_1/4}\sqrt{\log p})$ for sufficiently large $C$.

\medskip

{\bf Step 2}. Here we wish to prove (\ref{eq: WMB approximation res}).
Let $\hat{W}^{MB}:=\max_{1\leq j\leq p}\sqrt{n}\En[\epsilon_i(X_{i j}-\hat{\mu}_{j,0})]/\hat{\sigma}_j$. By (\ref{eq: sigma approximation res}) and (\ref{eq: sigma upper bound res}), with probability larger than $1-C n^{-c}$,
\begin{align*}
|W^{MB}-\hat{W}^{MB}|
&\leq \max_{1\leq j\leq p}\frac{|\sqrt{n}\En[\epsilon_i(\hat{X}_{i j}-X_{i j}-\hat{\mu}_j+\hat{\mu}_{j,0})]|}{\hat{\sigma}_j}\\
&\leq C\max_{1\leq j\leq p}\frac{|\sqrt{n}\En[\epsilon_i(\hat{X}_{i j}-X_{i j}-\hat{\mu}_j+\hat{\mu}_{j,0})]|}{\sigma_j}\\
&\leq C\max_{1\leq j\leq p}|\sqrt{n}\En[\epsilon_i(\hat{X}_{i j}-X_{i j}-\hat{\mu}_j+\hat{\mu}_{j,0})]|, 
\end{align*}
where the third inequality follows from the assumption that $\sigma_j\geq c_1$ for all $j=1,\dots,p$. 
Conditional on $X_1^n$, the vector $(\sqrt{n}\En[\epsilon_i(\hat{X}_{i j}-X_{i j}-\hat{\mu}_j+\hat{\mu}_{j,0})])_{1\leq j\leq p}$ is normal with mean zero and all diagonal elements of the covariance matrix bounded by $\max_{1\leq j\leq p}\En[(\hat{X}_{i j}-X_{i j}-\hat{\mu}_j+\hat{\mu}_{j,0})^2]$. As established in  in Step 4 below, the last quantity is bounded by $C\zeta_{n1}^2$ with probability larger than $1-C n^{-c}$. 
Therefore,
\begin{equation}\label{eq: establishing WMB 1}
\Pr(\Pr(|W^{MB}-\hat{W}^{MB}|>C\zeta_{n1}\sqrt{\log p}\mid X_1^n)>C n^{-c})\leq C n^{-c}.
\end{equation}
Moreover 
\[
|\hat{W}^{MB}-\bar{W}^{MB}|\leq \max_{1\leq j\leq p}|\sigma_j/\hat{\sigma}_j-1|\times \bar{W}^{MB}.
\]
Now observe that $\bar{W}^{MB}=\max_{1\leq j\leq p}\sqrt{n}\En[\epsilon_i(X_{i j}-\hat{\mu}_{j,0})/\sigma_j]$ and conditional on the data $X_1^n$, the vector $(\sqrt{n}\En[\epsilon_i(X_{i j}-\hat{\mu}_{j,0})/\sigma_j])_{1\leq j\leq p}$ is normal with mean zero and all diagonal elements of the covariance matrix bounded by $\max_{1\leq j\leq p}(\hat{\sigma}_{j,0}^2/\sigma_j^2)$. By (\ref{eq: sigma lower bound res}), the last quantity is bounded by 2 with probability larger than $1-C n^{-c}$. Therefore,
\begin{equation}\label{eq: establishing WMB 2}
\Pr(\Pr(|\hat{W}^{MB}-\bar{W}^{MB}|>\zeta'_{n1}\mid X_1^n)>C n^{-c})\leq C n^{-c}
\end{equation}
where $\zeta'_{n1}$ is defined in Step 1.
Combining (\ref{eq: establishing WMB 1}) and (\ref{eq: establishing WMB 2}) leads to the assertion of this step.

\medskip

{\bf Step 3}. Here we wish to prove (\ref{eq: WEB approximation res}). Let $\hat{W}^{EB}:=\max_{1\leq j\leq p}\sqrt{n}\En[X_{i j}^*-\hat{\mu}_{j,0}]/\hat{\sigma}_j$. By (\ref{eq: sigma approximation res}) and (\ref{eq: sigma upper bound res}), with probability larger than $1-C n^{-c}$,
\begin{align*}
|W^{EB}-\hat{W}^{EB}|
&\leq \max_{1\leq j\leq p}\frac{|\sqrt{n}\En[\hat{X}_{i j}^*-X_{i j}^*-\hat{\mu}_j+\hat{\mu}_{j,0}]|}{\hat{\sigma}_j}\\
&\leq C\max_{1\leq j\leq p}\frac{|\sqrt{n}\En[\hat{X}_{i j}^*-X_{i j}^*-\hat{\mu}_j+\hat{\mu}_{j,0}]|}{\sigma_j}\\
&\leq C\max_{1\leq j\leq p}|\sqrt{n}\En[\hat{X}_{i j}^*-X_{i j}^*-\hat{\mu}_j+\hat{\mu}_{j,0}]|,
\end{align*}
where the third inequality follows from the assumption that $\sigma_j\geq c_1$ for all $1 \leq j \leq p$. 
Applying Lemma \ref{lem: maximal ineq} conditional on the data $X_1^n$, we have
\begin{align*}
&\Ep\left[\max_{1\leq j\leq p}|\sqrt{n}\En[\hat{X}_{i j}^*-X_{i j}^*-\hat{\mu}_j+\hat{\mu}_{j,0}]|\mid X_1^n\right]\\
&\qquad \leq C\left(\max_{1\leq j\leq p}(\En[(\hat{X}_{i j}-X_{i j})^2]\log p)^{1/2}+\max_{i,j}|\hat{X}_{i j}-X_{i j}|(\log p)/\sqrt{n}\right).
\end{align*}
Therefore, by Markov's inequality, we have
\begin{equation}\label{eq: establishing WEB 1}
\Pr(\Pr(|W^{EB}-\hat{W}^{EB}|>C\zeta_{n1}n^{c_1/4}\sqrt{\log p}\mid X_1^n)>C n^{-c})\leq C n^{-c}.
\end{equation}
Moreover
\[
|\hat{W}^{EB}-\bar{W}^{EB}|\leq \max_{1\leq j\leq p}|\sigma_j/\hat{\sigma}_j-1|\times \bar{W}^{EB}.
\]
Applying Lemma \ref{lem: maximal ineq} conditional on the data $X_1^n$ once again, we have
\[
\Ep[\bar{W}^{EB}\mid X_1^n]\leq C\left(\max_{1\leq j\leq p}(\hat{\sigma}_{j,0}/\sigma_j)+\max_{i,j}\frac{|X_{i j}-\mu_j|}{\sigma_j}(\log p)/\sqrt{n}\right).
\]
By (\ref{eq: sigma lower bound res}), $\max_{1\leq j\leq p}(\hat{\sigma}_{j,0}/\sigma_j)\leq \sqrt{2}$ with probability larger than $1-C n^{-c}$.
Here for $Z_{i j}=(X_{i j}-\mu_j)/\sigma_j$,
\begin{align*}
\Ep\left[\max_{1\leq j\leq p}|Z_{i j}| \right]
&\leq \left(\Ep\left[\max_{i,j}|Z_{i j}|^4\right]\right)^{1/4}\leq \left(\Ep\left[n\max_{1\leq j\leq p}|Z_{i j}|^4\right]\right)^{1/4}=n^{1/4}B_n.
\end{align*}
Hence, by Markov's inequality and the assumption that $B_n^2\log^{7/2}(p n)\leq C_1 n^{1/2-c_1}$, we have $\max_{i,j}(|X_{i j}-\mu_j|/\sigma_j)(\log p)/\sqrt{n}\leq C\sqrt{\log p}$ with probability larger than $1- C n^{-c}$ for sufficiently large $C$. Therefore,
\begin{equation}\label{eq: establishing WEB 2}
\Pr(\Pr(|\hat{W}^{EB}-\bar{W}^{EB}|>C\zeta_{n1}\sqrt{\log p}\mid X_1^n)>C n^{-c})\leq C n^{-c}.
\end{equation}
Combining (\ref{eq: establishing WEB 1}) and (\ref{eq: establishing WEB 2}) leads to the assertion of this step.

\medskip

{\bf Step 4}. Here we wish to prove (\ref{eq: sigma approximation res}). 
Using (\ref{eq: sigma upper bound res}), we obtain that
with probability larger than $1-C n^{-c}$, for all $j=1,\dots,p$,
\begin{align*}
\left|1-\frac{\hat{\sigma}_j}{\hat{\sigma}_{j,0}}\right|
&\leq \left|1-\Big(\frac{\hat{\sigma}_j}{\hat{\sigma}_{j,0}}\Big)^2\right|= \frac{1}{\hat{\sigma}_{j,0}^2}\left|\hat{\sigma}_j^2-\hat{\sigma}_{j,0}^2\right|\leq \frac{2}{\sigma_j^2}\left|\hat{\sigma}_j^2-\hat{\sigma}_{j,0}^2\right|\\
&= \frac{2}{\sigma_j^2}\left|\En[(\hat{X}_{i j}-\hat{\mu}_j)^2-(X_{i j}-\hat{\mu}_{j,0})^2]\right|.
\end{align*}
Since $a^2-b^2=(a-b)^2+2b(a-b)$ for any $a,b \in \R$, we have, by the Cauchy-Schwarz inequality, 
\begin{align*}
&| \En[(\hat{X}_{i j}-\hat{\mu}_j)^2-(X_{i j}-\hat{\mu}_{j,0})^2]| \leq\En[(\hat{X}_{i j}-X_{i j}-\hat{\mu}_j+\hat{\mu}_{j,0})^2]\\
&\qquad +2\hat{\sigma}_{j,0}\left( \En[(\hat{X}_{i j}-X_{i j}-\hat{\mu}_j+\hat{\mu}_{j,0})^2]\right)^{1/2}.
\end{align*}
Also,
\[
\left( \En[(\hat{X}_{i j}-X_{i j}-\hat{\mu}_j+\hat{\mu}_{j,0})^2]\right)^{1/2}\leq (\En[(\hat{X}_{i j}-X_{i j})^2])^{1/2}+|\hat{\mu}_j-\hat{\mu}_{j,0}|,
\]
which is further bounded by $C \zeta_{n1}$ with probability larger than $1-C n^{-c}$. Taking these inequalities together, we conclude that with probability larger than $1-C n^{-c}$, for all $j=1,\dots,p$,
\[
\left|1-\frac{\hat{\sigma}_j}{\hat{\sigma}_{j,0}}\right|\leq \frac{2(C\zeta_{n 1})^2}{\sigma_j^2}+\frac{4\hat{\sigma}_{j,0}C\zeta_{n 1}}{\sigma_{j}^2}\leq C\zeta_{n1},
\]
where the last inequality follows from the assumption that $\sigma_j\geq c_1$ for all $j=1,\dots,p$ and inequality (\ref{eq: sigma lower bound res}).
This leads to the assertion of Step 4 and completes the proof of the theorem.
\qed

\section{High dimensional CLT under dependence}
\label{appendix: dependence}

In this section, we extend the results of \cite{CCK12} to dependent data.
Let $X_{1},\dots,X_{n}$ be possibly dependent random vectors in $\R^{p}$ with mean zero, defined on the probability space $(\Omega,\mathcal{A},\Pr)$, and let $\check{T} = \max_{1 \leq j \leq p} \sqrt{n} \En [ X_{ij} ]$.
For the sake of simplicity, we assume that there is some constant $D_n \geq 1$ such that
\[
| X_{ij} | \leq D_n, \ a.s., \ 1 \leq i \leq n; 1 \leq j \leq p.
\]
We follow the other notation used in Appendix \ref{sec: dependent data}.
%
In addition, define
\[
S_{l} = \sum_{i \in I_{l}} X_{i}, \ S_{l}' = \sum_{i \in J_{l}} X_{i},
\]
and let $\{ \tilde{S}_{l} \}_{l=1}^{m}$ and $\{ \tilde{S}_{l}' \}_{l=1}^{m}$ be two independent sequences of random vectors in $\R^{p}$ such that
\[
\tilde{S}_{l} \stackrel{d}{=} S_{l}, \ \tilde{S}_{l}' \stackrel{d}{=} S_{l}', 1 \leq l \leq m.
\]
Moreover, let $\check{Y} = (\check{Y}_1,\dots,\check{Y}_{p})^{T}$ be a centered normal random vector with covariance matrix $\Ep [ \check{Y} \check{Y}^{T} ] = (1/(mq)) \sum_{l=1}^{m} \Ep [S_{l}S_{l}^{T}]$.

\begin{theorem}[High dimensional CLT under dependence]
\label{thm: high dim CLT under dependence}
Suppose that there exist constants $0 < c_1 \leq C_1$ and $0 < c_2 < 1/4$ such that $c_{1} \leq \underline{\sigma}^{2}(q) \leq \overline{\sigma}^{2}( r ) \vee \overline{\sigma}^{2}(q)  \leq C_{1}, (r/q) \log^{2} p \leq C_{1} n^{-c_2}$, and
\[
\max \{ qD_n \log^{1/2} p, rD_n \log^{3/2} p, \sqrt{q} D_n \log^{7/2} (pn) \} \leq C_{1}n^{1/2-c_{2}}.
\]
Then there exist constants $c,C > 0$ depending only on $c_1,c_2,C_1$ such that
\[
\sup_{t \in \R} | \Pr(  \check{T} \leq t ) - \Pr (  \max_{1 \leq j \leq p}  \check{Y}_{j} \leq t  )  | \leq Cn^{-c} + 2(m-1)b_r.
\]
\end{theorem}

\begin{proof}
In this proof, $c,C$ denote generic positive constants depending only on $c_1,c_2,C_1$; their values may change from place to place.
We divide the proof into several steps.

\medskip

{\bf Step 1}. (Reduction to independence).
We wish to show that
\begin{align*}
&\Pr \left(  \max_{1 \leq j \leq p} \frac{1}{\sqrt{n}} \sum_{l=1}^{m} \tilde{S}_{lj} \leq t-Cn^{-c} \log^{-1/2} p \right) - n^{-c} - 2(m-1)b_{r} \\
&\qquad \leq \Pr(  \check{T} \leq t ) \\
&\qquad \leq \Pr \left(  \max_{1 \leq j \leq p} \frac{1}{\sqrt{n}} \sum_{l=1}^{m} \tilde{S}_{lj} \leq t+Cn^{-c} \log^{-1/2} p \right) + n^{-c}+ 2(m-1)b_{r}.
\end{align*}

We only prove the second inequality; the first inequality follows from the analogous argument.
Observe that $\sum_{i=1}^{n} X_{i} = \sum_{l=1}^{m} S_{l} + \sum_{l=1}^{m} S_{l}' + S'_{m+1}$,
so that
\[
|\max_{1 \leq j \leq p}  \sum_{i=1}^{n} X_{ij} - \max_{1 \leq j \leq p} \sum_{l=1}^{m} S_{lj} | \leq \max_{1 \leq j \leq p} | \sum_{l=1}^{m} S'_{lj} | + \max_{1 \leq j \leq p} |S'_{m+1,j}|.
\]
By Corollary 2.7 in \cite{Yu1994} \citep[see also][]{Eberlein1984}, we have
\begin{align*}
&\sup_{t \in \R} \left| \Pr \left (  \max_{1 \leq j \leq p} \sum_{l=1}^{m} S_{lj} \leq t \right ) - \Pr \left(  \max_{1 \leq j \leq p} \sum_{l=1}^{m} \tilde{S}_{lj} \leq t \right) \right| \leq (m-1) b_{r}, \\
&\sup_{t > 0} \left| \Pr \left(  \max_{1 \leq j \leq p}| \sum_{l=1}^{m} S'_{lj} | > t \right) - \Pr \left(  \max_{1 \leq j \leq p} | \sum_{l=1}^{m} \tilde{S}'_{lj} | >  t \right) \right| \leq (m-1)b_{q}.
\end{align*}
Hence for every $\delta_{1},\delta_{2} > 0$,
\begin{align*}
&\Pr (\check{T} \leq t ) \leq \Pr \left(  \max_{1 \leq j \leq p} \frac{1}{\sqrt{n}} \sum_{l=1}^{m} \tilde{S}_{lj} \leq t+\delta_{1}+\delta_{2} \right) \\
& +\Pr \left(  \max_{1 \leq j \leq p} | \frac{1}{\sqrt{n}} \sum_{l=1}^{m}\tilde{S}'_{lj} | > \delta_{1} \right ) + \Pr \left (\max_{1 \leq j \leq p} |S'_{m+1,j}| > \sqrt{n} \delta_{2} \right ) + 2(m-1)b_{r} \\
&=I+II+III+IV.
\end{align*}
Since $|S_{m+1,j}| \leq (q+r-1) D_n$ a.s.,  by taking $\delta_{2} = 2(q+r-1) D_n/\sqrt{n} \ (\leq C n^{-c} \log^{-1/2} p)$, we have $III=0$. Moreover, for every $\epsilon > 0$, by Markov's inequality, with $\delta_{1}=\epsilon^{-1} \Ep[\max_{1 \leq j \leq p} | n^{-1/2} \sum_{l=1}^{m}\tilde{S}'_{lj} |]$, $II \leq \epsilon$. It remains to bound the magnitude of $\Ep[\max_{1 \leq j \leq p} | n^{-1/2} \sum_{l=1}^{m}\tilde{S}'_{lj} |]$.
Since $\tilde{S}_{l}', 1 \leq l \leq m$, are independent with $|\tilde{S}'_{lj}| \leq rD_n$ a.s. and $\Var (\tilde{S}'_{lj}) \leq r \overline{\sigma}^{2}(r), 1 \leq l \leq m, 1 \leq j \leq p$, by Lemma \ref{lem: maximal ineq}, we have
\[
\Ep\left[\max_{1 \leq j \leq p} |\frac{1}{\sqrt{n}}  \sum_{l=1}^{m}\tilde{S}'_{lj} |\right] \leq K \left( \sqrt{(r/q) \overline{\sigma}^{2} (r)\log p} + n^{-1/2} rD_n \log p \right).
\]
where $K$ is universal (here we have used the simple fact that $m/n \leq 1/q$), so that the left side is bounded by $C n^{-2c}\log^{-1/2} p$ (by taking $c$ sufficiently small). The conclusion of this step follows from taking $\epsilon = n^{-c}$ so that $\delta_1 \leq C n^{-c}\log^{-1/2} p$.

\medskip

{\bf Step 2}. (Normal approximation to the sum of independent blocks). We wish to show that
\[
\sup_{t \in \R} \left | \Pr \left(  \max_{1 \leq j \leq p} \frac{1}{\sqrt{n}} \sum_{l=1}^{m} \tilde{S}_{lj} \leq t \right ) - \Pr \left(  \max_{1 \leq j \leq p} \sqrt{(mq)/n} \check{Y}_{j} \leq t \right ) \right | \leq Cn^{-c}.
\]

Since $\tilde{S}_{l}, 1 \leq l \leq m$, are independent, we may apply Corollary 2.1 in \cite{CCK12} (note that the covariance matrix of $\sqrt{(mq)/n}\check{Y}$ is the same as that of $n^{-1/2} \sum_{l=1}^{m} \tilde{S}_{l}$). We wish to verify the conditions of the corollary to this case. Observe that
\[
\frac{1}{\sqrt{n}} \sum_{l=1}^{m} \tilde{S}_{lj} = \frac{1}{\sqrt{m}} \sum_{l=1}^{m} \frac{\tilde{S}_{lj}}{\sqrt{n/m}},
\]
and $\sqrt{q} \leq \sqrt{n/m} \leq 2\sqrt{q}$ (recall that $q+r \leq n/2$). Hence
\[
c_{1}/4 \leq \underline{\sigma}^{2}(q)/4 \leq \Var \left (\tilde{S}_{lj}/\sqrt{n/m} \right) \leq \overline{\sigma}^{2}(q) \leq C_1,
\]
and $|\tilde{S}_{lj}/\sqrt{n/m}| \leq \sqrt{q} D_n$ a.s., so that the conditions of Corollary 2.1 (i) in \cite{CCK12} are verified with $B_n = \sqrt{q} D_n$, which leads to the assertion of this step (note that $q \leq C n^{1-c}$ so that $m \geq n/(4q) \geq C^{-1} n^{c}$).

\medskip

{\bf Step 3}. (Anti-concentration). We wish to verify that, for every $\epsilon > 0$,
\[
\sup_{t \in \R} \Pr  \left(  \Big| \max_{1 \leq j \leq p}  \check{Y}_{j}  - t  \Big| \leq \epsilon  \right) \leq C \epsilon \sqrt{1 \vee \log (p/\epsilon)}.
\]

Indeed, since $\check{Y}$ is a normal random vector with
\[
c_1 \leq \underline{\sigma}^{2}(q) \leq \Var (\check{Y}_{j}) \leq \overline{\sigma}^{2}(q) \leq C_1, 1 \leq \forall j \leq p,
\]
the desired assertion follows from application of Corollary 1 in \cite{CCK13}.

\medskip

{\bf Step 4}. (Conclusion). By Steps 1-3, we have
\[
\sup_{t \in \R} \left| \Pr ( \check{T} \leq t ) - \Pr \left(\max_{1 \leq j \leq p} \sqrt{(mq)/n} \check{Y}_{j} \leq t \right) \right| \leq C n^{-c}.
\]
It remains to replace $\sqrt{(mq)/n}$ by $1$ on the left side. Observe that
\[
1-\sqrt{(mq)/n} \leq 1-(mq)/n \leq 1-(n/(q+r)-1)(q/n) = r/(q+r) + q/n,
\]
and the right side is bounded by $C n^{-c} \log^{-1} p$. With this $c$, by Markov's inequality,
\[
\Pr \left( \left| \max_{1 \leq j \leq p} \check{Y}_{j} \right| > n^{c/2} \sqrt{\log p}\right) \leq Cn^{-c/2},
\]
as $\Ep [ | \max_{1 \leq j \leq p} \check{Y}_{j} | ] \leq C \sqrt{\log p}$, so that with probability larger than $1-Cn^{-c/2}$,
\[
(1-\sqrt{(mq)/n}) \left| \max_{1 \leq j \leq p} \check{Y}_{j} \right| \leq C'n^{-c/2} \log^{-1/2} p.
\]
By using the anti-concentration property of $\max_{1 \leq j \leq p} \check{Y}_{j}$ (see Step 3), we conclude that
\[
\sup_{t \in \R} \left|\Pr \left(\max_{1 \leq j \leq p} \sqrt{(mq)/n} \check{Y}_{j} \leq t \right) - \Pr \left( \max_{1 \leq j \leq p} \check{Y}_{j} \leq t\right) \right| \leq Cn^{-c}.
\]
This leads to the conclusion of the theorem.
\end{proof}

An inspection of the proof of the above theorem leads to the following corollary on high dimensional CLT for block sums, where the regularity conditions are weaker than those in Theorem \ref{thm: high dim CLT under dependence}.

\begin{corollary}[High dimensional CLT for block sums]
Suppose that there exist constants $C_{1} \geq c_1 > 0$ and $0 < c_2 < 1/2$ such that $c_{1} \leq \underline{\sigma}^{2}(q) \leq \overline{\sigma}^{2}(q)  \leq C_{1}$, and $\sqrt{q} D_n \log^{7/2} (pn) \leq C_{1}n^{1/2-c_{2}}$.
Then there exist constants $c,C > 0$ depending only on $c_1,c_2,C_1$ such that
\[
\sup_{t \in \R} \left | \Pr \left(  \max_{1 \leq j \leq p}  \frac{1}{\sqrt{mq}} \sum_{l=1}^{m} S_{lj} \leq t \right ) - \Pr (  \max_{1 \leq j \leq p} \check{Y}_{j} \leq t ) \right |
 \leq Cn^{-c} + (m-1)b_r.
\]
\end{corollary}

The following theorem is concerned with validity of the block multiplier bootstrap.

\begin{theorem}[Validity of block multiplier bootstrap]
\label{thm: validity of BMB prelim}
Let $\epsilon_1,\dots,\epsilon_m$ be independent standard normal random variables, independent of the data $X_{1}^{n}$. Suppose that there exist constants $0 < c_1 \leq C_1$ and $0 < c_2 < 1/2$ such that $c_{1} \leq \underline{\sigma}^{2}(q) \leq \overline{\sigma}^{2}(q)  \leq C_{1}$ and  $q D_n \log^{5/2} p \leq C_1 n^{1/2-c_2}$. Then there exist constants $c,c',C,C' > 0$ depending only on $c_1,c_2,C_1$ such that, with probability larger than $1-Cn^{-c}-(m-1) b_r$,
\begin{equation}
\sup_{t \in \R} \left | \Pr \left(  \max_{1 \leq j \leq p} \frac{1}{\sqrt{mq}} \sum_{l=1}^{m} \epsilon_{i} S_{lj} \leq t \mid X_{1}^{n} \right ) - \Pr (  \max_{1 \leq j \leq p} \check{Y}_{j} \leq t  ) \right |
\leq C' n^{-c'}.  \label{eq: conditional MB}
\end{equation}
\end{theorem}

\begin{proof}
Here $c,c',C,C'$ denote generic positive constants depending only on $c_1,c_2,C_1$; their values may change from place to place.
By Theorem 2 in \cite{CCK13}, the left side on (\ref{eq: conditional MB}) is bounded by $C \hat{\Delta}^{1/3} \{ 1 \vee \log (p/\hat{\Delta}) \}^{2/3}$,
where
\[
\hat{\Delta} =\max_{1 \leq j,k \leq p} {\textstyle | (1/(mq)) \sum_{l=1}^{m} (S_{lj} S_{lk} - \Ep [S_{lj} S_{lk}])  |}.
\]
Hence it suffices to prove that $\Pr ( \hat{\Delta} > C'n^{-c'} \log^{-2} p) \leq Cn^{-c} + (m-1) b_r$ with suitable $c,c',C,C'$.  By Corollary 2.7 in \cite{Yu1994}, for every $t > 0$,
\[
\Pr (\hat{\Delta} > t) \leq \Pr (\tilde{\Delta} > t) + (m-1) b_r,
\]
where $\tilde{\Delta} = \max_{1 \leq j,k \leq p} | (1/(mq)) \sum_{l=1}^{m} (\tilde{S}_{lj} \tilde{S}_{lk} - \Ep [S_{lj} S_{lk}]) |$ (recall that $\tilde{S}_{l}, 1 \leq l \leq m$, are independent with $\tilde{S}_{l} \stackrel{d}{=} S_{l}$). Observe that $| \tilde{S}_{lj}\tilde{S}_{lk} | \leq q^{2}D_n^{2}$ a.s. and $\Ep [ (\tilde{S}_{lj} \tilde{S}_{lk})^{2} ] \leq q^{3} D_n^{2} \overline{\sigma}^{2}(q)$.
Hence by Lemma \ref{lem: maximal ineq}, we have
\[
\Ep [ \tilde{\Delta} ] \leq C(  n^{-1/2} qD_n\sqrt{\log p} + n^{-1} q^{2} D_n^{2} \log p).
\]
Since $qD_n\log^{5/2} p \leq C_1 n^{1/2-c_2}$, the right side is bounded by $C'n^{-c_2}\log^{-2} p$. The conclusion of the theorem follows from application of Markov's inequality.
\end{proof}

{\small
\begin{table}[H]\label{tab: 1}
\caption{\small{Results of Monte Carlo experiments for rejection probability. Equicorrelated data, that is $\text{var}(\varepsilon_i)=\Sigma$ where $\Sigma_{j k}=1$ if $j=k$ and $\Sigma_{j k}=\rho$ if $j\neq k$. Design 1: $b = 0$. Design 2: $b = 0.8$.}}
 
\hspace*{-1cm}
\begin{tabular}{cccccccccccc}
\hline\hline
 \tabularnewline
\multicolumn{11}{c}{Design 1 ($\theta = 0$): Null Hypothesis is True}\tabularnewline
\hline\hline
\multirow{2}{*}{$\mathcal L(\epsilon)$} & \multirow{2}{*}{$p$} & \multirow{2}{*}{$\rho$} & \multicolumn{9}{c}{test type}\tabularnewline
\cline{4-12}
 & & & $SN_1$ & $SN_2$ & $MB_1$ & $MB_2$ & $MB_3$ & $EB_1$ & $EB_2$ & $EB_3$ & $AS$  \tabularnewline
\hline\hline

\multirow{9}{*}{$T$}
&\multirow{3}{*}{200}
&0       &.042    &.042    &.046    &.046    &.052    &.042    &.041    &.048    &.104\tabularnewline
&&0.5    &.013    &.013    &.048    &.045    &.047    &.047    &.044    &.048    &.045\tabularnewline
&&0.9    &.005    &.005    &.043    &.043    &.047    &.042    &.041    &.044    &.053\tabularnewline
\cline{3-12}

&\multirow{3}{*}{500}
&0       &.036    &.035    &.049    &.046    &.051    &.044    &.042    &.047    &.132\tabularnewline
&&0.5    &.012    &.011    &.052    &.051    &.049    &.050    &.046    &.042    &.054\tabularnewline
&&0.9    &.003    &.003    &.055    &.052    &.054    &.059    &.053    &.056    &.058\tabularnewline
\cline{3-12}

&\multirow{3}{*}{1000}
&0       &.028    &.025    &.044    &.044    &.051    &.034    &.034    &.047    &.154\tabularnewline
&&0.5    &.017    &.016    &.066    &.064    &.064    &.059    &.059    &.052    &.064\tabularnewline
&&0.9    &.001    &.001    &.054    &.050    &.056    &.050    &.048    &.049    &.054\tabularnewline
\hline
\multirow{9}{*}{$U$} 
&\multirow{3}{*}{200}
&0       &.048    &.048    &.063    &.059    &.052    &.060    &.056    &.049    &.113\tabularnewline
&&0.5    &.024    &.024    &.057    &.056    &.048    &.057    &.054    &.047    &.056\tabularnewline
&&0.9    &.000    &.000    &.049    &.046    &.044    &.050    &.049    &.043    &.049\tabularnewline
\cline{3-12}

&\multirow{3}{*}{500}
&0       &.053    &.049    &.064    &.063    &.057    &.065    &.064    &.055    &.140\tabularnewline
&&0.5    &.012    &.012    &.043    &.042    &.041    &.044    &.043    &.042    &.045\tabularnewline
&&0.9    &.002    &.002    &.050    &.048    &.045    &.042    &.042    &.044    &.053\tabularnewline
\cline{3-12}

&\multirow{3}{*}{1000}
&0       &.048    &.046    &.065    &.065    &.050    &.065    &.063    &.054    &.147\tabularnewline
&&0.5    &.015    &.013    &.062    &.061    &.062    &.062    &.061    &.058    &.052\tabularnewline
&&0.9    &.000    &.000    &.052    &.050    &.050    &.051    &.049    &.048    &.051\tabularnewline

\hline\hline

 \tabularnewline
\multicolumn{11}{c}{Design 2 ($\theta = 0$): Null Hypothesis is True}\tabularnewline
\hline\hline
\multirow{2}{*}{$\mathcal L(\epsilon)$} & \multirow{2}{*}{$p$} & \multirow{2}{*}{$\rho$} & \multicolumn{9}{c}{test type}\tabularnewline
\cline{4-12}
 & & & $SN_1$ & $SN_2$ & $MB_1$ & $MB_2$ & $MB_3$ & $EB_1$ & $EB_2$ & $EB_3$ & $AS$  \tabularnewline
\hline\hline

\multirow{9}{*}{$T$}
&\multirow{3}{*}{200}
&   0    &.003    &.050    &.004    &.060    &.056    &.003    &.053    &.050    &.001\tabularnewline
&&0.5    &.003    &.031    &.012    &.056    &.052    &.011    &.055    &.052    &.003\tabularnewline
&&0.9    &.002    &.010    &.024    &.048    &.047    &.024    &.043    &.043    &.011\tabularnewline
\cline{3-12}

&\multirow{3}{*}{500}
&0       &.003    &.046    &.006    &.056    &.052    &.005    &.051    &.052    &.000\tabularnewline
&&0.5    &.003    &.022    &.009    &.046    &.044    &.011    &.045    &.043    &.004\tabularnewline
&&0.9    &.000    &.002    &.022    &.045    &.042    &.021    &.041    &.040    &.004\tabularnewline
\cline{3-12}

&\multirow{3}{*}{1000}
&0       &.003    &.033    &.004    &.042    &.040    &.003    &.036    &.036    &.000\tabularnewline
&&0.5    &.001    &.018    &.008    &.048    &.047    &.008    &.043    &.043    &.004\tabularnewline
&&0.9    &.000    &.004    &.028    &.043    &.042    &.028    &.039    &.039    &.010\tabularnewline
\hline
\multirow{9}{*}{$U$} 
&\multirow{3}{*}{200}
&   0    &.006    &.056    &.006    &.060    &.058    &.006    &.060    &.058    &.001\tabularnewline
&&0.5    &.002    &.041    &.014    &.054    &.052    &.011    &.050    &.049    &.007\tabularnewline
&&0.9    &.003    &.009    &.033    &.060    &.058    &.032    &.057    &.054    &.017\tabularnewline
\cline{3-12}

&\multirow{3}{*}{500}
&0       &.002    &.048    &.004    &.052    &.052    &.002    &.054    &.052    &.000\tabularnewline
&&0.5    &.003    &.028    &.009    &.054    &.051    &.009    &.057    &.055    &.006\tabularnewline
&&0.9    &.000    &.004    &.021    &.036    &.034    &.022    &.037    &.035    &.008\tabularnewline
\cline{3-12}

&\multirow{3}{*}{1000}
&0       &.005    &.036    &.008    &.050    &.048    &.008    &.051    &.049    &.000\tabularnewline
&&0.5    &.006    &.024    &.015    &.052    &.050    &.015    &.059    &.055    &.006\tabularnewline
&&0.9    &.000    &.002    &.026    &.052    &.049    &.028    &.050    &.046    &.011\tabularnewline

\hline
\end{tabular}
\end{table}
}

{\small
\begin{table}[H]\label{tab: 2}
\caption{\small{Results of Monte Carlo experiments for rejection probability. Autocorrelated data, that is $\text{var}(\varepsilon_i)=\Sigma$ where $\Sigma_{j k}=\rho^{|j-k|}$. Design 3: $b = 0$. Design 4: $b = 0.8$.}}

\hspace*{-1cm}
\begin{tabular}{cccccccccccc}
\hline\hline
 \tabularnewline
\multicolumn{11}{c}{Design 3 ($\theta = 0$): Null Hypothesis is True}\tabularnewline
\hline\hline
\multirow{2}{*}{$\mathcal L(\epsilon)$} & \multirow{2}{*}{$p$} & \multirow{2}{*}{$\rho$} & \multicolumn{9}{c}{test type}\tabularnewline
\cline{4-12}
 & & & $SN_1$ & $SN_2$ & $MB_1$ & $MB_2$ & $MB_3$ & $EB_1$ & $EB_2$ & $EB_3$ & $AS$  \tabularnewline
\hline\hline

\multirow{9}{*}{$T$}
&\multirow{3}{*}{200}
&0       &.041    &.038    &.050    &.047    &.043    &.046    &.045    &.041    &.097\tabularnewline
&&0.5    &.028    &.028    &.041    &.041    &.049    &.035    &.033    &.053    &.077\tabularnewline
&&0.9    &.023    &.022    &.063    &.062    &.057    &.062    &.060    &.059    &.075\tabularnewline
\cline{3-12}

&\multirow{3}{*}{500}
&0       &.031    &.029    &.048    &.044    &.044    &.044    &.041    &.042    &.123\tabularnewline
&&0.5    &.043    &.041    &.053    &.052    &.040    &.047    &.046    &.041    &.117\tabularnewline
&&0.9    &.024    &.023    &.047    &.046    &.041    &.045    &.042    &.040    &.067\tabularnewline
\cline{3-12}

&\multirow{3}{*}{1000}
&0       &.039    &.039    &.056    &.054    &.044    &.049    &.047    &.043    &.151\tabularnewline
&&0.5    &.045    &.042    &.061    &.060    &.037    &.055    &.055    &.033    &.145\tabularnewline
&&0.9    &.022    &.020    &.052    &.052    &.048    &.053    &.048    &.044    &.083\tabularnewline
\hline
\multirow{9}{*}{$U$} 
&\multirow{3}{*}{200}
&0       &.047    &.040    &.056    &.054    &.060    &.054    &.054    &.063    &.121\tabularnewline
&&0.5    &.040    &.039    &.049    &.047    &.060    &.051    &.048    &.060    &.095\tabularnewline
&&0.9    &.029    &.025    &.066    &.064    &.058    &.067    &.063    &.063    &.078\tabularnewline
\cline{3-12}

&\multirow{3}{*}{500}
&0       &.051    &.049    &.073    &.073    &.064    &.077    &.073    &.065    &.142\tabularnewline
&&0.5    &.044    &.043    &.065    &.061    &.063    &.059    &.059    &.064    &.125\tabularnewline
&&0.9    &.014    &.014    &.051    &.048    &.055    &.051    &.050    &.052    &.085\tabularnewline
\cline{3-12}

&\multirow{3}{*}{1000}
&0       &.037    &.037    &.051    &.050    &.054    &.055    &.051    &.062    &.151\tabularnewline
&&0.5    &.044    &.041    &.064    &.059    &.061    &.064    &.059    &.060    &.139\tabularnewline
&&0.9    &.028    &.028    &.066    &.063    &.053    &.067    &.066    &.056    &.102\tabularnewline

\hline\hline

 \tabularnewline
\multicolumn{11}{c}{Design 4 ($\theta = 0$): Null Hypothesis is True}\tabularnewline
\hline\hline
\multirow{2}{*}{$\mathcal L(\epsilon)$} & \multirow{2}{*}{$p$} & \multirow{2}{*}{$\rho$} & \multicolumn{9}{c}{test type}\tabularnewline
\cline{4-12}
 & & & $SN_1$ & $SN_2$ & $MB_1$ & $MB_2$ & $MB_3$ & $EB_1$ & $EB_2$ & $EB_3$ & $AS$  \tabularnewline
\hline\hline

\multirow{9}{*}{$T$}
&\multirow{3}{*}{200}
&0       &.004    &.038    &.004    &.045    &.041    &.004    &.044    &.044    &.003\tabularnewline
&&0.5    &.009    &.057    &.012    &.068    &.066    &.010    &.062    &.063    &.010\tabularnewline
&&0.9    &.002    &.025    &.007    &.051    &.051    &.008    &.051    &.050    &.022\tabularnewline
\cline{3-12}

&\multirow{3}{*}{500}
&0       &.005    &.030    &.006    &.036    &.036    &.005    &.034    &.033    &.001\tabularnewline
&&0.5    &.002    &.033    &.003    &.044    &.043    &.003    &.044    &.041    &.001\tabularnewline
&&0.9    &.000    &.023    &.002    &.055    &.053    &.002    &.057    &.056    &.018\tabularnewline
\cline{3-12}

&\multirow{3}{*}{1000}
&0       &.001    &.041    &.002    &.049    &.047    &.002    &.043    &.045    &.000\tabularnewline
&&0.5    &.007    &.048    &.009    &.054    &.052    &.007    &.053    &.053    &.001\tabularnewline
&&0.9    &.003    &.029    &.004    &.062    &.062    &.004    &.064    &.062    &.013\tabularnewline
\hline
\multirow{9}{*}{$U$} 
&\multirow{3}{*}{200}
&0       &.006    &.046    &.007    &.048    &.047    &.007    &.051    &.049    &.004\tabularnewline
&&0.5    &.003    &.039    &.004    &.053    &.052    &.004    &.050    &.049    &.009\tabularnewline
&&0.9    &.002    &.022    &.004    &.048    &.044    &.003    &.049    &.046    &.021\tabularnewline
\cline{3-12}

&\multirow{3}{*}{500}
&0       &.003    &.038    &.005    &.048    &.046    &.005    &.049    &.045    &.000\tabularnewline
&&0.5    &.003    &.035    &.006    &.049    &.046    &.005    &.046    &.045    &.002\tabularnewline
&&0.9    &.003    &.021    &.006    &.048    &.045    &.006    &.048    &.046    &.015\tabularnewline
\cline{3-12}

&\multirow{3}{*}{1000}
&0       &.004    &.045    &.006    &.052    &.051    &.007    &.056    &.054    &.000\tabularnewline
&&0.5    &.003    &.028    &.005    &.047    &.046    &.005    &.045    &.045    &.000\tabularnewline
&&0.9    &.004    &.025    &.009    &.051    &.049    &.009    &.055    &.053    &.010\tabularnewline

\hline
\end{tabular}
\end{table}
}

{\small
\begin{table}[H]\label{tab: 3}
\caption{\small{Results of Monte Carlo experiments for rejection probability.  Equicorrelated data, that is $\text{var}(\varepsilon_i)=\Sigma$ where $\Sigma_{j k}=1$ if $j=k$ and $\Sigma_{j k}=\rho$ if $j\neq k$. Design 5: $b = 0$. Design 6: $b = 0.8$.}}
\hspace*{-1cm}
\begin{tabular}{cccccccccccc}
\hline\hline
 \tabularnewline
\multicolumn{11}{c}{Design 5 ($\theta = 0.07$): Null Hypothesis is False}\tabularnewline
\hline\hline
\multirow{2}{*}{$\mathcal L(\epsilon)$} & \multirow{2}{*}{$p$} & \multirow{2}{*}{$\rho$} & \multicolumn{9}{c}{test type}\tabularnewline
\cline{4-12}
 & & & $SN_1$ & $SN_2$ & $MB_1$ & $MB_2$ & $MB_3$ & $EB_1$ & $EB_2$ & $EB_3$ & $AS$  \tabularnewline
\hline\hline

\multirow{9}{*}{$T$}
&\multirow{3}{*}{200}
&0       &.447    &.437    &.518    &.501    &.842    &.476    &.467    &.830    &.999\tabularnewline
&&0.5    &.176    &.174    &.309    &.301    &.489    &.300    &.292    &.480    &.130\tabularnewline
&&0.9    &.050    &.047    &.332    &.321    &.392    &.326    &.318    &.393    &.096\tabularnewline
\cline{3-12}

&\multirow{3}{*}{500}
&0       &.538    &.529    &.597    &.587    &.922    &.570    &.562    &.914    &.999\tabularnewline
&&0.5    &.187    &.184    &.333    &.329    &.501    &.333    &.325    &.493    &.134\tabularnewline
&&0.9    &.043    &.043    &.344    &.338    &.407    &.336    &.333    &.400    &.099\tabularnewline
\cline{3-12}

&\multirow{3}{*}{1000}
&0       &.594    &.581    &.681    &.665    &.954    &.635    &.625    &.941    &.999\tabularnewline
&&0.5    &.191    &.187    &.401    &.393    &.517    &.379    &.366    &.518    &.153\tabularnewline
&&0.9    &.042    &.040    &.290    &.284    &.335    &.286    &.281    &.332    &.104\tabularnewline
\hline
\multirow{9}{*}{$U$} 
&\multirow{3}{*}{200}
&0       &.469    &.456    &.537    &.521    &.846    &.532    &.526    &.855    &.999\tabularnewline
&&0.5    &.204    &.199    &.354    &.346    &.525    &.358    &.353    &.523    &.136\tabularnewline
&&0.9    &.051    &.050    &.316    &.311    &.374    &.314    &.309    &.374    &.097\tabularnewline
\cline{3-12}

&\multirow{3}{*}{500}
&0       &.529    &.514    &.605    &.596    &.907    &.617    &.610    &.907    &.999\tabularnewline
&&0.5    &.187    &.184    &.356    &.348    &.505    &.351    &.345    &.503    &.138\tabularnewline
&&0.9    &.045    &.045    &.337    &.332    &.378    &.339    &.330    &.381    &.114\tabularnewline
\cline{3-12}

&\multirow{3}{*}{1000}
&0       &.572    &.562    &.659    &.646    &.934    &.667    &.658    &.942    &.999\tabularnewline
&&0.5    &.174    &.170    &.345    &.340    &.520    &.356    &.343    &.509    &.128\tabularnewline
&&0.9    &.033    &.032    &.340    &.334    &.371    &.336    &.331    &.373    &.101\tabularnewline

\hline\hline

 \tabularnewline
\multicolumn{11}{c}{Design 6 ($\theta = 0.07$): Null Hypothesis is False}\tabularnewline
\hline\hline
\multirow{2}{*}{$\mathcal L(\epsilon)$} & \multirow{2}{*}{$p$} & \multirow{2}{*}{$\rho$} & \multicolumn{9}{c}{test type}\tabularnewline
\cline{4-12}
 & & & $SN_1$ & $SN_2$ & $MB_1$ & $MB_2$ & $MB_3$ & $EB_1$ & $EB_2$ & $EB_3$ & $AS$  \tabularnewline
\hline\hline

\multirow{9}{*}{$T$}
&\multirow{3}{*}{200}
&0       &.244    &.737    &.286    &.767    &.762    &.259    &.759    &.750    &.966\tabularnewline
&&0.5    &.143    &.407    &.265    &.500    &.491    &.256    &.496    &.490    &.224\tabularnewline
&&0.9    &.052    &.176    &.290    &.387    &.379    &.295    &.388    &.384    &.187\tabularnewline
\cline{3-12}

&\multirow{3}{*}{500}
&0       &.318    &.851    &.369    &.871    &.867    &.333    &.856    &.864    &.999\tabularnewline
&&0.5    &.116    &.368    &.264    &.509    &.502    &.255    &.501    &.493    &.197\tabularnewline
&&0.9    &.038    &.135    &.303    &.389    &.387    &.300    &.384    &.373    &.188\tabularnewline
\cline{3-12}

&\multirow{3}{*}{1000}
&0       &.368    &.892    &.452    &.923    &.920    &.402    &.897    &.909    &.999\tabularnewline
&&0.5    &.115    &.357    &.263    &.513    &.504    &.259    &.500    &.501    &.193\tabularnewline
&&0.9    &.032    &.092    &.281    &.355    &.348    &.277    &.352    &.345    &.174\tabularnewline
\hline
\multirow{9}{*}{$U$} 
&\multirow{3}{*}{200}
&0       &.249    &.751    &.294    &.765    &.756    &.292    &.768    &.761    &.962\tabularnewline
&&0.5    &.147    &.416    &.255    &.518    &.507    &.260    &.511    &.507    &.217\tabularnewline
&&0.9    &.034    &.155    &.281    &.389    &.376    &.283    &.380    &.373    &.181\tabularnewline
\cline{3-12}

&\multirow{3}{*}{500}
&0       &.315    &.832    &.377    &.855    &.849    &.375    &.862    &.853    &.999\tabularnewline
&&0.5    &.120    &.360    &.246    &.486    &.482    &.250    &.487    &.476    &.199\tabularnewline
&&0.9    &.035    &.110    &.293    &.385    &.376    &.294    &.382    &.376    &.163\tabularnewline
\cline{3-12}

&\multirow{3}{*}{1000}
&0       &.351    &.890    &.430    &.917    &.911    &.430    &.920    &.918    &.999\tabularnewline
&&0.5    &.132    &.389    &.290    &.532    &.525    &.292    &.537    &.533    &.221\tabularnewline
&&0.9    &.028    &.107    &.323    &.390    &.383    &.323    &.396    &.391    &.194\tabularnewline

\hline
\end{tabular}
\end{table}
}

{\small
\begin{table}[H]\label{tab: 4}
\caption{\small{Results of Monte Carlo experiments for rejection probability. Autocorrelated data, that is $\text{var}(\varepsilon_i)=\Sigma$ where $\Sigma_{j k}=\rho^{|j-k|}$. Design 7: $b = 0$. Design 8: $b = 0.8$.}}

\hspace*{-1cm}
\begin{tabular}{cccccccccccc}
\hline\hline
 \tabularnewline
\multicolumn{11}{c}{Design 7 ($\theta = 0.07$): Null Hypothesis is False}\tabularnewline
\hline\hline
\multirow{2}{*}{$\mathcal L(\epsilon)$} & \multirow{2}{*}{$p$} & \multirow{2}{*}{$\rho$} & \multicolumn{9}{c}{test type}\tabularnewline
\cline{4-12}
 & & & $SN_1$ & $SN_2$ & $MB_1$ & $MB_2$ & $MB_3$ & $EB_1$ & $EB_2$ & $EB_3$ & $AS$  \tabularnewline
\hline\hline

\multirow{9}{*}{$T$}
&\multirow{3}{*}{200}
&0       &.429    &.420    &.489    &.481    &.826    &.464    &.458    &.814    &.999\tabularnewline
&&0.5    &.395    &.385    &.452    &.443    &.762    &.442    &.433    &.762    &.934\tabularnewline
&&0.9    &.183    &.180    &.303    &.295    &.531    &.301    &.289    &.535    &.391\tabularnewline
\cline{3-12}

&\multirow{3}{*}{500}
&0       &.560    &.548    &.631    &.621    &.924    &.598    &.589    &.913    &.999\tabularnewline
&&0.5    &.495    &.484    &.562    &.554    &.875    &.552    &.538    &.869    &.999\tabularnewline
&&0.9    &.243    &.238    &.393    &.382    &.663    &.391    &.382    &.658    &.655\tabularnewline
\cline{3-12}

&\multirow{3}{*}{1000}
&0       &.612    &.597    &.695    &.688    &.951    &.649    &.639    &.940    &.999\tabularnewline
&&0.5    &.586    &.576    &.693    &.682    &.938    &.663    &.652    &.931    &.999\tabularnewline
&&0.9    &.261    &.256    &.428    &.413    &.732    &.414    &.406    &.728    &.860\tabularnewline
\hline
\multirow{9}{*}{$U$} 
&\multirow{3}{*}{200}
&0       &.445    &.433    &.499    &.484    &.830    &.504    &.496    &.827    &.999\tabularnewline
&&0.5    &.392    &.382    &.454    &.442    &.745    &.455    &.444    &.744    &.930\tabularnewline
&&0.9    &.178    &.176    &.299    &.288    &.537    &.305    &.295    &.534    &.399\tabularnewline
\cline{3-12}

&\multirow{3}{*}{500}
&0       &.526    &.520    &.611    &.600    &.903    &.611    &.602    &.904    &.999\tabularnewline
&&0.5    &.489    &.475    &.558    &.548    &.845    &.561    &.552    &.851    &.999\tabularnewline
&&0.9    &.241    &.235    &.358    &.351    &.639    &.363    &.355    &.635    &.657\tabularnewline
\cline{3-12}

&\multirow{3}{*}{1000}
&0       &.604    &.595    &.703    &.683    &.950    &.702    &.694    &.953    &.999\tabularnewline
&&0.5    &.541    &.526    &.630    &.619    &.912    &.621    &.616    &.914    &.999\tabularnewline
&&0.9    &.272    &.267    &.445    &.433    &.740    &.440    &.421    &.746    &.890\tabularnewline

\hline\hline

 \tabularnewline
\multicolumn{11}{c}{Design 8 ($\theta = 0.07$): Null Hypothesis is False}\tabularnewline
\hline\hline
\multirow{2}{*}{$\mathcal L(\epsilon)$} & \multirow{2}{*}{$p$} & \multirow{2}{*}{$\rho$} & \multicolumn{9}{c}{test type}\tabularnewline
\cline{4-12}
 & & & $SN_1$ & $SN_2$ & $MB_1$ & $MB_2$ & $MB_3$ & $EB_1$ & $EB_2$ & $EB_3$ & $AS$  \tabularnewline
\hline\hline

\multirow{9}{*}{$T$}
&\multirow{3}{*}{200}
&0       &.231    &.731    &.274    &.758    &.753    &.257    &.746    &.741    &.968\tabularnewline
&&0.5    &.224    &.633    &.252    &.666    &.660    &.249    &.664    &.658    &.770\tabularnewline
&&0.9    &.095    &.316    &.167    &.472    &.464    &.167    &.473    &.465    &.368\tabularnewline
\cline{3-12}

&\multirow{3}{*}{500}
&0       &.338    &.842    &.387    &.866    &.861    &.368    &.859    &.859    &.999\tabularnewline
&&0.5    &.274    &.767    &.332    &.809    &.802    &.318    &.801    &.800    &.972\tabularnewline
&&0.9    &.118    &.387    &.196    &.557    &.552    &.196    &.552    &.546    &.528\tabularnewline
\cline{3-12}

&\multirow{3}{*}{1000}
&0       &.363    &.907    &.435    &.933    &.930    &.398    &.915    &.920    &.999\tabularnewline
&&0.5    &.333    &.856    &.403    &.899    &.893    &.382    &.880    &.882    &.999\tabularnewline
&&0.9    &.171    &.487    &.266    &.661    &.656    &.264    &.661    &.654    &.724\tabularnewline
\hline
\multirow{9}{*}{$U$} 
&\multirow{3}{*}{200}
&0       &.249    &.726    &.292    &.751    &.739    &.294    &.755    &.747    &.957\tabularnewline
&&0.5    &.203    &.650    &.240    &.697    &.688    &.246    &.698    &.683    &.793\tabularnewline
&&0.9    &.091    &.311    &.159    &.457    &.446    &.164    &.457    &.448    &.385\tabularnewline
\cline{3-12}

&\multirow{3}{*}{500}
&0       &.305    &.839    &.360    &.877    &.869    &.370    &.864    &.860    &.999\tabularnewline
&&0.5    &.263    &.748    &.316    &.802    &.795    &.321    &.809    &.795    &.970\tabularnewline
&&0.9    &.142    &.407    &.218    &.584    &.575    &.216    &.575    &.571    &.538\tabularnewline
\cline{3-12}

&\multirow{3}{*}{1000}
&0       &.345    &.898    &.420    &.914    &.910    &.421    &.918    &.915    &.999\tabularnewline
&&0.5    &.329    &.809    &.387    &.857    &.850    &.389    &.862    &.859    &.999\tabularnewline
&&0.9    &.174    &.480    &.269    &.654    &.646    &.270    &.652    &.640    &.716\tabularnewline

\hline
\end{tabular}
\end{table}
}

{\small
\begin{table}[H]\label{tab: 5a}
\caption{\small{Results of Monte Carlo experiments for rejection probability. Two-step MB method.}}

\hspace*{-1cm}
\begin{tabular}{ccccccccccc}
\hline\hline
 \tabularnewline
\multicolumn{11}{c}{$p=200$}\tabularnewline

 \hline\hline

 \multirow{2}{*}{$b$}   &  \multicolumn{10}{c}{$\beta$} \tabularnewline
\cline{2-11}
   &  .001  & .002   & .003   & .004   & .005   & .006   & .007   & .008   & .009   & .010 \tabularnewline
 \hline\hline
   
.05    &.289    &.279    &.269    &.262    &.256    &.245    &.238    &.231    &.226    &.221\tabularnewline
.10    &.272    &.262    &.254    &.247    &.242    &.231    &.225    &.218    &.213    &.208\tabularnewline
.15    &.272    &.262    &.254    &.247    &.242    &.231    &.225    &.218    &.213    &.208\tabularnewline
.20    &.272    &.262    &.254    &.247    &.242    &.231    &.225    &.218    &.213    &.208\tabularnewline
.25    &.272    &.262    &.254    &.247    &.242    &.231    &.225    &.219    &.213    &.208\tabularnewline
.30    &.272    &.263    &.255    &.249    &.243    &.232    &.227    &.222    &.214    &.210\tabularnewline
.35    &.272    &.268    &.260    &.255    &.253    &.243    &.237    &.230    &.226    &.215\tabularnewline
.40    &.293    &.289    &.280    &.285    &.278    &.277    &.273    &.271    &.264    &.257\tabularnewline
.45    &.354    &.364    &.371    &.373    &.377    &.377    &.377    &.376    &.374    &.370\tabularnewline
.50    &.479    &.493    &.513    &.527    &.526    &.529    &.525    &.525    &.524    &.519\tabularnewline
.55    &.627    &.642    &.651    &.658    &.656    &.647    &.644    &.636    &.625    &.613\tabularnewline
.60    &.731    &.730    &.728    &.723    &.710    &.702    &.688    &.677    &.659    &.644\tabularnewline
.65    &.757    &.750    &.741    &.732    &.722    &.711    &.699    &.686    &.665    &.648\tabularnewline
.70    &.765    &.754    &.742    &.733    &.722    &.712    &.700    &.686    &.666    &.648\tabularnewline
.75    &.766    &.754    &.742    &.733    &.722    &.712    &.700    &.686    &.666    &.648\tabularnewline
.80    &.766    &.754    &.742    &.733    &.722    &.712    &.700    &.686    &.666    &.648\tabularnewline

 \hline\hline
  \tabularnewline
 \multicolumn{11}{c}{$p=1000$}\tabularnewline
 \hline\hline
 
 \multirow{2}{*}{$b$}   &  \multicolumn{10}{c}{$\beta$} \tabularnewline
\cline{2-11}
   &  .001  & .002   & .003   & .004   & .005   & .006   & .007   & .008   & .009   & .010 \tabularnewline
 \hline\hline

.05    &.455    &.445    &.439    &.425    &.411    &.397    &.381    &.370    &.356    &.349\tabularnewline
.10    &.442    &.432    &.426    &.413    &.399    &.385    &.371    &.361    &.347    &.340\tabularnewline
.15    &.442    &.432    &.426    &.413    &.399    &.385    &.371    &.361    &.347    &.340\tabularnewline
.20    &.442    &.432    &.426    &.413    &.399    &.385    &.371    &.361    &.347    &.340\tabularnewline
.25    &.442    &.432    &.426    &.413    &.399    &.385    &.371    &.361    &.347    &.340\tabularnewline
.30    &.442    &.432    &.426    &.413    &.399    &.385    &.371    &.361    &.347    &.340\tabularnewline
.35    &.442    &.434    &.427    &.417    &.401    &.387    &.375    &.363    &.353    &.346\tabularnewline
.40    &.455    &.445    &.438    &.430    &.425    &.410    &.398    &.389    &.377    &.358\tabularnewline
.45    &.486    &.485    &.483    &.478    &.476    &.470    &.465    &.460    &.457    &.446\tabularnewline
.50    &.554    &.574    &.581    &.586    &.586    &.588    &.587    &.586    &.585    &.577\tabularnewline
.55    &.710    &.743    &.757    &.770    &.769    &.769    &.767    &.764    &.756    &.746\tabularnewline
.60    &.849    &.869    &.876    &.877    &.872    &.864    &.852    &.840    &.827    &.820\tabularnewline
.65    &.917    &.921    &.918    &.908    &.903    &.892    &.882    &.873    &.863    &.843\tabularnewline
.70    &.925    &.924    &.920    &.913    &.907    &.895    &.884    &.875    &.865    &.848\tabularnewline
.75    &.927    &.925    &.922    &.913    &.907    &.895    &.884    &.875    &.865    &.848\tabularnewline
.80    &.927    &.925    &.922    &.913    &.907    &.895    &.884    &.875    &.865    &.848\tabularnewline

\hline
\end{tabular}
\end{table}
}

{\small
\begin{table}[H]\label{tab: 6}
\caption{\small{Results of Monte Carlo experiments for rejection probability. Three-step MB method.}}

\hspace*{-1cm}
\begin{tabular}{ccccccccccc}
\hline\hline
 \tabularnewline
 \multicolumn{11}{c}{$p=200$}\tabularnewline
 \hline\hline

 \multirow{2}{*}{$b$}   &  \multicolumn{10}{c}{$\beta$} \tabularnewline
\cline{2-11}
   &  .001  & .002   & .003   & .004   & .005   & .006   & .007   & .008   & .009   & .010 \tabularnewline
 \hline\hline

.05    &.642    &.619    &.601    &.567    &.536    &.496    &.465    &.425    &.370    &.291\tabularnewline
.10    &.601    &.578    &.561    &.529    &.499    &.464    &.437    &.398    &.349    &.276\tabularnewline
.15    &.601    &.578    &.561    &.529    &.499    &.464    &.437    &.398    &.349    &.276\tabularnewline
.20    &.601    &.578    &.561    &.529    &.499    &.464    &.437    &.398    &.349    &.276\tabularnewline
.25    &.601    &.578    &.561    &.529    &.499    &.464    &.437    &.398    &.349    &.276\tabularnewline
.30    &.601    &.578    &.561    &.529    &.499    &.464    &.437    &.398    &.349    &.276\tabularnewline
.35    &.601    &.578    &.561    &.529    &.499    &.464    &.437    &.398    &.350    &.277\tabularnewline
.40    &.601    &.581    &.562    &.533    &.502    &.469    &.441    &.408    &.355    &.283\tabularnewline
.45    &.604    &.594    &.575    &.554    &.522    &.492    &.460    &.430    &.369    &.314\tabularnewline
.50    &.634    &.623    &.604    &.585    &.569    &.543    &.501    &.467    &.421    &.362\tabularnewline
.55    &.680    &.670    &.656    &.634    &.610    &.588    &.545    &.518    &.470    &.398\tabularnewline
.60    &.729    &.712    &.696    &.671    &.639    &.613    &.575    &.537    &.491    &.423\tabularnewline
.65    &.749    &.731    &.709    &.685    &.647    &.618    &.583    &.540    &.494    &.426\tabularnewline
.70    &.753    &.733    &.712    &.686    &.648    &.619    &.583    &.540    &.494    &.427\tabularnewline
.75    &.754    &.733    &.712    &.686    &.648    &.619    &.583    &.540    &.494    &.427\tabularnewline
.80    &.754    &.733    &.712    &.686    &.648    &.619    &.583    &.540    &.494    &.427\tabularnewline

 \hline\hline
  \tabularnewline
 \multicolumn{11}{c}{$p=1000$}\tabularnewline
 \hline\hline
 
 \multirow{2}{*}{$b$}   &  \multicolumn{10}{c}{$\beta$} \tabularnewline
\cline{2-11}
   &  .001  & .002   & .003   & .004   & .005   & .006   & .007   & .008   & .009   & .010 \tabularnewline
 \hline\hline

.05    &.809    &.794    &.770    &.752    &.719    &.676    &.631    &.594    &.534    &.473\tabularnewline
.10    &.790    &.774    &.751    &.730    &.698    &.657    &.611    &.575    &.518    &.460\tabularnewline
.15    &.790    &.774    &.751    &.730    &.698    &.657    &.611    &.575    &.518    &.460\tabularnewline
.20    &.790    &.774    &.751    &.730    &.698    &.657    &.611    &.575    &.518    &.460\tabularnewline
.25    &.790    &.774    &.751    &.730    &.698    &.657    &.611    &.575    &.518    &.460\tabularnewline
.30    &.790    &.774    &.751    &.730    &.698    &.657    &.611    &.575    &.518    &.460\tabularnewline
.35    &.790    &.774    &.751    &.730    &.698    &.657    &.611    &.575    &.518    &.460\tabularnewline
.40    &.790    &.774    &.751    &.732    &.699    &.657    &.612    &.577    &.518    &.463\tabularnewline
.45    &.792    &.779    &.757    &.738    &.707    &.664    &.620    &.581    &.529    &.469\tabularnewline
.50    &.802    &.790    &.773    &.756    &.725    &.688    &.648    &.606    &.554    &.496\tabularnewline
.55    &.838    &.820    &.809    &.789    &.777    &.748    &.709    &.661    &.605    &.543\tabularnewline
.60    &.879    &.874    &.863    &.843    &.812    &.789    &.764    &.707    &.650    &.583\tabularnewline
.65    &.918    &.908    &.888    &.871    &.842    &.809    &.780    &.729    &.674    &.604\tabularnewline
.70    &.924    &.913    &.894    &.875    &.848    &.811    &.782    &.737    &.675    &.607\tabularnewline
.75    &.925    &.913    &.895    &.875    &.848    &.811    &.782    &.737    &.675    &.607\tabularnewline
.80    &.925    &.913    &.895    &.875    &.848    &.811    &.782    &.737    &.675    &.607\tabularnewline

\hline
\end{tabular}
\end{table}
}

{\small
\begin{table}[H]\label{tab: 7}
\caption{\small{Results of Monte Carlo experiments for rejection probability. Market structure model.}}

\hspace*{-1cm}
\begin{tabular}{ccccccccccc}
\hline\hline
 \tabularnewline
 
\multirow{2}{*}{$\frac{\Delta\theta}{0.25}$}&\multirow{2}{*}{CT/GH}   & \multirow{2}{*}{$n$} &  \multicolumn{8}{c}{test type} \tabularnewline
\cline{4-11}
& & & $SN_1$ & $SN_2$ & $MB_1$ & $MB_2$ & $MB_3$ & $EB_1$ & $EB_2$ & $EB_3$  \tabularnewline
\hline\hline
 
\multirow{6}{*}{$(0,0,0)$}
&\multirow{3}{*}{CT}

&1000    &.027    &.027    &.028    &.028    &.011    &.026    &.027    &.009\tabularnewline
&&2000    &.036    &.037    &.038    &.038    &.008    &.036    &.037    &.009\tabularnewline
&&5000    &.024    &.029    &.028    &.032    &.035    &.026    &.034    &.031\tabularnewline
\cline{2-11}

&\multirow{3}{*}{GH} 
&1000    &.021    &.021    &.022    &.021    &.004    &.022    &.021    &.003\tabularnewline
&&2000    &.006    &.006    &.011    &.011    &.000    &.011    &.011    &.000\tabularnewline
&&5000    &.001    &.005    &.010    &.013    &.013    &.010    &.013    &.013\tabularnewline
\cline{1-11}

\multirow{6}{*}{$(1,0,0)$}

&\multirow{3}{*}{CT}
&1000    &.168    &.141    &.154    &.156    &.027    &.135    &.136    &.028\tabularnewline
&&2000    &.183    &.186    &.189    &.209    &.213    &.188    &.208    &.208\tabularnewline
&&5000    &.249    &.302    &.271    &.307    &.304    &.279    &.306    &.307\tabularnewline
\cline{2-11}

&\multirow{3}{*}{GH}
&1000    &.086    &.086    &.124    &.124    &.069    &.121    &.120    &.071\tabularnewline
&&2000    &.169    &.169    &.242    &.254    &.253    &.231    &.253    &.257\tabularnewline
&&5000    &.414    &.480    &.579    &.630    &.630    &.569    &.619    &.634\tabularnewline
\cline{1-11}

\multirow{6}{*}{$(-1,0,0)$}
&\multirow{3}{*}{CT}

&1000    &.066    &.066    &.072    &.073    &.074    &.070    &.070    &.075\tabularnewline
&&2000    &.164    &.179    &.176    &.203    &.193    &.172    &.194    &.192\tabularnewline
&&5000    &.611    &.684    &.628    &.704    &.700    &.618    &.702    &.704\tabularnewline
\cline{2-11}

&\multirow{3}{*}{GH}
 
&1000    &.079    &.079    &.160    &.159    &.158    &.155    &.153    &.158\tabularnewline
&&2000    &.327    &.350    &.520    &.546    &.559    &.527    &.548    &.549\tabularnewline
&&5000    &.953    &.972    &.984    &.994    &.994    &.989    &.994    &.995\tabularnewline
\cline{1-11}

\multirow{6}{*}{$(0,1,0)$}
&\multirow{3}{*}{CT}
&1000    &.205    &.205    &.203    &.202    &.203    &.204    &.205    &.204\tabularnewline
&&2000    &.289    &.302    &.296    &.300    &.304    &.298    &.302    &.304\tabularnewline
&&5000    &.547    &.554    &.531    &.566    &.573    &.520    &.570    &.574\tabularnewline
\cline{2-11}

&\multirow{3}{*}{GH}
&1000    &.097    &.091    &.180    &.174    &.166    &.176    &.171    &.156\tabularnewline
&&2000    &.145    &.145    &.248    &.247    &.248    &.246    &.251    &.245\tabularnewline
&&5000    &.330    &.358    &.484    &.525    &.515    &.484    &.524    &.522\tabularnewline
\cline{1-11}

\multirow{6}{*}{$(0,-1,0)$}
&\multirow{3}{*}{CT}

&1000    &.031    &.033    &.041    &.042    &.042    &.043    &.041    &.036\tabularnewline
&&2000    &.064    &.075    &.068    &.078    &.076    &.067    &.075    &.078\tabularnewline
&&5000    &.323    &.439    &.336    &.479    &.470    &.337    &.460    &.466\tabularnewline
\cline{2-11}

&\multirow{3}{*}{GH}
 
&1000    &.006    &.006    &.014    &.014    &.011    &.015    &.015    &.009\tabularnewline
&&2000    &.013    &.015    &.040    &.040    &.039    &.040    &.048    &.041\tabularnewline
&&5000    &.113    &.179    &.256    &.343    &.358    &.252    &.351    &.357\tabularnewline
\cline{1-11}

\multirow{6}{*}{$(0,0,1)$}
&\multirow{3}{*}{CT}

&1000    &.212    &.212    &.211    &.212    &.069    &.211    &.213    &.070\tabularnewline
&&2000    &.377    &.377    &.356    &.363    &.363    &.341    &.357    &.365\tabularnewline
&&5000    &.700    &.762    &.720    &.764    &.764    &.719    &.766    &.768\tabularnewline
\cline{2-11}

&\multirow{3}{*}{GH}
 
&1000    &.116    &.116    &.205    &.205    &.080    &.202    &.203    &.082\tabularnewline
&&2000    &.201    &.201    &.289    &.292    &.268    &.287    &.291    &.262\tabularnewline
&&5000    &.496    &.549    &.657    &.698    &.703    &.654    &.702    &.704\tabularnewline
\cline{1-11}

\multirow{6}{*}{$(0,0,-1)$}
&\multirow{3}{*}{CT}
&1000    &.032    &.032    &.033    &.033    &.035    &.034    &.035    &.033\tabularnewline
&&2000    &.069    &.085    &.077    &.092    &.074    &.084    &.094    &.078\tabularnewline
&&5000    &.239    &.358    &.278    &.390    &.394    &.267    &.390    &.383\tabularnewline
\cline{2-11}

&\multirow{3}{*}{GH}

&1000    &.010    &.010    &.023    &.022    &.022    &.023    &.023    &.022\tabularnewline
&&2000    &.019    &.024    &.054    &.058    &.056    &.058    &.062    &.057\tabularnewline
&&5000    &.081    &.139    &.209    &.307    &.305    &.207    &.298    &.300\tabularnewline
\hline

\end{tabular}
\end{table}
}

\end{document}